\documentclass[a4paper,12pt]{report}
\usepackage[utf8x]{inputenc}
\usepackage{amsmath}
\usepackage{amssymb}
\usepackage{psfrag}
\usepackage[all,cmtip]{xy}
\usepackage{amsmath,amscd, amsthm}
\usepackage{amssymb}
\usepackage{graphicx}
\usepackage{wasysym}
\usepackage{textcomp}
\usepackage{epic,eepic}
\usepackage{graphicx} 
\DeclareGraphicsExtensions{.pdf,.eps,.fig}
\newcommand{\cat}[1]{\ensuremath{\mathbf{#1}}}
\newcommand{\Cat}[1]{\ensuremath{\mathbf{#1}}}

\newcommand{\after}{\ensuremath{\circ}}
\newcommand{\relto}{\ensuremath{\xymatrix@1@C-2ex{\ar|(.4)@{|}[r] &}}}
\newcommand{\defined}{\ensuremath{\mathop{\downarrow}}}
\newcommand{\ext}{\ensuremath{^{\mathrm{ext}}}}
\newcommand{\func}{\ensuremath{^{\mathrm{func}}}}
\newcommand{\mfunc}{\ensuremath{^{\mathrm{mfunc}}}}
\newcommand{\name}[1]{\ensuremath{\ulcorner #1 \urcorner}}
\newcommand{\swapmap}{\ensuremath{\sigma}}
\newcommand{\Id}{\operatorname{Id}}

\newtheorem{definition}{Definition}[section]
\newtheorem{proposition}[definition]{Proposition}
\newtheorem{lemma}[definition]{Lemma}
\newtheorem{theorem}[definition]{Theorem}
\newtheorem{remark}[definition]{Remark}
\newtheorem{example}[definition]{Example}
\newtheorem{corollary}[definition]{Corollary}
\newtheorem{conjecture}[definition]{Conjecture}
\begin{document}
\title{  Relational symplectic groupoids and Poisson sigma models with boundary}
\author{Iv\'an Guillermo Contreras Palacios}
\maketitle
\section*{Acknowledgements}

First of all, I thank my PhD supervisor, Alberto Cattaneo; his guidance, support and patience during these years made this work possible and enjoyable, every discussion with him, concerning maths or any other topic,  gave me new insights and motivation to keep going. \\

I am  also grateful to those who gave me important remarks, comments and ideas; specially  Camilo Arias Abad and Pavel Mn\"ev helped me all this time with useful explanations on many aspects of mathematics and physics. I thank Iakovos Androulidakis, David Martinez, Florian Sch\"atz and Marco Zambon for their useful remarks and for taking the time to have a look at preliminary stages of this document.\\

I thank in particular Chris Heunen for introducing me to relative Frobenius and H*-algebras and for his crucial contribution for the results in Chapter \ref{Frobenius}.\\

I am also indebted to the commitee of referees: Henrique Bursztyn, Giovanni Felder, Marco Gualtieri and Alan Weinstein who provided me with quite useful and 
pertinent observations. \\

I should also mention my gratitude to all my friends at the Institute of Mathematics at University of Z\"urich  and ETH who made my stay in Switzerland quite pleasant. Only to mention a few, Anna-Lena, Urs, Davide, Michele, Emanuele(s), all the Thursdays soccer friends and the German course colleagues.\\

Last but not least I want to thank to those who have been close to me with their support and affection: my mother and my brother to whom I will be eternally indebted.\\

And above all, I thank my wife Eva, who supported me at every time with her love and delightful presence and to Sofia, my daughter,  the responsible of many hours of laughter  and happiness.

\newpage
$\,$

\newpage
\emph{To my ladies, Eva and Sofia}
\newpage

\section*{Synopsis}
Lie algebroids and their integration have been present as an important area of research in Lie theory, Poisson geometry and foliation theory, among others. The important developments made by Weinstein, Mackenzie-Xu, Cattaneo-Felder, Crainic-Fernandes,  Zhu-Tseng   among many others give a solid framework to study the links between the geometry and topology of Lie algebroids, Lie groupoids and several aspects of mathematical physics.\\

More precisely, the work of Cattaneo and Felder on Poisson sigma models (PSM) gives an explicit contruction of the symplectic groupoid associated to an integrable Poisson manifold in terms of the reduced phase space of the Poisson sigma model with certain boundary conditions. \\

The main purpose of this thesis is to make further developments on this construction, bringing on the same footing integrable and non integrable Poisson manifolds, through the construction of what we called \emph{relational symplectic groupoids}, which are, roughly speaking, a version before reduction of the phase space of PSM with boundary. \\

We describe such objects in categorical terms, as special objects in the \textit{extended symplectic category} for which the defining axioms are written in terms of special canonical relations (immersed Lagrangian submanifolds of weak symplectic manifolds). We prove natural properties of the construction, for instance, the connection between the symplectic structure on the relational groupoid and the Poisson structure on the base. We give finite dimensional examples of this construction and introduce the notion of equivalence of relational symplectic groupoids, which is helpful to connect our construction with the usual version of symplectic groupoids.\\

The main example comes from Poisson geometry, where the relational symplectic groupoid is infinite dimensional:  the cotangent bundle of the path space of a given Poisson manifold. In this case, we characterize the integrability conditions given by Crainic-Fernandes in terms of the relational symplectic groupoid and we also describe the different integrations of Poisson manifolds in this new perspective. \\

This motivational example suggested the study of the relational symplectic groupoid in more generality, which gives rise to the description of such objects (or their analogous) in different monoidal dagger categories. It turns out that this new way to integrate Poisson manifolds fits better with quantization and the study of PSM with branes.
\newpage

\section*{Zusammenfasung}

Lie  Algebroide und deren Integration kommen unter anderem in der Lie-Theorie, der Poisson-Geometrie und
der Folierung-Theorie vor.
Wichtige Erkenntnisse von Weinstein, Mackenzie-Xu, Cattaneo-Felder, Crainic-Fernandes,  Zhu-Tseng geben einen festen Rahmen um die Beziehung zwischen Geometrie und Topologie der Lie Algebroide, Lie Gruppoide und Aspekte der mathematischen Physik zu untersuchen.
Die Arbeit von Cattaneo und Felder betreffend Poisson
Sigma-Modelle (PSM) bietet eine Konstruktion der symplektischen Gruppoide assoziiert mit einer integrierbaren Poisson Mannigfaltigkeit. Diese Konstruktion bezieht sich den reduzierten Phasenraum des Poisson Sigma-Modells mit
bestimmten Randbedingungen.\\

Der Hauptzweck dieser Arbeit ist die Weiterentwicklung dieser
Konstruktion. Integrierbare und nicht integrierbare Poisson-
Mannigfaltigkeiten k\"onnen durch die Konstruktion von \emph{relationale symplektische Gruppoide} gleichgestellt werden.  Diese sind eine Version des Phasenraums PSM mit Rand vor der Reduktion.\\

Wir beschreiben ein solches Objekt in kategorischen Begriffen, als ein besonderes Objekt in der \emph{erweiterten symplektische Kategorie}, f\"ur
welches die definierenden Axiome in Bezug auf die besonderen kanonischen Relationen geschrieben werden.
Wir zeigen nat\"urliche Eigenschaften dieser Konstruktion, wie zum Beispiel, die
 Verbindung zwischen der symplektischen Struktur auf dem relationalen
 Gruppoid und der Poisson-Struktur auf der Basis. Wir geben
 endlichdimensionale Beispiele dieser Konstruktion und f\"uhren den Begriff
der
 \"Aquivalenz von relationalen symplektischen Gruppoiden ein, die hilfreich 
 ist um unsere Konstruktion mit der \"ublichen Version von symplektischen
 Gruppoiden zu verbinden.\\
 
 Das wichtigste Beispiel geht aus der Poisson-Geometrie hervor, wobei 
 das relationale symplektische Gruppoid unendlichdimensional ist: die
 Kotangentialb\"undel des Pfades innerhalb einer bestimmten Poisson
 Mannigfaltigkeit. In diesem Fall charakterisieren wir
 Integrierbarkeitsbedingungen von Crainic-Fernandes 
 anhand des relationalen symplektischen Groupoids und wir beschreiben die
 verschiedenen Integrationen von Poisson Mannigfaltigkeiten aus dieser
 neuen Perspektive.\\
 
 Dieses Beispiel motiviert zur Untersuchung des
 relationalen symplektischen Gruppoids in gr\"osserer Allgemeinheit, und gibt Anlass
 zur Beschreibung eines solchen Objektes in
 verschiedenen monoidalen Dolch Kategorien. 
Diese neue Art der Integration von Poisson Mannigfaltigkeiten ist bestens geeignet um PSM mit Branen sowie Quantifizierung zu studieren.
\newpage
\tableofcontents
\chapter{Introduction}
\section{Conventions and notations}
Most of the main definitions in the thesis contain an explanation on the conventions and notations. However, we include a list of general remarks on the conventions that can be useful for reading:
\begin{itemize}
\item The symbol $\nrightarrow$ denotes a relation and $\to$ denotes a map.
\item The special object $pt$ denotes the connected zero dimensional (symplectic) manifold.
\item The symbol $\langle, \rangle$ denotes the natural pairing between $TM$ and $T^*M$.
\item For a vector bundle $E\to M$,  $\Gamma^k(E)$ denotes the space of $k-$ differentiable sections of $E$. If $k$ does not appear, it is assumed that we consider smooth sections.
\item $\mathfrak X (M)$ denotes the space of smooth vector fields on a smooth manifold $M$.
\item $\triangle(M)$ symbolizes the diagonal subspace (or submanifold) of $M \times M$.
\item $\Id$ refers to the identity map.
\item $\dagger$ denotes in general the adjoint, either for objects or morphisms.
\end{itemize}
\newpage
\section{Motivation}

Symplectic groupoids have been studied in detail since their introduction independently by Coste, Dazord and Weinstein \cite{Cost}, Karas\"ev \cite {Karasev} and Zakrzewski \cite{Zakrzewski},  
and they appear naturally in Poisson and symplectic geometry, foliation theory and, as we will see later,  in some instances of the study of topological field theories (TFT).\\
 
In terms of category theory, a groupoid is a small category with invertible morphisms and, depending on the case, 
one can be interested in studying groupoids over sets, topological spaces, manifolds, etc. 
For example, in the symplectic world one would like to impose that the spaces and structure maps defining the groupoid are smooth maps and that there exists a symplectic structure on the space of morphisms compatible with the groupoid multiplication; 
more precisely, if $G \rightrightarrows M$ denotes a groupoid over $M$, there exists a symplectic form $\omega \in \Omega^2(G)$ such that the graph of the multiplication is a Lagrangian submanifold of $(G,\omega)\times(G,\omega)\times (G,-\omega)$. \\

Concerning the link between symplectic groupoids and field theories,  it was proven \cite{Cat} that the reduced phase space of a 2-dimensional TFT, called the Poisson sigma model (PSM), under certain boundary conditions and assuming that such reduced space is a smooth manifold, has the structure of a symplectic groupoid and it integrates the cotangent bundle of a given Poisson manifold $M$, regarded as a Lie algebroid over $M$.\\

The integration of Lie algebroids appears
as a generalization of what is called in the literature the \textbf{Lie third theorem} that can be stated as follows (see e.g. \cite{Duistermaat}):
\begin{itemize}
\item Let $\mathfrak g$ be a finite dimensional Lie algebra. Then there exists a unique simply connected Lie group $G$ with Lie algebra isomorphic 
to $\mathfrak g$.
\end{itemize}
A question that would expect the same type of answer can be stated in the case of Lie algebroids and Lie groupoids:
\begin{itemize}
\item Is there a Lie groupoid $G\rightrightarrows M$ such that its infinitesimal version corresponds to a given Lie algebroid $A$?
\end{itemize}
In particular,  the  case when the Lie algebroid is $T^*M$, where $M$ is  a Poisson manifold has particular interest since the construction of a smooth groupoid with a compatible symplectic structure gives, between other things, a symplectic realization of the given Poisson manifold. 
More precisely, we deal with the following question:
\begin{itemize}
\item Given a Poisson manifold $(M, \Pi)$, is there a symplectic groupoid $G\rightrightarrows M$ such that its space of objects is precisely $M$, the (target) source is an (anti) Poisson map and $M$ is a Lagrangian submanifold of $G$?
\end{itemize}

In general, a Poisson manifold can fail to be integrable and this means that such groupoid might be non smooth. In fact, the obstruction to integrability of Poisson manifolds has topological nature, depending on the second homotopy group of the leaves of the symplectic foliation of the Poisson manifold (see the work of Crainic and Fernandes \cite{Crai})
and it is given by the uniformly discreteness  of  what  are called \textit{monodromy groups} 
of the infinitesimal foliation whose space of leaves determines the groupoid $G$.\\ For example, the linear Poisson structure of $\mathfrak{su}(2)$ can be deformed in such way that the monodromy group of  certain leaves of the foliation is not uniform discrete, violating the afore mentioned integrability conditions.\\

For integrable cases, there exists a unique source fiber simply connected groupoid integrating $T^*M$ and one can wonder if such groupoid can be explicitly constructed. At this point, the link to topological field theories makes its appearance. In the work of Cattaneo and Felder \cite{Cat}, the symplectic groupoid for integrable Poisson manifolds 
is constructed as the space of reduced boundary fields of PSM, with the source space homeomorphic to a disc and with vanishing boundary conditions.\\


The properties of $G\rightrightarrows M$ are of special interest in Poisson geometry, since it is possible to equip $G$ with a symplectic structure $\omega$ compatible 
with the multiplication map in such a way that $G$ is a symplectic realization for $(M,\Pi)$ (i.e. the source map is Poisson, the target is anti-Poisson).\\


The main purpose of this thesis is the study of the geometric, algebraic and categorical features of the non reduced phase space of the  PSM, also in the case of more general boundary conditions. This leads in particular to a generalization of what is known on integration of Poisson manifolds, extending the construction of the symplectic groupoid associated to a Poisson manifold $M$ to the infinite dimensional setting,  through the construction of what we call \emph{relational symplectic groupoid}, that is, roughly speaking, a groupoid in the \emph{extended symplectic category} defined ``up to equivalence'', where the structure maps of the usual groupoids are replaced by special morphisms in the $\mbox{\textbf {Symp}}^{Ext}$, which contains as objects 
(possibly weak) symplectic manifolds and a special type of immersed canonical relations, as morphisms; namely,
the multiplication, inverse and unit of usual symplectic groupoids are replaced by infinite dimensional immersed Lagrangian submanifolds, obeying certain compatibility axioms.\\

It is important to remark here that the extended symplectic category is not formally speaking a category! \footnote{In the literature this is sometimes called a \emph{categroid}, (e.g. \cite{Kashaev})}, since the composition of smooth canonical relations is not smooth in general, and such composition is not continuous with respect to the topology for the space of morphisms.
However, for our construction, the defining morphisms will be well defined and composition of them make sense in $\mbox{\textbf {Symp}}^{Ext}$. \\

A particular type of relational symplectic groupoid , called \textit{regular},  admits the notion of smooth space of objects, denoted by $M$, that is constructed as a quotient using the defining data of $\mathcal G$, and also has source and target relations from $\mathcal G$ to $M$. We prove that $M$ can be equipped with a unique Poisson structure compatible with the symplectic structure on $M$ as a generalization of the unique Poisson structure that can be constructed in the space of objects of an usual symplectic groupoid \cite{Cost}.\\

It can be checked, for instance,  that usual symplectic groupoids are natural examples of finite dimensional relational symplectic groupoids. Fortunately,  this is not the only source for finite dimensional examples. Given a symplectic manifold $M$ and an immersed Lagrangian submanifold $L$, we can construct a relational symplectic groupoid out of it in a natural way. In addition, powers of symplectic groupoids appear also naturally as finite dimensional examples of this construction.\\

 However, the motivational example of relational symplectic groupoids is infinite dimensional and is associated, as we mentioned before, to the phase space of PSM with boundary.\\
In order to understand the reasons why this type of object requires special attention, some further motivation has to be done, since the challenges to fully understand the connection between integration of Poisson manifolds and PSM appear in an early stage.\\

In a more recent perspective (see \cite{AlbertoPavel1, AlbertoPavel2}), the study of the phase space before reduction of the PSM plays a crucial role, in the context of Lagrangian field theories with boundary. In general, the space of reduced boundary fields does not have to be a smooth finite dimensional manifold and in fact, the failure of being smooth is controlled by the integrability conditions of Poisson manifolds.\\

The construction of relational symplectic groupoids allows us  to deal with with non integrable Poisson structures, for which the reduced phase space is singular, on an equal footing as 
the integrable ones. 
This new approach differs from the stacky perspective of Zhu, Tseng (see \cite{Chenchang}) or more recently Ca\~nes \cite{Canes}; however, it is interesting to see the connections between these points of view. Relational symplectic groupoids seem to be better adapted to symplectic geometry and to quantization (see Section \ref{Perspectives}). \\

The important feature is that  relational symplectic groupoids can always be defined independently of the integrability of the Poisson manifold  and that it reduces, in the integrable cases, to the usual version of symplectic groupoids.\\ 
With the previous construction, there is a compatibility between the Poisson structure in the target and the symplectic structure in the relational symplectic groupoid
(see Theorem \ref{Theo1Poisson1}) and the integrability conditions translate into embeddability conditions for certain canonical relations in the new construction (see Theorem \ref{Embeddability}).\\


In particular, the axioms of the relational symplectic groupoid allow to extend the construction to special categories different from $\mbox{\textbf{Symp}}^{Ext}$, namely, the category $\mbox{\textbf{Hilb}}^{Ext}$ of Hilbert spaces as objects and linear subspaces as morphisms. In this setting , the quantization of relational symplectic groupoids would correspond \footnote{Assuming that the quantization is functorial.} to a functor from monoids in $\mbox{\textbf{Symp}}^{Ext}$ to monoids in $\mbox{\textbf{Hilb}}^{Ext}$.\\

The existence of this construction for general Poisson manifolds is guaranteed by proving that the morphisms defining the relational symplectic groupoid are \emph{immersed} Lagrangian submanifolds of certain Banach manifolds, that are in this setting particular types of spaces of maps, where the use of adapted versions of Frobenius and the inverse function theorem is allowed \cite{Lang, Lang1}. This restricts us to dealing with the regularity type of the space of boundary fields as well as the equations defining the dynamics of PSM. However, the generalization in the case where the space of boundary fields is of type $\mathcal{C}^{\infty}$ is discussed in Chapter \ref{PSMmain} and this leads us to the extension of relational symplectic groupoids in the framework of Frech\'et manifolds, where some of the main features of the construction are expected to hold. The full development of the construction in the Fr\'echet category is still work in progress.\\

Some extensions of this construction in different dagger categories lead the work developped in \cite{FrobeniusChris}, where the connection between groupoids and Frobenius algebras is made precise. Namely, there is a way to understand groupoids in the category \textbf{Set} as what we called \emph{Relative Frobenius algebras}, a special type of dagger Frobenius 
algebra \cite{Zakrzewski} in the category \textbf{Rel}, where the objects are sets and the morphisms are relations.\\
In addition, there exists an adjunction between a special type of semigroupoids 
(a more relaxed version of groupoids where the identities or inverses do not necessarily exist) and H*-algebras, a structure similar to Frobenius algebras 
but without unitality conditions and a more relaxed Frobenius relation.\\

This suggested the study of the defining axioms of the relational symplectic groupoid in different categories, namely, dagger monoidal categories, where the notion of adjoint morphism makes sense. In this direction we define weak monoids, weak *-monoids and cyclic weak *-monoids, which are more relaxed versions of relational symplectic groupoids, but they appear in some other scenarios, for instance, deformation quantization.\\

Following this construction there are many open questions and developments to be made. For instance, the non reduced version of the phase space 
for PSM with branes deserves special attention (see, e.g. \cite{Coiso,Branes}) in an effort to connect PSM with special boundary conditions and what is known in the literature as dual pairs 
\cite{Dual}. The presence of handles in relational symplectic groupoids would conjecturally give more geometric information on PSM with genera and the non regular version of the 
construction could give the correct framework to work with Poisson sigma models with singular target space. The formal link between the relational symplectic groupoids and symplectic microgeometry \cite{Benoit, Benoit2,Benoit3} is still work in progress.



\newpage
\section{Summary of the main results}

This section summarizes the main results of this thesis. 
After defining in general the notion of relational symplectic groupoids, we are able to relate this construction in the regular case  with the usual notion of topological and symplectic groupoids, in particular, the symplectic structure in the latter is given via symplectic reduction. More precisely, we define the relational symplectic groupoid $\mathcal G$ in terms of certain immersed Lagrangian submanifolds $L_1, L_2, L_3$ of $\mathcal G, \, \mathcal G \times \overline {\mathcal G}$  and $\mathcal G \times \mathcal G \times \overline{\mathcal G}$ respectively, satisfying certain compatibility relations. In particular, the immersed canonical relation $L_2$ induces 
an equivalence relation on a coisotropic submanifold $C$ of $\mathcal G$. (see Equation \ref{C} on page \pageref{C} ) 
We prove the following

\begin{enumerate}

\item \textbf{Theorem \ref{TheoSym}}
\emph{ Let $(\mathcal G,L,I)$ be a regular relational symplectic groupoid. 
Then $G:= C/L_2\rightrightarrows M$ is a topological groupoid over $M$. Moreover, if $G$ is smooth manifold, then $G\rightrightarrows M$ is a symplectic groupoid over $M:= L_1/L_2$.
}\\

In the regular case, we also prove the connection between the symplectic structure of the relational symplectic groupoid structure and the Poisson structure on the space of objects. Namely, 
it corresponds with the analogue existence/uniqueness of a Poisson structure on the space of objects of usual symplectic groupoids, where the (target) source of the groupoid is an (anti) Poisson map. The compatibility in our  construction  is described in terms of Dirac maps between presymplectic and Poisson manifolds.\\
\item \textbf{Theorem \ref{Theo1Poisson1}}
\emph{ Let $(\mathcal G, L, I)$ be a regular relational symplectic groupoid, with $M= L_1/ L_2$.  Then, assuming that the $s$-fibers are connected, there exists a unique Poisson structure $\Pi$ on $M$ such that the map $s\colon C\to M$ (or $t$) is a forward-Dirac map.
}\\

Conjecturally, the connectedness condition of the $s$-fiber can be dropped (see Conjecture \ref{Theo1PoissonC}).


The following theorem leads to the discussion of relational symplectic groupoids coming from the PSM, namely, relational symplectic groupoids correspond to a generalization of the construction presented by Cattaneo and Felder in \cite{Cat}, as now we deal with the case of singular reduced phase spaces. In fact,  this theorem holds also for singular Poisson manifolds, for instance, in Example \ref{torus}. Studying the phase space of the Poisson sigma model before reduction we prove the following
 \item \textbf{Theorem \ref{Relational}}. 
 \emph{
Given an arbitrary Poisson manifold $(M, \Pi)$, there exists a relational symplectic groupoid $(\mathcal G, L,I)$ integrating it.}\\

By integration in this theorem we mean the following 
\begin{enumerate}
\item The existence of an infinite dimensional relational symplectic groupoid that is regular, such that $L_1/ L_2=M$, regarded as a Poisson manifold.
\item  In the case where the Lie algebroid $T^*M$ is integrable in the usual sense, the reduction $\underline C$ of $(\mathcal G, L,I)$ corresponds to the usual symplectic groupoid $G \rightrightarrows M$.
\end{enumerate} 
This infinite dimensional integration for Poisson manifolds extends the usual integration through the phase space of PSM, since it can be defined independently whether the reduced phase space of the PSM is smooth or not.


More precisely, the following theorem states the connection between the symplectic groupoid for an integrable Poisson manifold and the associated relational symplectic groupoid.

  \item \textbf{Theorem \ref{reduction}} \emph{Let $(M,\Pi)$ be an integrable Poisson manifold. Let $\mathcal{G}$ be the relational symplectic groupoid associated to the cotangent bundle of the path space $T^*PM$ and let $G= \underline{C_{\Pi}}$ be the symplectic Lie groupoid associated to the characteristic foliation on $C_{\Pi}$. Then $\mathcal{G}$ and $G$ are equivalent as relational groupoids.}

 
 The following theorem states the connection between the construction of the relational symplectic groupoid for a Poisson manifold $(M, \Pi)$ and the integrability obstructions of the Lie algebroid $T^*M$. More concretely, if $(M, \Pi)$ is integrable, the uniform discreteness of the monodromy groups described in \cite{Crai} implies that there exists a tubular neighborhood of the zero section of $T^*PM$, denoted by $N(\Gamma_0(T^*PM))$, such that, for 
$$\overline{L_1}:= \bigcup_{x_0 \in M} T^*_{(\overline{X,\eta})}PM \cap L_1,$$
where $(\overline{X,\eta})= \{(X,\eta) \vert X \equiv X_0, \eta \in \ker \Pi^{\#}\}$, the following theorem holds
 
  \item \textbf{Theorem \ref{Embeddability}}
 \emph{If $T^*M$ is integrable, then $\overline{L_1}\cap N(\Gamma_0(T^*PM))$ is an embeded submanifold of $T^*(PM)$}\\

In this point of view, $L_1$ is regarded as the space of $T^*M-$ 
paths that are $T^*M-$ homotopy equivalent to the trivial $T^*M-$ paths.


 In addition, the equivalence of relational symplectic groupoids allows to compare different integrations for Poisson manifolds. In this new perspective, different groupoids integrating a given Poisson manifold are equivalent.

 \item \textbf{Proposition \ref{Integrations}} 
 \emph{Let $G\rightrightarrows M$ and $G^{'}\rightrightarrows M$ be two s-fiber connected symplectic groupoid integrating the same Poisson manifold $(M, \Pi)$. Then $(G,L,I)$ and $(G^{'}, L^{'}, I^{'})$ are equivalent as relational symplectic groupoids.}
 The proof of this theorem is based on the fact that any $s$- fiber connected integration of a given Lie algebroid $A$ comes from a quotient with respect to the action of a discrete group on the s-fiber simply connected Lie groupoid that integrates $A$.\\

Trying to understand groupoid structures in particular types of categories, it is possible to link groupoids and what are called 
\emph{relative Frobenius algebras}. In \cite{FrobeniusChris} the interaction between groupoids in the category \textbf{Set} and relative 
Frobenius algebras in the category \textbf{Rel} is studied, as well as the version ``before reduction'', namely, an adjunction between what is called \emph{locally cancellative regular semigroupoids} in \textbf{Set} 
and relative H*-algebras in \textbf{Rel}.
 \item \textbf{Theorem \ref{Chris1}}
 \emph{There exists an equivalence of categories between the category $\mbox{\textbf{Frob(Rel)}}^{Rel}$ of 
relative Frobenius algebras with suitable morphisms and $\mbox{\textbf{Gpd}}^{Rel}$ of groupoids over sets whose morphisms are subgroupoids of the cartesian product}.

\item \textbf{Theorem \ref{Chris2}}
\emph{There exists an adjunction of categories between the category \newline $\mbox{\textbf{(H*-alg)(Rel)}}^{Rel}$of relative
 $H$*- algebras with suitable morphims and the category $\mbox{\textbf{LRSgpd}}$ of locally cancelative semigroupoids. }
\\


\end{enumerate}
\newpage
\section {Structure of the thesis}
Each one of the chapters of the thesis has a brief introduction at the beginning and they are organized as follows:\\

Chapter 2 is an introduction to symplectic linear algebra, symplectic groupoids and to the problem of integration of Lie algebroids specialized to the case of Poisson manifolds. 
The notation and main concepts on symplectic  and Poisson geometry used throughout the 
thesis will be introduced here. \\

Chapter 3 deals with the main features of PSM with boundary, in particular, their space of 
boundary fields. Here, the main construction, the \emph{relational symplectic groupoid} is introduced; its axioms are stated and we 
give the main examples as well. \\

Chapter 4 is devoted to discuss algebraic and categorical aspects of the construction, namely, the definition of \textit{monoid} and 
\textit{groupoid} like structures in more general categories and examples of them. In the same direction we describe the connection between groupoids and Frobenius 
algebras, encoded in Theorems 
\ref{Chris1} and \ref{Chris2}.\\ 

At the end of the Chapter we expose possible generalizations and the perspectives of the construction of relational 
symplectic groupoids.
\newpage

\chapter{Symplectic groupoids}
The idea of the first part of this chapter is to discuss general facts and results on symplectic linear algebra which will be used throughout the thesis.
The first section intends to unify the treatment of both finite and infinite dimensional symplectic vector spaces, 
therefore, in the sequel, the vector spaces are possibly infinite dimensional. We discuss briefly about Poisson algebras and Poisson manifolds and 
then we state the main facts about integration of Poisson manifolds.
\section{Symplectic linear algebra } 
\subsection{Definitions}
\begin{definition}
\emph{Let $V$ be vector space over $\mathbb R$. A skew symmetric form $\omega \in \Lambda^2(V^*)$ is called \textit{non degenerate} if 
the induced linear map
\begin{eqnarray*}
\omega^{\#}\colon V &\to& V^*\\
\omega^{\#}(v)(w)&:=& \omega(v,w)
\end{eqnarray*}
is an isomorphism}
\end{definition}
\begin{definition}\emph{$\omega$ is called \emph{weakly non degenerate} if $\omega^{\#}$ is injective.
}
 \end{definition}

\begin{remark} \emph{If $V$ is finite dimensional, a \emph{weakly non degenerate} form is also non degenerate.}
\end{remark}

\begin{definition}
\emph{A bilinear symmetric form  $\omega$ on $V$ is called \textit{weak symplectic} if is skew symmetric and weakly non degenerate.}
\end{definition}
\begin{definition}\label{weak}\emph{ A vector space $V$ equipped with a weak symplectic form $\omega$ is called a \emph{weak symplectic vector space}.}
\end{definition}
\begin{definition}\label{symplectomorphism}\emph{ A bijective linear map $f: (V,\omega_V)\to (W, \omega_W)$ is called a \emph{symplectomorphism} if $f^*\omega_W=\omega_V$.
}
\end{definition}

\begin{definition}
\emph{Let $V$ be a symplectic space and $W$ a linear subspace of $V$. We define its \textit{symplectic orthogonal space}, by
\[W^{\perp}:= \{v \in V \vert\,  \omega(v,w)=0, \forall w \in W \}.\]} 

\end{definition}
Some properties of the symplectic orthogonal spaces that will be used from now on, are the following
\begin{proposition} \label{properties}\emph{ 
Let $V$ be a weak symplectic space and $W$, $Z$ subspaces of $V$. Then:
\begin{enumerate}
 \item $W \subset Z \Longrightarrow Z^{\perp} \subset W^{\perp}$.
\item $(W+Z)^{\perp} = W^{\perp} \cap Z^{\perp}$.
\item $W^{\perp}+ Z^{\perp} \subset (W \cap Z)^{\perp}$.
\item $W \subset W^{\perp \perp}$.
\item $W^{\perp}= W^{\perp \perp \perp}$.
\end{enumerate}
}
\end{proposition}

We can define $\omega^{\# W}\colon V \to W^{*}$ as the restriction of $\omega^{\#}(V)$ to $W$, namely,
\[\omega^{\#}(v)(w):= \omega (v,w), \forall v \in V, w \in W.\]
In this way,
\[W^{\perp}= \ker \omega^{\# W}\] and therefore, we have the induced map $\omega^{\# W}\colon V / W^{\perp} \hookrightarrow W^*$.
\begin{remark}
\emph{In the finite dimensional setting
\[V / W^{\perp} \simeq W^*\]
and this implies that $\dim W^{\perp}= \dim V -\dim W $.}
\end{remark}
Denoting by $\omega \vert_ W$ the restriction of $\omega$ to $W$ we get that this induces a bilinear form on $W$ and 
\[(\omega\vert_W)^{\#}= (\omega^{\# W})\vert_W.\] Therefore
\[\ker(\omega \vert_W)^{\#}= W \cap W^{\perp}.\]
Now we define special subspaces of special interest for our purposes.
\begin{definition}
\emph{ A subspace $W$ of $V$ is called
\begin{enumerate}
 \item  \textit{symplectic} if $\omega\vert_W$ is a symplectic form or equivalently, 
$W \cap W^{\perp}= \{0\}.$
 \item \textit{isotropic} if $W \subset W^{\perp}$.
\item \textit{cosiotropic} if $W^{\perp} \subset W$.
\item \textit{Lagrangian} if $W^{\perp}=W$. 
\end{enumerate}}

\end{definition}
\subsection{Symplectic reduction}\label{properties}
Let $W$ be a linear subspace of a symplectic vector space $V$. We define
\[ \underline{W}:= W / W \cap W^{\perp},\]
called the \textit{reduction} of $V$. The form $\omega$ induces a symplectic form $\underline{\omega}$ on $\underline{W}$ given by
\[\underline{\omega}([w_1],[w_2]):= \omega(w_1, w_2).\] It can be easily proven that $\underline{\omega}$ is an antisymmetric, non degenerate form.
The following properties for symplectic reductions and the special subspaces of symplectic spaces hold
\begin{enumerate}
 \item In the finite dimensional case, if $W$ is a Lagrangian subspace, then $\dim W= \frac 1 2 \dim V$.
 \item If $W$ is coisotropic, 
\[\underline {W}= W / W^{\perp}.\]
\item $W$ is isotropic $\Longleftrightarrow \underline{W}= \{0\}.$ 
\item If $ Z \subset L$, where $L$ is Lagrangian, then $Z$ is isotropic.
\item If $ L \subset Z$, where $L$ is Lagrangian, then $Z$ is coisotropic.
\item If $L \subset Z$, where $L$ is Lagrangian and $Z$ is isotropic, then $Z=L$. That is why a Lagrangian space is also called 
\textit{maximally isotropic}.
\end{enumerate}
The following are some propositions that  will be useful later on.
\begin{proposition}\label{Isotropic}
\emph{Let $W$ and $L$ be a coisotropic and a Lagrangian subspace of $V$, respectively. Let $P_{W}\colon W \to \underline{W}$ be the quotient map and let 
\[L_W:= P_W (L\cap W)= L \cap W \diagup L\cap W^{\perp},\]
then $L_W$ is isotropic in $\underline{W}$.} 
\end{proposition}
\begin{proof}
We get that
\begin{eqnarray*}
P_{W}^{-1}(L_W^{\perp})&=& (P_W^{-1}(L_W))^{\perp}\cap W\\
                       &=&(L\cap W + W^{\perp})^{\perp}\cap W \\
		       &=& (L\cap W)^{\perp} \cap W^{\perp \perp}\cap W\\
		       &\stackrel{W \subset W^{\perp \perp}}{=}& (L\cap W)^{\perp} \cap W\\
		       &\supseteq& (L^{\perp})\\
		       &\stackrel{Coisotropicity \, of \, W}{\supseteq}& L^{\perp} \cap W + W^{\perp}.
\end{eqnarray*}

Applying $P_W$ in both sides we get 
\begin{eqnarray*}
L_W^{\perp} &\supseteq& P_W(L^{\perp}\cap W + W^{\perp})\\
            &=& P_{W}(L^{\perp} \cap W)\\
	    &\stackrel{Isotropicity \, of\, L}{\supseteq}& P_W(L\cap W)\\
	    &=& L_W,
\end{eqnarray*}
 as we wanted. 
\end{proof}
\begin{corollary}\label{Lag projection}
\emph{In the finite dimensional case, it follows from dimensional reasons that if $L$ is Lagrangian, then $L_W$ is Lagrangian.} 
\end{corollary}

\begin{proposition}\label{Coisotropic}
\emph{Let $(V,\omega)$ be a symplectic space. Let $C$ be a coisotropic subspace of V. Let $L$ be a subspace such that 
\[C^{\perp} \subset L\subset C \subset V. \]
Assume that $\underline{L}:= L / C^{\perp}$ is Lagrangian in $\underline{C}:= C / C^{\perp}.$ Then, $L$ is a Lagrangian subspace of $V$.}
\end{proposition}

\begin{proof}
 Let $P\colon C \to \underline{C}$ denote the projection map to the reduced space.
The idea is to proof that $L=L^{\perp}$ using the fact that $p^{-1}(\underline{L})=p^{-1}(\underline{L}^{\perp})$. First, we prove that 
$p^{-1}(\underline{L}^{\perp})=L^{\perp}$. Let $v \in p^{-1}(\underline{L}^{\perp})$. We have that, by definition,
\[\underline{\omega}(p(v),\underline{w})=0, \forall \underline{w} \in \underline{L} \]
and this implies that
\[ \omega(v,w)=0, \forall w\in V \vert \underline{w} \in \underline{L},\]
therefore,
\[\omega(v,w)=0, \forall w \in L.\]

Also by definition we have that 
\[p^{-1}(\underline{L})=L,\]
therefore, using the fact that $\underline{L}$ is Lagrangian, $\underline{L}^{\perp}=\underline{L}$, 
we obtain that $L=L^{\perp},$ as we wanted.
\end{proof}
\begin{definition}\emph{ Let $V$ be a symplectic vector space. Its \emph{conjugated} space, denoted by $\overline V$ is the same vector space $V$ but now equipped with the symplectic form $-\omega$.
}
\end{definition}

\begin{proposition} \label{Coiso-Lagrangian}
\emph{Let $L$ be a Lagrangian subspace of $V \oplus \overline{V}$, such that $L \subset C\oplus C$, for some $C \subset V$. Then $C$ is 
coisotropic in $V$.}
\end{proposition}

\begin{proof} Since \[L \subset C\oplus C,\] we obtain that 
\[(C\oplus C)^{\perp}\subset L^{\perp}=L \subset {C \oplus C} \subset V \oplus \overline{V},\]
therefore $C\oplus C$ is coisotropic. Since $(C\oplus C)^{\perp}=C^{\perp} \oplus C^{\perp}$, we conclude that $C^{\perp} \subset C$, 
as we wanted.
\end{proof}

\subsection{Canonical relations}
Lagrangian subspaces play an important role in symplectic geometry and this section is devoted to study some of their properties. 
A relation $R$ between two sets $M$ and $N$ is a subset of the cartesian product $M\times N$ and we will use the notation
$R \colon M \nrightarrow N$. If $S\colon N \nrightarrow P$ is another relation, its composition is given by
\[S\circ R:= \{(m,p)\in M \times P \vert \exists n \in N, \, (m,n)\in R,\, (n,p)\in S \}\colon M \nrightarrow P.\]
\begin{remark}
\emph{In the case where the sets $M$ and $N$ are vector spaces, a relation $R$ is called \emph{linear} if it corresponds to a 
linear subspace of $M \oplus N$.}
\end{remark}
\begin{definition}
\emph{Let $(M,\omega_M)$ and $(N,\omega_N)$ be symplectic spaces. A linear relation $L\colon M \nrightarrow N$ is called \emph{canonical}, 
if $R$ is a Lagrangian subspace of $\overline{M}\oplus N$, where $\overline{M}$ denotes the vector space $M$ with the negative 
symplectic form $-\omega_M$.} 
\end{definition}
\begin{remark}
\emph{If $L\colon V \nrightarrow W$ is a canonical relation, then $L$ is a Lagrangian subspace of $\overline W \oplus V$, then we have the 
canonical relation
\[L^{\dagger}\colon W \nrightarrow V\]
called the \emph{transpose} of $L$}. 
 
\end{remark}

\begin{remark}
\emph{A Lagrangian submanifold $L$ of $V$ is a canonical relation $L\colon \{0\} \nrightarrow V$ with transpose 
$L^{\dagger}\colon V \nrightarrow \{0\}$ .} 
\end{remark}
\begin{proposition}\label{Composition} \emph{In the finite dimensional case, the composition of canonical relations is a canonical relation.}
\end{proposition}
\begin{proof}
Let $L_1\colon U \nrightarrow V$ and $L_2\colon V \nrightarrow W $ be canonical relations.
Defining 
\[\triangle_V \colon= \{(v,v),\, v \in V\} \subset \overline V \oplus V ,\]
it is easy to check that $\triangle\colon V \nrightarrow V$ is a canonical relation. In addition, if we define
\[ D:= \overline U \oplus \triangle_V \oplus Z \subset \overline U \oplus V \oplus \overline V \oplus W,\]
since $D^{\perp}= \{0\} \oplus \triangle_V \oplus \{0\},$ then $D$ is a coisotropic subspace and $\underline D= \overline U \oplus W.$\\
Since $L_1\subset \overline U \oplus V $ and $L_2\subset \overline V \oplus W $ are Lagrangian subspaces,
\[L_1 \oplus L_2 \subset \overline U \oplus V \oplus \overline V \oplus W \] is Lagrangian and then, after projecting onto the quotient,
$P_D(L_1 \oplus L_2)$ is Lagrangian in $\underline D$.
Since \[L_1 \oplus L_2 \cap D =\{(u,v,v,w)\vert (u,v) \in L_1,\, (v,w) \in L_2\},\]
then 
\[P_D(L_1 \oplus L_2)=P_D(L_1 \oplus L_2 \cap D)= L_2 \circ L_1\]
as we wanted.
\end{proof}
\begin{remark} In the infinite dimensional case, the compositition of two canonical relations is not in general a canonical relation, as we will see in Example \ref{Compo}. It can be easily checked that the composition of canonical relations is in general isotropic.
\end{remark}
 
Now, following \cite{Alan}, consider a coisotropic subspace $W \subset V$. It follows that $W \oplus V$ is a coisotropic subspace of 
$\overline V \oplus V$ (since $(W \oplus V)^{\perp}= W^{\perp} \oplus \{0\}$). 
Since $\triangle_V \subset \overline V \oplus V$ is a Lagrangian subspace, by Corollary 
\ref{Lag projection}, $P_{W \oplus V}(\triangle_V)$ is a Lagrangian subspace of
 $\underline{\overline W \oplus V}= \overline{\underline{W}}\oplus V$, when $V$ is finite dimensional. This projection will be denoted by $I$ and is a canonically defined canonical relation $I\colon \underline W \nrightarrow V$. 
In fact, this also holds in the infinite dimensional setting due to the following
\begin{proposition} \label{projection}
\emph{For any (possibly infinite dimensional) symplectic space $V$, $I$ is a canonical relation.}
\end{proposition}
\begin{proof}
Explicitely we have that
\[I=\{([w],w) \in (W \diagup W^{\perp})\times V \vert w \in W \},\]
therefore
\[I^{\perp}=\{([v],z)\vert [v] \in W \diagup W^{\perp}, z\in V, \underline{\omega}([v],[w])-\omega(z,w)=0, \forall w \in W\}.\] 
By linearity  and the construction of the reduction this is equivalent to
\begin{eqnarray*}
I^{\perp}&=&\{([v],z)\vert [v] \in W \diagup W^{\perp}, z\in V, \omega(v-z,w)=0, \forall w \in W, v \in [v]\}\\
&=& \{([v],z)\vert [v] \in W \diagup W^{\perp}, \, z \in V, v-z \in W^{\perp}, \forall v \in [v]\}.
\end{eqnarray*}
This implies in particular that $z$ and $v$ belong to the same equivalence class and since $v \in C$ and $C$ is 
coisotropic it follows that $z \in C$ and therefore $[v]=[z]$, hence $I^{\perp}=I$, as we wanted.
\end{proof}

 We denote by $P:= I^{\dagger}: 
V \nrightarrow \underline{W}$, the transpose of $I$. We can prove then the following
\begin{proposition}
 \emph{The following relations hold 
\begin{enumerate}
 \item $P\circ I\colon \underline{W}\nrightarrow \underline W = \mbox{Graph(Id)}.$
 \item $I\circ P\colon V \nrightarrow V= \{(w,w^{'})\in V \times V \vert P_W(w)=P_W(w^{'})\}.$ Furthermore, $I\circ P \subset W \times W$ is
 an equivalence relation on $W$ and \[W \diagup I\circ P= \underline W.\]
\end{enumerate}
}
\end{proposition}
\begin{proof}
Direct computation. 
\end{proof}
As an application of this lemma we have 
\begin{definition}
 \emph{Let $\underline L\colon \underline W \nrightarrow \underline W$ be a canonical relation. By the proposition \ref{Coisotropic},
\[l(\underline L):= I \circ \underline L \circ P\colon V \nrightarrow V\] 
is a canonical relation and will be called the \emph{canonical lift} of $\underline L$.
}
\end{definition}
\begin{definition}\label{canonicalproj}
\emph{Let $L\colon V \nrightarrow V$ be a canonical relation. Then 
\[p(L):= P \circ L \circ I\colon \underline W \nrightarrow \underline W\]
is also a canonical relation called the \emph{canonical projection} of $L$}.
\end{definition}
These two particular relations have interesting properties. Observe first that the composition $p\circ l$ is the identity for Lagrangian subspaces of 
the symplectic reduction $\underline W$, however, the composition $l \circ p$ is not the identity. An example of the application of these construction is the following
\begin{example} (Canonical lift of symplectic flows). 
\emph{ We consider a symplectic space $V$, a vector field $X \in \mathfrak X(V)$ and a smooth map
\[\Phi\colon V \times \mathbb R \to V \]
such that
\begin{enumerate}
 \item $\frac{d}{dt}\Phi(x,t)=X(x)$ 
 \item $\Phi(x,0)=x$
 \item For every time $t$
\[\phi_t:=\Phi(\cdot, t): V \to V\]
is a symplectomorphism.
\end{enumerate}
Such $\phi_t$ is called a \emph{symplectic flow}.
Defining $\underline {L_t}:= \mbox{Graph }(\phi_t)$, they satisfy that
\begin{enumerate}
 \item $\underline{L_t}\circ \underline{L_s}= \underline{L_{t+s}}.$
\item $\underline{L_0}= \mbox{Graph (Id)}.$
\end{enumerate}
After the canonical lift for these relations we get that, being $L_t:= l(\underline{L_t})$, then
\begin{enumerate}
 \item 
\begin{eqnarray*}
L_t \circ L_s&=& I \circ \underline{L_t}\circ P \circ I \circ \underline {L_s}\circ P\\
&=& I \circ \underline{L_t}\circ \underline {L_s} \circ P\\
&=& I \circ \underline{L_{t+s}}\circ P\\
&=& L_{t+s}.
\end{eqnarray*}
\item \[L_0= l(\mbox{Id})= I \circ P \neq \mbox{Graph(Id)}.\]
\end{enumerate}
Therefore we conclude that the canonical lift of a symplectic flow is not a flow strictly speaking. 
However this lift is a symplectic flow \emph{up to an equivalence relation} and it corresponds to a good definition for symplectic flows in terms of canonical relations.
}

\end{example}
\section{Poisson algebras and Poisson manifolds}

Poisson algebras appear closely related to the study of symplectic spaces and they are relevant in the study of classical dynamics of physical systems. Poisson structures are a generalized version (admitting degeneracies) of the symplectic ones. Poisson manifolds, which are smooth manifolds equipped with such structure will appear very often in this work and this section is devoted to introduce the main concepts and to fix notation.

\begin{definition} \emph{A Poisson algebra is an associative algebra $P$ equipped with a bracket $\{\, ,\, \}$ satisfying the following properties
\begin{enumerate}
\item (Lie bracket). $\{\, ,\, \}$ is skew symmetric and it satisfies the Jacobi identity
$$\{\,f , \{\,g ,h\, \}\, \}+\{\,g , \{\,h ,f\, \}\, \}+ \{\,h , \{\,f ,g\, \}\, \}=0$$
\item (Leibniz property). $\{\,,\, \}$ acts as derivation on the product of $P$, i.e.
$$\{\,f,gh\}=\{f,g \}h + g\{f,h\}$$
\end{enumerate}}
\end{definition}
\begin{definition} \emph{A Poisson manifold is a pair $(M, \Pi)$, where $M$ is a smooth manifold and $\Pi \in \Gamma(TM \wedge TM)$ such that $(\mathcal C^{\infty}(M), \{,\}_{\Pi})$ is a Poisson algebra, where
$$\{f,g\}_{\Pi}:= \Pi(df,dg), \, \forall f,g \in \mathcal C^{\infty}(M).$$}
\end{definition}
In local coordinates, $\Pi$ is a bivector written in the form
$$\Pi(x)= \sum_{i<j} \Pi^{ij}(x)\frac{\partial}{\partial x_i}\wedge \frac{\partial}{\partial x_j}$$ 
the condition of $\Pi$ to be Poisson reads in coordinates as follows
\begin{equation}\label{SN}
\sum_{r}\Pi^{sr}(x)(\partial_r) \Pi^{lk}(x)+\Pi^{kr}(x)(\partial_r) \Pi^{sl}(x)+\Pi^{lr}(x)(\partial_r) \Pi^{ks}(x)=0,
\end{equation} that is equivalent to the vanishing condition for the Schouten-Nijenhuis bracket of $\Pi.$ 
Natural examples of Poisson manifolds are:
\begin{example}\emph{(Zero Poisson structure) Any manifold $M$ has a trivial Poisson structure, setting $\Pi(x)\equiv 0$.}
\end{example}
\begin{example} \emph{(Constant Poisson structure). Similar to the previous example, 
we consider an open subset of $\mathbb R^n$ and 
setting $$\Pi^{ij}(x)\equiv c_{ij},$$ for some real constants $c_{ij}$, we obtain a bivector satisfying Equation \ref{SN}.}\end{example}
\begin{example}\emph{(Linear Poisson structure). Let $\mathfrak g$ be a finite dimensional Lie algebra with structure constants $\{c_{ij}^k\}$ with respect to the basis $e_1, e_2, \cdots e_n$. Consider its dual space $\mathfrak g^*$; then, if $x_l$ denotes the linear functions on $g^*$ which corresponds by duality to $e_l$, the bivector $\Pi$ written down in coordinates as
$$\Pi(x)= \sum_{\stackrel{i<j}{k}} c_{ij}^k x_k\frac{\partial}{\partial x_i}\wedge \frac{\partial}{\partial x_j}$$ 
is Poisson.}
\end{example}
\begin{example} \emph{(Symplectic manifolds). If $(M, \omega)$ is a finite dimensional symplectic manifold, then with the bracket 
$$\{f,g\}:= \omega (X_f,X_g),$$
where $X_f$ and $X_g$ are the Hamiltonian vector fields associated to $f$ and $g$ respectively, 
$(\mathcal C^{\infty}(M), \{,\})$ is a Poisson algebra.}
\end{example}
\begin{definition} \emph{(Poisson morphisms). Let $(M, \{ , \}_M)$ and $(N, \{ , \}_N)$ be two Poisson manifolds. A map $\phi\colon M \to N$ is called a Poisson map or a Poisson morphism if the pull-back map $$\phi^*\colon \mathcal C ^{\infty}(N) \to \mathcal 
C ^{\infty}(M)$$ is a Lie algebra homomorphism with respect to the corresponding Poisson brackets.}\end{definition}
\begin{definition}\emph{(Coisotropic submanifolds). Let $(M, \Pi)$ be a Poisson manifold and $C$ be a submanifold of $M$. $C$ is called coisotropic if $$\Pi^{\#}(N^*C) \subset TC,$$ where $N^*C$ is the conormal bundle of $C$, defined by
$$N^*_xC:= \{\alpha \in T^*_xM \vert \langle \alpha, v \rangle=0, \forall v \in T_xC\}$$
where $\langle, \rangle$ denotes the natural pairing between $T^*_xM$ and $T_xM$ and
\begin{eqnarray}
\Pi^{\#}\colon T_xM &\to& T^*_xM\\
\alpha &\mapsto& \Pi(x)(\alpha, \cdot).
\end{eqnarray}}
\end{definition}
\section{The extended symplectic category}
Now the idea is to pass from the linear setting to the world of smooth manifolds, where some of the features previously discussed hold, some others fail.
From now on, a smooth manifold $M$ is called symplectic if is is equipped with a 2-form $\omega \in \Omega^2(M)$ that is closed and every tangent space
$(T_xM,\omega_x)$ is (weak) symplectic (Definition \ref{weak}).\\
We define the linear symplectic category, denoted by \textbf{LinSymp}, where the objects are symplectic spaces and the 
morphisms are symplectomorphisms (Definition \ref{symplectomorphism}). In order to allow more generality, one should want to add canonical relations as morphisms. However, as we have seen, some complications appear.
The failure to do this extensions are based on the following
\begin{enumerate}
 \item The composition of canonical relations is not canonical in general for the infinite dimensional case. 
\item The composition of canonical relations is not continuous in general (the set of Lagrangian subspaces is a 
homogeneous space and therefore it is naturally equipped with an induced topology).\\
\end{enumerate}
It is possible to define the category $\mbox{\textbf{FLinSymp}}^{Ext}$, where the objects are finite dimensional symplectic spaces and the 
morphisms are canonical relations. In the smooth case, we proceed in a similar way. We denote by \textbf{SympMan} the category which objects are symplectic 
manifolds and morphisms are symplectomorphisms. In a similar way  $\mbox{\textbf{SympMan}}^{Ext}$ would have (smooth) canonical relations as morphisms, but some problems appear here, even though we restrict ourselves to the finite dimensional setting. In general the (set theoretical)
composition of canonical relations is not a smooth manifold and therefore a well defined morphism. In order to guarantee smoothness of the composition, some transversality conditions should be imposed, namely
\begin{definition}
\emph{Two smooth relations $R\colon M\nrightarrow N$ and $S\colon N\nrightarrow P$ are called \emph{transversal} if the submanifolds $R \times S$ and $M\times \triangle_N\times P$ of $M\times N \times N \times P$ intersect transversally.}
\end{definition}

\begin{definition}
\emph{Two transversal relations $R\colon M\nrightarrow N$ and $S\colon N\nrightarrow P$ are \emph{strongly transversal} if the space \[
(R\times S)\cap (M\times \triangle_N \times P)\]
projects to an embedded submanifold of $M\times P$.} 
\end{definition}
It can be easily checked \cite{Alan} that
\begin{proposition} 
\emph{The composition of two smooth strongly transversal relations is a smooth relation}. 
\end{proposition}

Since Lagrangianity of smooth relations is a linear problem, it is automatic that the composition of smooth strongly transversal canonical relations is a smooth 
canonical relation.\\
Unfortunately, not every pair of composable smooth canonical relations are strongly transversal, therefore $\mbox{\textbf{SympMan}}^{Ext}$ is not technically a category, some people would call it 
a \emph{categroid} \cite{Kashaev}.
Another disadvantage is that the Lagrangianity condition does not behave well in the infinite dimensional framework, i.e. the composition of (linear) 
canonical relations is not in general a canonical relation as we can see in the following
\begin{example}\label{Compo}
\emph{Consider the space $V:= \mathcal C^{0}([-1/2, 1/2])\oplus \mathcal C^{0}([-1/2,1/2])$ and the bilinear form
$\omega(f\oplus g, k\oplus l)= \int_{-1/2}^{1/2} (fl-kg) dx.$ It can be checked that $\omega$ is a weak symplectic form on $V$ and that 
the space 
\[L:=\{(f\oplus g) \in V \vert \int_{-1/2}^{1/2}f dx =0,\, g \mbox{ constant} \}\]
is a Lagrangian subspace of $V$.Also, it can be proven that the space
\[C:=\{(f\oplus g) \in V \vert g(0)=0\}\] is both a coisotropic and a symplectic subspace of $V$(since $C^{\perp}=\{0\oplus 0\}).$ 
This implies that $\underline C=C$. It follows that the projection $L_C$ of $L$ in $C$ is 
given by 
$$L_C= L\cap C=\{(f\oplus g) \in V \vert \int_{-1/2}^{1/2}f dx =0,\, g\equiv 0  .$$
$L_C$ is isotropic since 
$$(L\cap C)\subset L=L^{\perp}\subset L^{\perp}+C^{\perp}=(L\cap C)^{\perp},$$
but it is not Lagrangian since the vector $(k\oplus 1)$, with $\int_{-1/2}^{1/2}k dx =0$ satisfies that
$$\omega(f\oplus g, k\oplus 1)=0, \forall (f\oplus g) \in L\cap C,$$
but  $(k\oplus 1)\not\in L\cap C.$
Thus, from the following diagram
\[
\xymatrixrowsep{2pc} \xymatrixcolsep{2pc} \xymatrix{\emptyset \ar@{->}[r]^{L} |-{/}\ar@{->}[dr]_{L_C}&V\ar[d]|-{-}^{P}\\ &C} 
\]
we get two canonical relations whose composition is not a canonical relation.}
\end{example}
\subsection{Possible alternatives}
In order to overcome these difficulties and to make precise sense of the extended symplectic category, different approaches can be taken.
One natural approach in the infinite dimensional linear case is to consider the category $\mbox{ \textbf{LinSymp}}^{Iso}$, where the objects are symplectic spaces and the morphisms are 
isotropic subspaces of the direct sum of two symplectic spaces. In this category the composition of morphisms is now well defined. However, the transversality condition in the smooth case cannot be skipped if we consider isotropic submanifolds instead of canonical relations. Here 
we present briefly in a disgression two possible solutions in the literature \cite{Wei}, \cite{Benoit, Benoit2, Benoit3}, \cite{Wehrheim}, that will not be used in the sequel . Then we present the model of the extended symplectic category that we will deal throughout this thesis, namely, defining only partial composition of morphisms and allowing also infinite dimensional symplectic manifolds.\\
\subsubsection{The Wehrheim-Woodward category}The construction given by Wehrheim and Woodward consists roughly on the minimal category that contains canonical relations as generators of morphisms.
More precisely, in their perspective, the objects of the symplectic category are symplectic manifolds and the morphisms are generated by composable canonical relations, i.e. they are subject to the relation that to morphisms $f$ and $g$ are composable in the category whenever 
the strong transversality condition is satisfied. First, they define sequences of composable canonical relations $(f_1,\, f_2 \cdots f_n)$ and they associate the empty sequence to every object.
Then they define the equivalence relation that is minimal under inclusion, that makes the sequence $(f,g)$ equivalent to $f\circ g$ when $f$ and $g$ are strongly transversal. Under this equivalence, the identity morphisms correspond to the diagonals of every object.  These morphisms are usually called \emph{generalized Lagrangian correspondences} \cite{Wehrheim, Wei}.
\subsubsection{Symplectic microgeometry} 
In this perspective, the motivation comes from the ``cotangent functor'' 
$T^*\colon \mbox{{\textbf{Man}}}\to \mbox{\textbf{SympMan}}^{Ext}$, a contravariant functor that associates to a smooth manifold $M$ 
its cotangent bundle $T^*M$ and to a map $\phi\colon X\to Y$ it associates the canonical relation denoted by $T^*\phi$ called the \emph{cotangent lift of $\phi$}:
\[T^*\phi:=\{((x,(T\phi)^*(\eta)), (\phi(x), \eta)) \vert x \in X \mbox{and }\eta \in T^*_{\phi(x)}Y \}.\]
A relevant fact to observe is that compositions of cotangent lifts satisfy the strong transversality condition, hence, we can understand the image of the 
category \textbf{Man} under $T$ as a subcategory of $\mbox{\textbf{Symp}}^{Ext}$. In order to generalize this lift for symplectic manifolds that are not 
cotangent bundles, a localization procedure is performed, defining what is called \emph{symplectic microfolds} and \emph{symplectic micromorphisms}; 
for references see, for example, \cite{Benoit, Benoit2, Benoit3}.\\ \\
Although these two different approaches are interesting per se in order to describe rigorously the symplectic category, we will consider the extended symplectic category as a categroid. If 
compositions of immersed canonical relations are involved, it would be proven that such compositions are well defined morphisms. More precisely we will have the following definition of the ``extended symplectic category'' as a categroid (from now on we will abuse the language and call it a category). 
\subsubsection{The categroid $\mbox{\bf{Symp}}^{Ext}$ } 
\begin{definition}\label{Extended}\emph{
$\mbox{\textbf{Symp}}^{Ext}$ is a categroid\footnote{This is not an honest category since the composition of immersed canonical relations is not in general a smooth manifold. }  in which the objects are (weak) symplectic manifolds. 
A morphism between two symplectic manifolds $(M,\omega_M)$ and $(N,\omega_N)$ is a pair $(L, \phi)$, where 
\begin{enumerate}
 \item $L$ is a smooth manifold (in the infinite dimensional setting we will restrict ourselves to Banach manifolds).
\item $\phi\colon L \to M \times N$ is an immersion.
\footnote{Observe here that usually one considers embedded Lagrangian submanifolds, but we consider immersed ones.}
\item $T\phi_x$ applied to $T_xL$ is a Lagrangian subspace of $T_{\phi(x)}(\overline M\times N),\, \forall x \in L$.
\end{enumerate}
We will call these morphisms \textit{immersed canonical relations} and denote them by $L\colon M \nrightarrow N$. In the sequel, for simplicity of the convention, we will denote by $(L,\phi)$ the pair defining an immersed canonical relation, where $\phi$ is a representative of the class $[\phi]$.
The partial composition of morphisms is given by composition of relations as sets.}
\end{definition}

\begin{remark}
\emph{
Observe that $\mbox{\textbf{Symp}}^{Ext}$ carries an involution $\dagger\colon(\mbox{\textbf{Symp}}^{Ext})^{op} \to  \mbox{\textbf{Symp}}^{Ext} $ that is the identity in objects and is the relational 
converse in morphisms, i.e. for $f\colon A \nrightarrow B$, $f^{\dagger}:= \{(b,a) \in B\times A \vert (a,b) \in f,\, a \in A,\, b\in B\}$.}
\end{remark}

\begin{remark}
\emph{This categroid extends the usual symplectic category in the sense that the symplectomorphisms can be thought in terms of immersed canonical relations, namely, if $\phi\colon (M,\omega_M) \to (N, \omega_N))$ is a symplectomorphism between two finite dimensional symplectic manifolds, then $(\mbox{graph}(\phi), \iota)$, where $\iota$ is the inclusion of $\mbox{graph}(\phi)$ in $M \times N$, is a morphism in $\mbox{\textbf{Symp}}^{Ext}$.} 
\end{remark}

\section{Lie groupoids and symplectic structures }
Categorically, we understand a groupoid as a small category where all the morphisms are invertible. 
Writing down explicitely  we get the following

\begin{definition}\label{groupoid} \emph{A groupoid in a small category $\mathcal {C}$ which has fiber products, corresponds to two objects $\mathcal{G},\, M$ and a set of morphisms 
in $\mathcal{C}$ described in the following diagram
\\
\xymatrixrowsep{4pc} \xymatrixcolsep{3pc} \xymatrix{
    &\,\,\,\,\;\;\,\,\;\,G\times_{(s,t)}   G  \ar[r]^{\,\,\,\,\,\;\;\;\;\;\;\;\;\mu}  & G \ar[r]^{i}   &G \ar@/_/[r]_t  \ar@/^/[r]^s & M \ar[l]_{\varepsilon}  & 
    }
\\
where $G\times_{(s,t)}   G$ is the fiber product for the maps $s,t: G \to M$, such that the following axioms hold (denoting $G_{(x,y)}:= s^{-1}(x) \cap t^{-1}(y)$):
\\
\newline
    \underline{\textbf{A.1}} $s\circ \varepsilon= t\circ \varepsilon =id_M$ \\
\newline
    \underline{\textbf{A.2 }} If $g \in G_{(x,y)}$ and $h \in G_{(y,z)}$ then $\mu (g,h)\in G_{(x,z)}$\\
\newline
    \underline{\textbf{A.3 }} $\mu(\varepsilon\circ s \times id_{G})=\mu(id_{G}\times \varepsilon\circ t)=id_{G}  $ \\
\newline
    \underline{\textbf{A.4 }} $\mu(id_{G} \times i)=\varepsilon \circ t$\\
\newline
    \underline{\textbf{A.5 }}$\mu(i \times id_{G})=\varepsilon \circ s $\\
\newline
    \underline{\textbf{A.6}} $\mu(\mu \times id_{G})=\mu (id_{G} \times \mu)$.\\}
\end{definition}
When the category $\mathcal{C}$ corresponds to one whose objects are smooth manifolds and morphisms are smooth maps, 
we call such groupoid a \textit{Lie groupoid}. \footnote{ In terms of the defining axioms, a Lie groupoid is defined whenever the source/ target are surjective submersions 
and both $M$ and $G$ are smooth manifolds, see \cite{Mor}.}

\subsection{Examples of Lie groupoids}
Any group can be understood as a groupoid with one single object. But more interesting, 
it is possible to associate groupoid structure to more general set of objects with different interpretations.
\begin{example} \emph{(Pair groupoid). In this case $G= M \times M$, the source and target corresponds to first and 
second projection respectively. The multiplication is given by \[\mu((x,y),(y,z))=(x,z),\] the inverse is transposition and the unit map is 
the diagonal map.}
\end{example}
 \begin{example} \emph{(Action groupoid). Given a Lie group $G$  and a smooth left $G-$ action on a smooth manifold $M$, 
the action groupoid, denoted by $G \ltimes M$, over $M$, has as space of morphisms the product $G \times M$, the source is the projection onto the first component and the target is the action map.The multiplication is given by
\[\mu((g,x),(g^{'},x^{'}))=(g\cdot g^{'}, x).\]}
\end{example}
\begin{example}\emph{(Bundle groupoid). In this example, we start with a vector bundle $E \to M$. The bundle groupoid has as objects the points of the manifold $M$ and the morphisms are the vectors of the 
fibers of $E$. The groupoid multiplication is fiber multiplication (addition). The source and target coincide and correspond to the bundle projection, the inverse is the fiber inverse and the unit is the zero section.}
 \end{example}
\begin{example}\emph{(Fundamental groupoid). This groupoid is denoted by $\Pi(M)$ and the space of morphisms is the space of homotopy classes of paths, leaving invariant the initial and final point.
The multiplication is induced by the concatenation of paths.}
\end{example}

\begin{definition} \emph{ A morphism $F\colon G \to H$ between two Lie groupoids $G$ and $H$ is a functor of categories in addition that it should respect the smooth structure, i.e. is a smooth map for the objects and morphisms}. 
\end{definition}
\begin{remark}\emph{
This allows us to define the category \textbf{Lie Grpd} with objects Lie groupoids and morphisms Lie groupoid morphisms. However, there is an extended version of this category where the objects are the same but a morphism between two groupoids $G$ and $H$ corresponds to an immersed subgroupoid of the groupoid product $G \times H$. This categroid will be denoted by $\textbf{Lie Grpd}^{Ext}$}.
\end{remark}
\begin{definition}
\emph{
A Lie groupoid such that the space of morphisms $G$ is equipped with a nondegenerate closed 2-form $\omega \in \Omega^2(G)$  
satisfying the following condition
\[\mu^{*}(\omega)= Pr_1^{*}(\omega)+Pr_2^{*}(\omega),\]
where $Pr_1$ and $Pr_2$ denote the first and second projections of $G\times G$ onto G, is called a \textbf{symplectic groupoid} and denoted by 
$(G,\omega)\rightrightarrows M$. In this case we say that the symplectic form $\omega$ is \emph{multiplicative}.}

\end{definition}
The multiplicativity of a 2-form $\omega$ is equivalent to the following condition on the multiplication map $\mu$, for finite dimensional symplectic groupoids:
\begin{lemma} \emph{\cite{Cost}.
$\omega$ is multiplicative if and only if the graph of the multiplication map is a Lagrangian submanifold of 
$G\times G \times \bar{G}$, where $\bar{G}$ denotes the manifold $G$ equipped with the opposite symplectic structure, $-\omega$. 
}
\end{lemma}

The multiplicativity condition imposes compatibility between $\omega$ and the other structure maps on the Lie groupoid. 
More specifically: 
\begin{proposition}\emph{Let $(G,\omega)\rightrightarrows M $ be a finite dimensional symplectic groupoid. Then, the following holds
\begin{enumerate}
 \item The image of the unit map, $\varepsilon(M)$ is a Lagrangian submanifold of $G$.
 \item The inverse map $i$ is an antisymplectomorphism.
 \item There is a unique Poisson structure on $M$ such that $s$ is a Poisson map (or equivalently, such that $t$ is an anti-Poisson map).
\end{enumerate}}
\end{proposition}

\begin{example}
\emph{The pair groupoid $M\times M \rightrightarrows M$, with $(M, \omega)$ a symplectic manifold, is a trivial example of a 
symplectic groupoid, where the symplectic structure on $M\times M$ is $\omega \oplus -\omega$.}
\end{example}

The next example will be fundamental in the construction of the symplectic groupoid that integrates Poisson structures:
\begin{example}\emph{\textbf{The cotangent bundle $T^*M$ of a manifold.} As a bundle groupoid, $T^*M \rightrightarrows M$ is a symplectic groupoid, 
where the symplectic structure on $T^*M$ is the canonical one (i.e. the exterior differential of the Liouville form).}
 
\end{example}
\begin{example}
\emph{
Given a Lie group $G$, we consider the action groupoid $G\times \mathfrak g^* \rightrightarrows \mathfrak g ^*$, with respect to the 
coadjoint action of $G$. By the triviality of the cotangent bundle of $G$, we can identify $G\times \mathfrak g^*$ with $T^* G$ and hence we can endow the action groupoid with the canonical symplectic 
structure of $T^*G$. 
}
\end{example}

\begin{definition} \emph{A \textbf{morphism between symplectic groupoids} $(G,\omega_G)$ and $(H,\omega_H)$ 
is a functor $\phi$ between $G$ and $H$ that is a bijective symplectomorphism on the space of morphisms, that means
\begin{enumerate}
 \item $\phi$ is bijective and $\phi^*(\omega_H)=\omega_G$ (symplectomorphism).
 \item $\phi(s(g))=s(\phi(g)),\, \phi(t(g))=t(\phi(g)),\forall g \in G$ (compatibility with source and target).
\item $\phi(\varepsilon(m))=\varepsilon(\phi(m)),\forall m \in M$  (compatibility with units).
\item $\phi(\mu(g_1,g_2))=\mu(\phi(g_1),\phi(g_2)), \forall g_1,\,g_2 \in G$ (compatibility with multiplication).
\end{enumerate}
Equivalently, a morphism of symplectic groupoids is a functor of Lie groupoids that respects the symplectic structure.}
\end{definition}
\begin{remark}\emph{This makes possible to define the category \mbox{\textbf{SympGrpd}} with objects symplectic groupoids and morphisms the above defined. However, there is an extended version of this category (as a categroid), denoted by  $\mbox{\textbf{SympGrpd}}^{Ext}$ where  morphisms are immersed Lagrangian subgroupoids of the groupoid product. This notion will be explored in detail later, after introducing relational symplectic groupoids}.
\end{remark}
\section {Lie algebroids}
Lie algebroids correspond to infinitesimal versions of Lie groupoids and a natural generalization for Lie algebras.
A pair $(A, \rho)$, where $A$ is a vector bundle over $M$ and $\rho$ (called the anchor map) is a vector bundle morphism from $A$ to $TM$ is called a 
\textit{Lie algebroid} if
\begin{enumerate}
 \item There is Lie bracket $[,]_{A}$ on $\Gamma(A)$ such that the induced map $\rho_{*}\colon \Gamma (A) \to \mathfrak{X}(M)$ is a Lie algebra homomorphism.
\item \textit{Leibniz identity:} \[[X,fY]_{A}=f[X,Y] + \rho_{*}(X)(f)Y, \forall\, X,Y \in \Gamma (A), f \in \mathcal{C}^{\infty}(M).\]
\end{enumerate}
\subsection{Examples}
The following are natural examples of Lie algebroids, and they appear naturally as infinitesimal versions of the Lie groupoids previously discussed.
\begin{example}\emph{(Lie algebras). Any Lie algebra $\mathfrak g$ is a Lie algebroid over a point. 
The anchor map in this case is the projection to the point.}\end{example}
\begin{example}\emph{(Lie algebra bundles). Consider a vector bundle $p: E\to M$ such that each fiber is 
a Lie algebra and in addition, for $x$ in $M$, there is an open set $U$ containing $x$, a Lie algebra $L$ and a homeomorphism
$    \phi\colon U\times L\to p^{-1}(U)$ such that
$\phi_x\colon x\times L \rightarrow p^{-1}(x)$
is a Lie algebra isomorphism.
The Lie bracket on sections for this bundle is given by the fiberwise well defined Lie bracket and the anchor map is $dp$.}
\end{example}
\begin{example}\label{Tangent}\emph{(Tangent bundles). The tangent bundle $TM$ of a manifold $M$ is naturally a Lie algebroid, where the Lie bracket in sections is the usual Lie bracket for vector fields 
and the anchor is the identity map.}
\end{example}
\begin{example}\label{cotangent}\emph{(Cotangent bundle of a Poisson manifold). If $M$ is a Poisson manifold, then $T^*M$ is a Lie algebroid, where $[,]_{T^*M}$ is the Koszul bracket for 1-forms, that is defined 
for exact forms by
$$[df,dg]:=d\{f,g\}, \forall f,g \in \mathcal C ^{\infty}(M),$$
whereas for general 1-forms it is recovered by Leibniz and the anchor map given by
$\Pi^{\#}\colon T^*M \to TM$.}  
\end{example}

To define a morphism of Lie algebroids we consider the complex $\Lambda^{\bullet}A^{*}$, where $A^*$ is the dual bundle and a differential $\delta_A$ is defined by the rules
\begin{enumerate}
 \item \[\delta_A f:= \rho^* df,\,\forall f \in \mathcal{C}^{\infty}(M).\] 
 \item 
\begin{eqnarray*}\langle \delta_A \alpha, X \wedge Y \rangle&:=& -\langle \alpha,\, [X,Y]_A  \rangle + \langle \delta \langle \alpha, X \rangle, Y \rangle\\ &-&\langle \delta \langle \alpha, Y \rangle, X \rangle, \, \forall X,Y \in \Gamma(A), \alpha \in \Gamma (A^*),
\end{eqnarray*}
where $\langle, \rangle$ is the natural pairing between $\Gamma(A)$ and $\Gamma(A^*)$.

\end{enumerate}
\begin{definition}\label{morphism} \emph{
A vector bundle morphism $\varphi\colon A \to B$ is a Lie algebroid morphism if 
\[\delta_A \varphi^*= \varphi^* \delta_B.\]
}
\end{definition}
It is important to observe that any Lie algebroid $(A, \rho)$ determines a (possibly singular) foliation on the manifold $M$ given by the involutive distribution $\mathfrak Im (\rho) \subset TM$. The orbit through $x\in M$ in such foliation will be denoted by $\mathcal O_x$.
\subsection{Lie algebroid of a Lie groupoid}
In the classical Lie theory, it is well known the connection between Lie algebras and Lie groups, i.e. the first ones are an infinitesimal version of the second ones. This can be made more precise by stating the following 
\begin{proposition}
\emph{Let $G$ be a Lie group. Then there is a natural Lie bracket on $T_eG$, where $e$ is the unit of $G$, determined uniquely by the left invariant vector fields on $G$.}
\end{proposition}
This sometimes is called \emph{differentiation} of Lie groups and it has a natural generalization for Lie groupoids, as follows.\\
Given a Lie groupoid $G \rightrightarrows M$, with source and target denoted by $s$ and $t$, respectively, we call a vector field $X$ on $G$ left invariant if it satisfies the following properties:
\begin{enumerate}
\item $X$ is tangent to the $t$-fibers (these are defined as $t^{-1}(x), \forall x \in M$).
\item For all $g \in G$, $X$ is invariant with respect to the left action of $G$ on itself, i.e. with respect to the (right) left translation
$$Lg\colon h\mapsto \mu(g,h) $$
that maps $t^{-1}(s(g))$ to $t^{-1}t(g), \, \forall g \in G$.
\end{enumerate}
The definition for right invariant vector field is analogous. 
Denoting by $\mathfrak X_L(G)$ the space of smooth left invariant vector fields on $G$, it is easy to check that $\mathfrak X_L(G)$ is a Lie subalgebra of $\mathfrak X(G)$.\\
Now, let us denote by $\mbox{Lie} (G)$ the vector bundle over $M$ with fibers corresponding to the tangent spaces to the $t-$ fibers,
$$\mbox{Lie} (G)_x= T_x (t^{-1}(x)),\, \forall x \in M.$$
By using left translation, the space of sections of $\mbox{Lie} (G)$ can be uniquely identified with the space of left invariant vector fields of $G$, therefore, it can be equipped with a Lie algebra structure. In addition, we have the map $ds\colon \mbox{Lie}(G) \to TM$, the differential of the source map and it can be checked that $\mbox{Lie} (G)$ with the previously defined Lie bracket on sections and $ds$ as anchor map is a Lie algebroid over $M$.\\
With this definition, it is possible to observe that, for example:
\begin{example}
\emph{ $\mbox{Lie} (G)=\mathfrak g$, where $G$ is a Lie group and $\mathfrak g$ is its associated Lie algebra}
\end{example}
\begin{example}
\emph{ $\mbox{Lie} (M \times M \rightrightarrows M)= TM$, where $M \times M \rightrightarrows M$ is the pair groupoid and $TM$ is the tangent bundle regarded as a Lie algebroid (Example \ref{Tangent}).}
\end{example}
\begin{example}
\emph{$\mbox{Lie}((G, \omega)\rightrightarrows M)= T^{*}(M)$, where $(G, \omega)\rightrightarrows M$ is a symplectic groupoid and $T^*(M)$ is the cotangent bundle of the Poisson manifold $(M, \Pi)$, regarded as a Lie algebroid.}
\end{example}

\section{Integrability of Lie algebroids}
Considering the differentiation procedure described above for Lie groupoids as a functor  (denoted by Diff), from the category $\mbox {\textbf{Lie Grpd}}$ to $\mbox{\textbf {LieAlgbd}}$ (this is in fact possible since in fact the differentiation procedure is functorial), a natural question is whether there exists an \emph{integration} functor from $\mbox{\textbf {LieAlgbd}}$ to $\mbox {\textbf{Lie Grpd}}$. Rephrasing this we can ask the following question
\begin{itemize}
\item Given a Lie algebroid $A$ over $M$, is there a Lie groupoid $G\rightrightarrows M$ such that Diff($G$)=$A$?
\end{itemize}
The idea of the integrability for Lie algebroid is to recover the groupoid $G$ out of the algebroid $A$ (for further details, see \cite{Crai, CraiNotes}.) 
First let start with a groupoid $G \rightrightarrows M$ that is connected and has simply connected s-fibers and $A= \mbox{Diff}(G)$. 
We will recover $G$ from the algebroid $A$. For this, we introduce the notions of $G$-paths and $A$-paths.
\begin{definition} \emph{ A  \emph{G-path} is  map $g\colon I=[0,1] \to G$ of type $\mathcal C^2$,  such that the following conditions are satisfied:
\begin{itemize}
\item $g(0)= \varepsilon(x)$, 
for some $x \in M.$
\item $s(g(\tau))=x, \forall \tau \in I.$
\end{itemize}
}
\end{definition}
The infinitesimal version of a \emph{G}-path is what we call an $A$-path, defined in general as follows

\begin{definition}\emph{
Let $\mathcal{(A, \rho, []_A)}$ be an algebroid, with projection $\pi\colon A \longrightarrow M$ and anchor $\rho$. A $\mathcal C ^1$-path $a\colon I \longrightarrow 
A$ is an $A$-path if
\begin{equation}
\rho(a(t))=\frac{d}{dt}(\pi(a(t)),\forall t\in I.
\end{equation}
This is equivalent to saying that $a\, dt\colon TI \longrightarrow A$ is a Lie algebroid morphism covering 
$\gamma:=\pi \circ a\colon I \longrightarrow M$. An $A$- path is called trivial if $a(t)=p, \forall t\in I$, for some $p \in M$.}
\end{definition}
More precisely, the $A$-path associated to a $G$-path $g$ is given by
\begin{equation}
a(\tau)= g(\tau)^{-1} \frac{dg}{d\tau}(\tau). 
\end{equation}
Now, given two $G$-paths $g$ and $g^{'}$ such that $g(0)=g^{'}(0)$, there exists a homotopy $g_{\rho}$ of $G$-paths with fixed end points 
on the same $s$-fiber \footnote{Here we use the fact that the $s$- fibers are simply connected} with $g_0=g$ and $g_1= g^{'}.$\\

The infinitesimal version of such homotopy gives rise to a family $a_{\rho}$ of $A$-paths. We can define an equivalence $\sim$
of $A$-paths by declaring that $a$ is equivalent to $a^{'}$ if the corresponding $G$-paths have the same initial and terminal point. Then we get \cite{Crai, CraiNotes} that
\begin{equation}
G=\mbox{\emph{A}-paths}/ \sim.
\end{equation}

Now, the plan is to give a description of such equivalence in terms of only the Lie algebroid $A$. 
 In order to do this, we consider a family $a_{\rho}$, of class $\mathcal C^2$ 
with respect to the parameter $\rho$, consisting of $A$-paths. Now, we denote by $\gamma_{\rho}:= \pi \circ a_{\rho}$ as the family 
of paths on the base manifold $M$. In the case where $a_{\rho}$ is the differentiation of a $G$-homotopy $g_{\rho}$, then 
\begin{equation}
\gamma_{\rho}(\tau)= s(g_{\rho}(\tau)). 
\end{equation}
Let us assume also that $\gamma_{\rho}(0)= \gamma_0(0)$ and that $\gamma_{\rho}(1)= \gamma_0(1), \forall \gamma \in I$.
We consider s family $\xi_{\rho}$ of $\mathcal C^2$- time dependent sections of $A$ satisfying
\begin{equation}
\xi_{\rho}(\tau, \gamma_{\rho}(\tau))= a_{\rho}(\tau) 
\end{equation}
and we define
\begin{equation}
b(\rho, \tau):=\int_0^{\tau} \phi_{\xi_{\rho}}^{\tau, r} \frac{d\xi_{\rho}}{d\rho}(r, \gamma_{\rho}(r))dr, 
\end{equation}
where here $\phi_{\xi_{\rho}}^{\tau, r}$ denotes the flow in $A$ generated by the section $\xi_{\rho}$.
It can be proven \cite{Crai} that $b$ does not depend on the choice of the time dependent sections $\xi_{\rho}$, 
it only depends on $a_{\rho}$.
Explicitely, we get that
\begin{equation}
b(\rho, \tau)= g_{\rho}(\tau)^{-1} \frac{dg_{\rho}}{d\rho}. 
\end{equation}
Moreover if the family of $G$-paths $g_{\rho}$ is a $G$-homotopy, then the following equation holds 
\begin{equation}\label{Apathcondition}
b(\rho,1)=0, \forall \rho \in I. 
\end{equation}

This suggests the following two equivalent definitions for $A$-homotopy
\begin{definition} \label{Ahomotopy1}
 \emph{\cite{Crai}. Two $A$- paths $a$ and $a^{'}$  are $A$-homotopic if $\pi\circ a(0)= \pi\circ a^{'}(0),\, \pi\circ a(1)= \pi\circ a^{'}(1) $ and they are connected by a $\mathcal C^2$-family of $A$-paths $a_{\rho}$ such that $a_0=a, a_1= a^{'}$  and Equation \ref{Apathcondition} holds.}
\end{definition}

\begin{definition}\label{Ahomotopy2} \emph{\cite{Crai}. Two $A$- paths $a$ and $a^{'}$ are called \emph{$A$-homotopic}  if there exists a Lie algebroid morphism 
$g\colon T\square \longrightarrow \mathcal{A}$ (where $\square$ is a square with four distinguished boundary components $I_j, \, 0\leq j \leq 3$ such that $I_0$ and $I_2$ are horizontal) in such way that,
restricting to $\partial \square$ we have $g\vert_{TI_0}=a_0,\, \pi\circ(g \vert_{TI_1})\equiv a(0),\,
\pi\circ(g \vert_{TI_3})\equiv a(1)$ and $g\vert_{TI_2}=a_1$.
}
\end{definition} 

Following these definitions we define 
\begin{equation}
G=\mbox{\emph{A}-paths}/\mbox{\emph{A}-homotopy}.
\end{equation}
\begin{proposition}
\emph{$G$ is in general a topological groupoid and in the integrable case, is the s-fiber simply connected Lie groupoid integrating the Lie algebroid $A$.}
\end{proposition}
\begin{proof}
The topological \footnote{Here we consider the uniformly convergence topology for the space of $A$-paths.} groupoid structure $G \rightrightarrows M$ is given by 
\begin{eqnarray*}
s([a])&=&a(0)\\
t([a])&=& a(1)\\
\varepsilon(x)&=& [a \vert \pi\circ a\equiv x ]\\
\mu([a_1],[a_2])&=&[a_1* a_2], \, ([a_1],[a_2]) \in G\times_{(s,t)} G\\
\iota ([a])&=&[a^{-1}].
\end{eqnarray*}
One should prove that the multiplication (where here * denotes path concatenation) is well defined with respect to $A$- homotopy. Independently in \cite{Crai} and \cite{Cat} \footnote{Here it is proven for the case where $A= T^*M$ but the argument can be adapted for general algebroids} it is proven that the concatenation of two $A$-paths of class $\mathcal C^1$ is homotopic to a $\mathcal C^1$-class $A$-path.
\end{proof}
This is sometimes called in the literature the \emph{Weinstein groupoid} of $A$ \cite{Crai}.

\begin{figure}
\psfrag{Rq}{$\mathbb{R}^q$}
\centering%
\center{\includegraphics[scale=0.55]{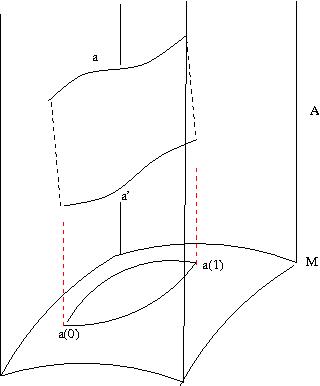}}
\caption{A-homotopy of A-paths.}
\label{fig:FigureExample}
\end{figure}

\subsection{Obstruction for integrability}Now we describe the integrability obstructions for Lie algebroids, following the work of Crainic and Loja Fernandes \cite{Crai}.

Let us consider  $\mathfrak{g}_x:= \ker\rho$, 
called the \emph{isotropy Lie algebra} 
of $A$ at the point $x \in M$ and let $G(\mathfrak{g}_x)$ be the simply connected Lie group integrating $\mathfrak{g}_x$, defined as

\begin{equation}
G(\mathfrak{g}_x)=\mathfrak{g}_x\mbox{-paths}/\mathfrak{g}_x\mbox{-homotopy}.
\end{equation}
Since the space of $\mathfrak g_x$-paths is naturally included in the space of $A$-paths (as $A$- paths with constant base path), there is a natural group homomorphism 
\begin{equation}
\alpha_x\colon G(\mathfrak{g}_x)\to G(A).
\end{equation}
Let $\tilde{N}_x(A):= \ker{\alpha_x}$ be the subgroup of $G(\mathfrak{g}_x)$ consisting of the $\mathfrak g_x$-paths that are $A$-homotopy equivalent.
The smoothness of the groupoid $G(A)$ implies that $\tilde{N}_x(A)$ is a discrete subgroup, that is equivalent to the fact that the quotient $G(\mathfrak g_x)/ \tilde{N}_x(A)$ is a Lie group with Lie algebra $\mathfrak g_x$.\\
There is an alternative way to present the group $\tilde{N}_x(A)$, as the image of the second monodromy group of the leaf $\mathcal O_x$ through $x$. First, we consider an element $[\gamma] \in \pi^2(\mathcal O_x, x)$ such that $\partial \gamma$ is the constant path at $x$ and a Lie algebroid morphism 
\begin{eqnarray}
ad\tau + b d\rho\colon TI \times TI &\to& A\\
u(\rho, \tau) \frac{\partial}{\partial \tau}+v(\rho, \tau) \frac{\partial}{\partial \rho} &\mapsto& u(\rho, \tau)a(\rho, \tau)+ v(\rho, \tau)b(\rho, \tau)
\end{eqnarray}
such that it lifts 
\begin{equation}
d\gamma\colon TI \times TI \to T\mathcal O_x
\end{equation}
via $\rho$ and in addition $a(0,\rho)=b(\rho,0)=b(\rho,1)\equiv 0$. The fact that $\partial\gamma$ is constant implies that the path $a_1(\tau):=a(1,\tau)$ is a $\mathfrak g_x$-path. Let $\partial (\gamma)$ be  the equivalence class under $A$-homotopy of $a_1$ in $G(\mathfrak g_x)$. The following Lemma ensures that $\partial$ is a well defined map
\begin{lemma} \emph{\cite{Crai}
The map $\partial$ defined by 
\begin{equation}
\partial\colon \gamma \to \partial (\gamma)
\end{equation}
depends only on the class $[\gamma] \in \pi_2(\mathcal O_x, x)$ and defines a group homomorphism $\partial\colon \pi^2(\mathcal O_x, x)\to G(\mathfrak g_x)$ such that $\mathfrak{Im}(\partial)= \tilde{N}_x(A).$}
\end{lemma}
We denote by $N_x(A)$ the subset of elements in $\mathfrak g_x$ defined by
\begin{equation}
N_x(A):= \{v \in \mathfrak g_x  \vert v \sim x\},
\end{equation}

where here $v$ is regarded as a constant $A$-path, $x$ is the trivial $A$-path over $x$ and $\sim$ denotes $A$- homotopy equivalence.
It is proven that
\begin{lemma}\emph{\cite{Crai}.
There is a group isomorphism between $N_x(A)$ and $\tilde{N}_x(A) \cap Z^0(G(\mathfrak g_x))$, where $Z^0(G(\mathfrak g_x))$ denotes the connected component of the identity in the center $Z(G(\mathfrak g_x))$ of $G(\mathfrak g_x)$.}
\end{lemma}
Now, we are able to describe the integrability conditions for Lie algebroids. For that we need the notion of \emph{locally uniform discreteness} of the groups $N_x(A)$. We fix a metric $d$ on $A$ that is continous along the $A$-fibers and we define the function $r(x)$ by 
\begin{equation}
r(x) = \left\{ \begin{array}{cc}
\infty &\mbox{ if $N_x(A)=0$} \\
 d(0, N_x(A)- \{0\}) &\mbox{ otherwise}
       \end{array} \right.
\end{equation}
\begin{theorem}\label{Integrability}\emph{(Crainic- Fernandes \cite{Crai}). A Lie algebroid $A$ is integrable if and only if the two following conditions hold
\begin{enumerate}
\item The group $N_x(A)$ is discrete
\item $\lim \inf_{y \to x} r(y)>0$ \emph{(uniform discreteness)}.
\end{enumerate}}
\end{theorem}
Moreover, Crainic and Fernandes give a formula that computes explicitly for many cases the monodromy groups and is possible to check the integrability conditions. More concretely, given the surjective map $\rho \mid_{\mathcal O}: A_{\mathcal O}\to T \mathcal O$ and a splitting $\sigma\colon T \mathcal O \to A_{\mathcal O}$, we define the curvature of $\sigma$ by
\begin{equation}
\Omega_{\sigma}(X,Y)= [\sigma(X), \sigma(Y)]- \sigma([X,Y])
\end{equation}
\begin{lemma}\emph{\cite{Crai}. If the image of $\Omega_{\sigma}$lies in $Z(\mathfrak g_{\mathcal O})$, then
\begin{equation}
N_x(A)= \{ \int_{\gamma} \Omega_{\gamma} \mid [\gamma] \in \pi_2(\mathcal O_x,x)\}\subset Z(\mathfrak g_x).
\end{equation}
}
\end{lemma}
By this formula it is possible to compute the monodromy groups for non integrable Lie algebroids and check that  the integrability conditions are violated (see Example \ref{nonintegrable}).
\subsection{The Poisson case}
When we refer to integration of Poisson manifolds (or Poisson brackets), we mean the integration of the algebroid $T^*M$ over $M$ as in Example \ref{cotangent}. 
More specifically, it can be checked that the Lie algebroids $T^*M$ for $M$ a zero or linear Poisson manifold are integrable, and also any 2-dimensional Poisson manifold \cite{Crai, Cat}. We will discuss some of these cases and also a non integrable one (Example \ref{nonintegrable}).
In the next chapter we will discuss the construction of the integration of Poisson manifolds through PSM and we will come back to the integrability conditions for Lie algebroids.

\chapter{Poisson sigma models and the main construction}\label{PSMmain}
In the first part of this chapter we introduce the Poisson sigma model
associated to a Poisson manifold
(the classical version of the model through the Hamiltonian
formalism). After the construction of the reduced phase space of PSM
associated to integrable Poisson manifolds we generalize
the construction to the non reduced version by defining the
\emph{relational symplectic groupoid}.
We give examples and we concentrate our attention on the examples
coming from Poisson geometry.

\begin{definition} 
\emph{
A Poisson sigma model (PSM) corresponds to the following data:
\begin{enumerate}
\item A compact surface $\Sigma$, possibly with boundary, called the \textit{source}.
\item A finite dimensional Poisson manifold $(M,\Pi)$, called the \textit{target}.
\end{enumerate}
}



\end{definition}
The space of fields for this theory is denoted with $\Phi$ and corresponds to the space of vector bundle morphisms of class $\mathcal C^{k+1}$ between $T\Sigma$ and $T^*M$.
This space can be parametrized by a pair $(X, \eta)$, where $X\in \mathcal C^{k+1}(\Sigma, M)$  and $\eta \in 
\Gamma^k(\Sigma, T^*\Sigma \otimes X^*T^*M),$ and $k \in \{0,\,1,\, \cdots \, , \infty\}$ denotes the regularity type of the map that we choose to work with.\\
On $\Phi$, the following first order action  is defined:
\[S(X,\eta):= \int_{\Sigma} \langle \eta,\, dX\rangle+  \frac 1 2 \langle \eta, \, (\Pi^{\#}\circ X) \eta \rangle,\]
where,

 \begin{eqnarray}
 \Pi^{\#}\colon T^*M&\to& TM\\
  \psi &\mapsto& \Pi(\psi, \cdot).
 \end{eqnarray}
 

Here,  $dX$ and $\eta$ are regarded as elements in $\Omega^1(\Sigma, X^*(TM))$, $\Omega^1(\Sigma, X^*(T^*M))$, respectively and  $\langle \,,\, \rangle $ is the pairing between $\Omega^1(\Sigma,  X^*(TM))$ and $\Omega^1(\Sigma,  X^*(T^*M))$ induced by the 
natural pairing between $T_xM$ and $T_x^*M$, for all $x \in M$.\\
\begin{remark}\emph{
This model has significant importance for deformation quantization. Namely, the perturbative expansion through Feynman path integral 
for PSM, in the case that $\Sigma$ is a disc, gives rise to Kontsevich's star product formula 
\cite{Kontsevich1, Cat, CattaneoDeformation}, i.e. the semiclassical expansion of the path integral
\begin{equation}
\int_{X(r)=x}f(X(p)) g (X(q))\exp{(\frac{i}{\hbar} S(X,\eta)) dX d\eta}
\end{equation}
around the critical point $X(u)\equiv x, \eta \equiv 0$, where $p,q$ and $r$ 
are three distinct points of $\partial \Sigma$, corresponds to the star product $f\star g (x)$. More details on the star product for Poisson manifolds are given on Chapter \ref{Frobenius}.}
\end{remark}

\section{PSM and its phase space}
For this model, we consider the constraint equations and the space of gauge symmetries. These will allow us to understand the geometry of the phase space and its reduction. First, we define
\[\mbox{EL}_{\Sigma}=\{\mbox{Solutions of the Euler-Lagrange equations}\}\subset \Phi,\]
where, using integration by parts \[\delta S= \int_{\Sigma} \frac{\delta \mathcal{L}}{\delta X} \delta X+ \frac{\delta \mathcal{L}}{\delta \eta} \delta \eta + \mbox{boundary terms}.\]The partial variations correspond to:
\begin{eqnarray}
\frac{\delta \mathcal{L}}{\delta X}&=& dX+ \Pi^{\#}(X)\eta=0\\
\frac{\delta \mathcal{L}}{\delta \eta}&=& d\eta+ \frac 1 2 \partial \Pi^{\#}(X)\eta\wedge\eta=0.
\end{eqnarray}

Now, if we restrict to the boundary, the general space of boundary fields corresponds to

\[\Phi_{\partial}:=\{\mbox{vector bundle morphisms between  } T (\partial \Sigma) \mbox{  and  } T^*M\}.\]
 
Following \cite{Cat, AlbertoPavel1}, $\Phi_{\partial}$ is endowed with a weak symplectic form and a surjective submersion $p\colon  \Phi \to \Phi_{\partial}$. 
Explicitely we have the following description of the space of boundary fields:
 \begin{remark}\emph{$\Phi_{\partial}$ can be identified with the Banach manifold $$P(T^*M):= \mathcal C^{k+1}(I, T^*M),$$ therefore it is a weak symplectic Banach manifold locally modeled by $\mathcal C^{k+1}(I, \mathbb R ^{2n})$.  
 } 
 \end{remark}
In order to see this, we understand $\Phi_{\partial}$ as a \emph{fiber bundle} over the path space $PM$, that is naturally equipped with the topology of uniform convergence. The fibers of the bundle are isomorphic to the Banach space of class $\mathcal C^{k}$
\begin{equation}
T^*_X(PM):= \Omega^1(I, X^*(T^*M))
\end{equation}
Therefore, as a set, $\Phi_{\partial}$ corresponds to
\begin{equation}
\Phi_{\partial}= \bigcup_{(X \in PM)} T^*_X(PM).
\end{equation}
The identification with $P(T^*M)$ is explicitly given by
\begin{eqnarray}
\psi\colon  T^*(PM) &\to& P(T^*M)\\
     (X,\eta) &\mapsto& (\gamma\colon  t \mapsto (X(t),\eta(t)))
\end{eqnarray}
and this allows to define a 2 form $\omega$ in $\phi_{\partial}$ in the following way. 
Identifying the tangent space $T_{\gamma} (P(T^*(M)))$ with the space of vector fields along the curve $\gamma$
\begin{equation}
T_{\gamma} (P(T^*(M)))= \{\delta\gamma: I \to TT^*M \mid \delta \gamma(t) \in T_{\gamma(t)}T^*M \}
\end{equation}
The two form $\omega$ in $\phi_{\partial}$ is given by
\begin{equation}
\omega_{\gamma}(\delta_1\gamma, \delta_2\gamma)= \int_0^1 \omega^{Liouv}(\delta_1\gamma(t),\delta_2 \gamma(t)) dt,
\end{equation}
where $\omega^{Liouv}= d \alpha^{Liouv}$ is the canonical symplectic form on $T^*M$. In local coordinates, where we consider paths in an open neighborhood of a point $X(0)=p \in M$,  if  $\gamma$ is described by the functions $X^1(t), X^2(t), \cdots X^n(t) \in \mathcal C^{k+1}(I)$ and $\eta_1,\eta_2,\cdots \eta_n \in \Omega^1(I)$ of class $\mathcal C^{k}$, then
\begin{equation}\label{symplectic}
\omega_{\gamma}(\delta_1\gamma, \delta_2\gamma)= \int_0^1(\delta_1X^i(t)\delta_2\eta_i(t)-\delta_2X^i(t)\delta_1\eta_i(t))dt.
\end{equation}
This form is clearly closed since
\begin{equation}
d\omega_{\gamma}(\delta_1\gamma, \delta_2\gamma,\delta_3\gamma)= \int_0^1 d\omega^{Liouv}(\delta_1\gamma(t),\delta_2 \gamma(t), \delta_3\gamma(t)) dt=0,
\end{equation}
and it is weak symplectic since we have for the map
\begin{equation}
\omega^{\sharp}(\delta \gamma)= \omega_\gamma(\delta \gamma, \cdot),
\end{equation}
if $\omega^{\sharp}(\delta_1 \gamma)=\omega^{\sharp}(\delta_1^{'} \gamma)$, then, setting $\delta_1\eta_i\equiv 0, \forall 1 \leq i \leq n$
\begin{equation}
\int_0^1 (\delta_1X^i(t)-\delta_1(X^i)^{'}(t))\delta_2\eta_i(t)=0, \forall \delta_2\eta_i(t),
\end{equation}
which implies that $\delta_1X^i(t)=\delta_1(X^i)^{'}(t)$ and setting $$\delta_2\eta_i\equiv 0, \forall 1 \leq i \leq n$$ we conclued in a similar way that $\delta_1\eta_i(t)=\delta_1(\eta_i)^{'}(t).$

Now, we define, following \cite{AlbertoPavel1}
$$L_{\Sigma}:= p(EL),$$
where $p: \phi \to C_{\partial \Sigma}$ is a surjective submersion and $ C_{\partial \Sigma}$ denotes the space of Cauchy data of the PSM restricted to the boundary $\partial \Sigma$.

Finally, we define $C_{\Pi}$ as the set of fields in $\Phi_{\partial}$ which  can be completed to a field in  $L_{\Sigma^{'}}$, with $\Sigma^{'}:= \partial \Sigma \times [0, \varepsilon]$, for some $\varepsilon$.  \\
 It can be proven that 
 \begin{proposition}\label{Coiso} \emph{\cite{Cat}.
 \begin{enumerate}
 \item The space $C_{\Pi}$ is described by
 \begin{equation}\label{Cois}C_{\Pi}= \{(X,\eta)\vert dX= \pi^{\#}(X)\eta,\, X\colon  \partial \Sigma \to M, \, \eta \in \Gamma (T^*I \otimes X^*(T^*M))\}.\end{equation}
\item The space $C_{\Pi}$ is a coisotropic Banach submanifold of $\Phi_{\partial}$ and its associated characteristic foliation has codimension $2n$, where $n = \emph{dim} (M)$.
\end{enumerate}
}
\end{proposition}
In fact, the converse of the second property also holds in the following sense. If we define $S(X,\eta)$ and $C_{\Pi}$ 
in the same way as before, without assuming that $\Pi$ satisfies Equation
(\ref{SN}) it can be proven that

\begin{proposition}\label{Coi}\emph{\cite{Coiso}.
If $C_{\Pi}$ is a coisotropic submanifold of $\Phi_{\partial}$, then $\Pi$ is a Poisson bivector field.}
\end{proposition}
The following geometric interpretation of $\mathcal C_{\Pi}$   will lead us to the connection between Lie algebroids and Lie groupoids in Poisson geometry with PSM. Following Definition \ref{morphism} for a morphism of Lie algebroids, in local coordinates, the condition for a vector bundle morphism to preserve the Lie algebroid structure gives rise to some PDE's that the anchor maps and the structure functions for $\Gamma(A)$ and  $\Gamma(B)$ should satisfy. For the case of PSM, regarding $T^*M$ as a Lie algebroid, we can prove that 

\[C_{\Pi}=\{\mbox{Lie algebroid morphisms between  } T (\partial \Sigma) \mbox{  and  } T^*M\},\]
where the Lie algebroid structure on the left is given by the Lie bracket of vector fields on $T (\partial \Sigma)$ with identity anchor map.

\section{Symplectic reduction}
Since $C_{\Pi}$ is a coisotropic submanifold, it is possible to perform symplectic reduction, which yields, when it is smooth, a symplectic finite dimensional manifold. In the case of $\Sigma$ being a rectangle and with vanishing boundary conditions for $\eta$ (see \cite{Cat}), following the notation in \cite{Crai} and \cite{Severa}, we could also reinterpret the reduced phase space $\underline{C_{\Pi}}$ as 

\[\underline{C_{\Pi}}=\left \{ \frac{\mbox{$T^*M$-paths}}{ T^*M \mbox{-homotopy}} \right \}.\]
In the integrable case, it was proven in \cite{Cat} that

\begin{theorem} The following data
\begin{eqnarray*}
G_0&=& M \\
G_1&=& \underline{C_{\Pi}}\\
G_2&=&\{ [X_1, \eta_1], [X_2, \eta_2] \vert X_1(1)=X_2(0)\}
\\
m&\colon & G_2 \to G:= ([X_1, \eta_1], [X_2, \eta_2]) \mapsto [(X_1* X_2, \eta_1* \eta_2)] \\
\varepsilon&:& G_0 \to G_1:= x\mapsto [X\equiv x, \eta\equiv 0] \\
s&\colon &G_1 \to G_0:= [X, \eta]\mapsto X(0) \\
t&\colon &G_1 \to G_0:= [X, \eta]\mapsto X(1) \\
\iota&:& G_1 \to G_1:= [X, \eta] \to [i^* \circ X,i^* \circ \eta]\\
&\mbox{                   }& i\colon [0,1]\to [0, 1]:= t\to 1-t, 
\end{eqnarray*}
correspond to a symplectic groupoid that integrates the Lie algebroid $T^*M$.  \footnote{here $*$ denotes path concatenation}
\end{theorem}
\begin{remark}\emph{In \cite{Cat}, this construction is also expressed as the Marsden-Weinstein reduction of the Hamiltonian action  of the (infinite dimensional) Lie algebra
$P_0\Omega^1(M):= \{\beta \in  \mathcal C^{k+1}(I, \Omega^1(M)) \mid \beta(0)=\beta(1)=0\}$
with Lie bracket
\begin{equation}
[\beta,\gamma](u)=d\langle \beta(u), \Pi^{\sharp} \gamma(u) \rangle- \iota_{\Pi^{\sharp}(\beta(u))} d\gamma(u)+ \iota_{\Pi^{\sharp}(\gamma(u))} d\beta(u)
\end{equation}
on the space $T^*(PM)$, on which the moment map $ \mu\colon  T^*(PM) \to P_0\Omega^1(M)^*$  is described by the equation 
\begin{equation}
\langle \mu(X,\eta), \beta \rangle= \int_0^1 \langle dX(u)+ \Pi^{\sharp}(X(u))\eta(u), \beta(X(u),u) \rangle du.
\end{equation}}
\end{remark}


\section{Relational symplectic groupoids}
This section contains the general description of relational symplectic groupoids, defined as special objects in $\mbox{\textbf{Symp}}^{Ext}$. It is a way to model the space of boundary fields before reduction of the PSM and to define a more general version of integration of Poisson manifolds. We give the main definitions, we discuss the connection with Poisson structures and we give some natural examples. For the motivational example of $T^*PM$ we prove that in fact we obtain relational symplectic groupoids for any Poisson manifold $M$ and we explain geometrically the integrability conditions for $T^*M$ in terms of the immersed canonical relations defining the relational symplectic groupoid.
\begin{definition}\emph{
A \textbf{relational symplectic groupoid} is a triple $(\mathcal G,\, L,\, I)$ where 
\begin{enumerate}
 \item $\mathcal G$ is a weak symplectic manifold. \footnote{In the infinite dimensional setting we restrict to the case of Banach manifolds.}
\item $L$ is an immersed Lagrangian submanifold of $\mathcal G ^3.$
 \item $I$ is an antisymplectomorphism of $\mathcal G$ called the \emph{inversion},
\end{enumerate}
satisfying the following six axioms A.1-A.6:}
\end{definition}

\begin{itemize}
 \item \textbf{\underline{A.1}} $L$ is cyclically symmetric, i.e. if $(x,y,z) \in L$, then $(y,z,x) \in L$.

\item \textbf{\underline{A.2}} $I$ is an involution (i.e. $I^2=Id$).\\

\textbf{\underline{Notation}} $L$ is an immersed canonical relation $\mathcal G \times \mathcal G \nrightarrow \bar{\mathcal G}$ and will 
be denoted by $L_{rel}.$ Since the graph of $I$ is a Lagrangian submanifold of $\mathcal G \times \mathcal G$, $I$ is an immersed canonical 
relation $\bar {\mathcal G} \nrightarrow \mathcal G$ and will be denoted by $I_{rel}$.\\
$L$ and $I$ can be regarded as well as immersed canonical relations 
\[ \bar {\mathcal G} \times \bar{\mathcal G} \nrightarrow \mathcal G \mbox{ and } \mathcal G \nrightarrow \bar{\mathcal G}\]
respectively, which will be denoted by $\overline{L_{rel}}$ and $\overline{I_{rel}}.$ The transposition 
\begin{eqnarray*} 
T\colon  \mathcal G \times \mathcal G &\to& \mathcal G \times \mathcal G\\
(x,y) &\mapsto& (y,x)
\end{eqnarray*}
induces canonical relations
\[T_{rel}\colon \mathcal G \times \mathcal G \nrightarrow \mathcal G \times \mathcal G \mbox{ and } \overline{T_{rel}}: 
\bar {\mathcal G}\times \bar{\mathcal G} \nrightarrow \bar {\mathcal G}\times \bar{\mathcal G}.\]
The identity map $Id\colon  \mathcal G \to \mathcal G$ as a relation will be denoted by $\mbox{Id}_{rel}\colon  \mathcal G \nrightarrow \mathcal G$ and by $\overline{\mbox{Id}_{rel}}\colon  \overline{\mathcal G} \nrightarrow \overline{\mathcal G}. $

Since $I$ and $T$ are diffeomorphisms, it follows that $I_{rel} \circ L_{rel}$ and $\overline{L}_{rel} \circ \overline{T}_{rel}\circ (\overline{I_{rel}} \times \overline{I_{rel}})
$ are immersed submanifolds.  For a relational symplectic groupoid we want that these two compositions to be morphisms $\mathcal G \times \mathcal G \nrightarrow \mathcal G$, and moreover we want them to coincide.

\item \textbf{\underline{A.3}} 
\begin{enumerate}
\item The compositions $I_{rel} \circ L_{rel}$ and $\overline{L}_{rel} \circ \overline{T}_{rel}\circ (\overline{I_{rel}} \times \overline{I_{rel}})$ are immersed submanifolds of $\mathcal G^3$.
\item $I_{rel} \circ L_{rel}$ and $\overline{L}_{rel} \circ \overline{T}_{rel}\circ (\overline{I_{rel}} \times \overline{I_{rel}})$ are Lagrangian submanifolds of $\overline{\mathcal G^2} \times \mathcal G$.
\item 
\begin{equation}
I_{rel} \circ L_{rel}= \overline{L}_{rel} \circ \overline{T}_{rel}\circ (\overline{I_{rel}} \times \overline{I_{rel}}).
\end{equation}
\end{enumerate}


Now, define \[L_3:= I_{rel} \circ L_{rel}\colon  \mathcal G \times \mathcal G \nrightarrow \mathcal G.\] As a corollary of the previous axioms we get that
\begin{corollary}
$\overline{I_{rel}}\circ L_3 = \overline{L_3} \circ \overline{T_{rel}}\circ(\overline{I_{rel}}\times \overline{I_{rel}}).$ 
\end{corollary}
\begin{proof}By Axiom A.2 and by definition of $L_3$, the left hand side of the equation, as a relation from $\mathcal G \times \mathcal G$ to $\overline {\mathcal G}$ 
can be rewritten as 
\begin{equation}\label{A}
\overline{I_{rel}}\circ L_3= \overline{Id_{rel}}\circ L_{rel}.
\end{equation}
In the right hand side, by axiom A.3, we can rewrite
\begin{eqnarray}
\overline{L_3} \circ \overline{T_{rel}}\circ(\overline{I_{rel}}\times \overline{I_{rel}})&=&
\overline{I_{rel}}\circ (\overline{L_{rel}}\circ \overline{T_{rel}}\circ (\overline{I_{rel}}\times \overline{I_{rel}}))\\
&=&\overline{I_{rel}}\circ I_{rel} \circ L_{rel}\\
&=& \overline{Id_{rel}} \circ L_{rel}.\label{B} 
\end{eqnarray}
Comparing (\ref{A}) and (\ref{B}) we obtain the desired result.
\end{proof}

\item \textbf{\underline{A.4}} 
\begin{enumerate}
\item The compositions $L_3 \circ (L_3 \times \Id)$ and $L_3\circ (\Id \times L_3)$ are immersed subamifolds of $\mathcal G^4$.
\item  $L_3 \circ (L_3 \times \Id)$ and $L_3\circ (\Id \times L_3)$ are Lagrangian submanifolds of $\overline{\mathcal G^3}\times \mathcal G$.
\item \begin{equation}\label{asso}
L_3 \circ (L_3 \times Id)= L_3\circ (\Id \times L_3)\end{equation}
\end{enumerate}

\begin{remark}\emph{The part 2 of A.4. follows automatically in the finite dimensional case from the fact that, since $I$ is an antisymplectomorphism, its graph is Lagrangian, 
therefore $L_3$ is Lagrangian, and so $(Id \times L_3)$ and $(L_3 \times Id).$}
\end{remark}
The graph of the map $I$, as a relation $* \nrightarrow \mathcal G \times \mathcal G$ will be denoted by $L_I$.\\
\item \textbf{\underline{A.5}} 
\begin{enumerate}
\item The compositions
$L_3 \circ L_I$ and $L_3\circ(L_3\circ L_I\times L_3\circ L_I)$ are immersed submanifolds of $\mathcal G$. 
\item $L_3 \circ L_I$ and $L_3\circ(L_3\circ L_I\times L_3\circ L_I)$ are Lagrangian submanifolds of $\mathcal G$.
\item Denoting by $L_1$ the morphism $L_1:= L_3\circ L_I \colon * \nrightarrow \mathcal G$, then
\begin{equation} \label{unit}
 L_3\circ(L_1 \times L_1)= L_1.
\end{equation}
 \end{enumerate}

 From the definitions above we get the following
\begin{corollary}\emph{
\[\overline{I_{rel}}\circ L_1= \overline{L_1},\]
that is also equivalent to 
\[I(L_1)= \overline{L_1},\]
where $L_1$ is regarded as an immersed Lagrangian submanifold of $\mathcal G$.}
\end{corollary}

\begin{proof}
We have that
\begin{eqnarray}
\overline{I_{rel}}\circ L_1&=& \overline{I_{rel}}\circ L_3 \circ L_1\\
 &\stackrel{\mbox{Cor.1}}{=}& \overline{L_3} \times \overline{T_{rel}} \circ (\overline{I} \times \overline{I})\circ L_I.
\end{eqnarray}
As sets, we have that
\begin{eqnarray}
\overline{T_{rel}}\circ (\overline{I} \times \overline{I})\circ L_I&=& T(I\times I (L_I))\\
L_I&=&\{(x,\, I(x)), \, x \in \mathcal G\}\\
I\times I (L_I)&=& \{(I(x), I^2(x)),\, x\in \mathcal G \}\\
&\stackrel{A.2}{=}&\{(I(x), x),\, x\in \mathcal G\}\\
T(I\times I (L_I))&=& \{(x, I(x)),\, x \in \mathcal G\}= L_I \label{C}. 
\end{eqnarray}
From (\ref{C}) we get
\[\overline{T_{rel}}\circ (\overline{I} \times \overline{I})\circ L_I= \overline{L_1}
\] and therefore
\[\overline{I_{rel}}\circ L_1 = \overline{L_3} \circ \overline{L_I}= \overline{L_1}.\]

\end{proof}

\item \textbf{\underline{A.6}} 

\begin{enumerate}
\item $L_3\circ (L_1 \times \Id)$ and $L_3\circ (\Id \times L_1)$ are immersed submanifolds of $\mathcal G \times \mathcal G.$
\item $L_3\circ (L_1 \times \Id)$ and $L_3\circ (\Id \times L_1)$ are Lagrangian submanifolds of $\overline{\mathcal G}\times \mathcal{G}$.
\item If we define the morphism $$L_2:=L_3\circ (L_1 \times \Id)\colon  \mathcal G \nrightarrow \mathcal G,$$
then the following equations hold
\begin{enumerate}
\item
\begin{equation}
L_2=L_3\circ (\Id \times L_1).
\end{equation}
\item $L_2$ leaves invariant $L_1$ and $L_3$, i.e.
\begin{eqnarray} 
L_2\circ L_1&=& L_1\label{invariance1}\\
L_2\circ L_3&=& L_3\circ (L_2 \times L_2)=L_3\label{invariance3}.
\end{eqnarray}
\item 
\begin{equation}
\overline{I_{rel}}\circ L_2= \overline{L_2}\circ \overline{I_{rel}}\mbox{ and } L_2^{\dagger}=L_2.\label{symmetric}
\end{equation}
\end{enumerate}
\end{enumerate}
\begin{corollary}\label{invariance2}\emph{$L_2$ is idempotent, i.e.}
\begin{equation}
L_2\circ L_2= L_2. 
\end{equation}
\end{corollary}
\begin{proof}
It follows directly from the definition of $L_2$ and Equations \ref{asso} and \ref{unit}.
\end{proof}

\begin{remark}

\begin{itemize}\emph{The following is an interpretation of the axioms of the relational symplectic groupoid:   
\item The cyclicity axiom  (A.1) encodes the cyclic behavior of the multiplication and inversion maps for groups, namely, if $a,b,c$ are elements of a group $G$ with unit $e$ such that $abc=e$, then $ab=c^{-1}, \, bc= a^{-1}, ca=b^ {-1}$.
\item (A.2) encodes the involutivity property of the inversion map of a group, i.e. $(g^{-1})^{-1}=g, \forall g \in G$.
\item (A.3) encodes the compatibility between multiplication and inversion:
$$(ab)^{-1}=b^{-1}a^{-1}, \forall a,b \in G.$$
\item (A.4) encodes the associativity of the product: $a(bc)=(ab)c, \forall a,b,c \in G$.
\item (A.5) encodes the property of the unit of a group of being idempotent: $ee=e$.
\item The axiom (A.6)  states an important difference between the construction of relational symplectic groupoids and usual groupoids. The compatibility between the multiplication and the unit is defined up to an equivalence relation, denoted by $L_2$, whereas for groupoids such compatibility is strict (see Axiom A.3 in Definition \ref{groupoid}); more precisely, for groupoids such equivalence relation is the identity. In addition, the multiplication and the unit are equivalent with respect to $L_2$.
}
\end{itemize}
\emph{This description explains why the choice of the axioms of the relational symplectic groupoid are \emph{natural}.}
\end{remark}

\begin{remark} \label{counter1}
\emph{Equations \ref{unit}, \ref{invariance1}, \ref{invariance2}, \ref{invariance3} and \ref{symmetric} have to be stated as part of the axioms and they cannot be deduced as corollaries. Here there is an example of a structure that satisfies  the axioms from A.1. to A.4 but not A.5 or A.6.
\begin{enumerate}
\item $\mathcal G = \mathbb Z$ (as a non connected zero dimensional symplectic manifold)
\item $L=\{(n,m,-n-m-1) \in \mathbb Z ^3\}$
\item $I\colon  n \mapsto -n$
\end{enumerate}
For this example, the spaces $L_i$ are given by
\begin{eqnarray}
L_1&=& \{ 1\}\\
L_2&=& \{ (m,m+2)\mid m\in \mathbb Z\}\\
L_3&=&\{(m,n,m+n+1)\mid n,m \in \mathbb Z ^2 \}
\end{eqnarray}
for which we get that
\begin{eqnarray}
L_3\circ(L_1 \times L_1)&=&\{3\} \neq L_1\\
L_2\circ L_1&=&\{3\} \neq L_1\\
L_2\circ L_2&=&\{(m,m+4)\mid m\in \mathbb Z\} \neq L_2\\
L_2\circ L_3&=&\{(m,n,m+m+3)\mid m,n \in \mathbb Z\} \neq L_3\\
\overline{I_{rel}}\circ L_2&=&(m, -m-2) \neq (m,-m+2)= \overline{L_2} \circ \overline {I_{rel}}.
\end{eqnarray}
This counterexample has also a finite set version, replacing $\mathbb Z$ by $\mathbb Z / k \mathbb Z$, with $k \geq 3$.}

\end{remark}

 
\end{itemize}


\begin{figure}
\psfrag{Rq}{$\mathbb{R}^q$}
\centering%
\center{\includegraphics[scale=0.62]{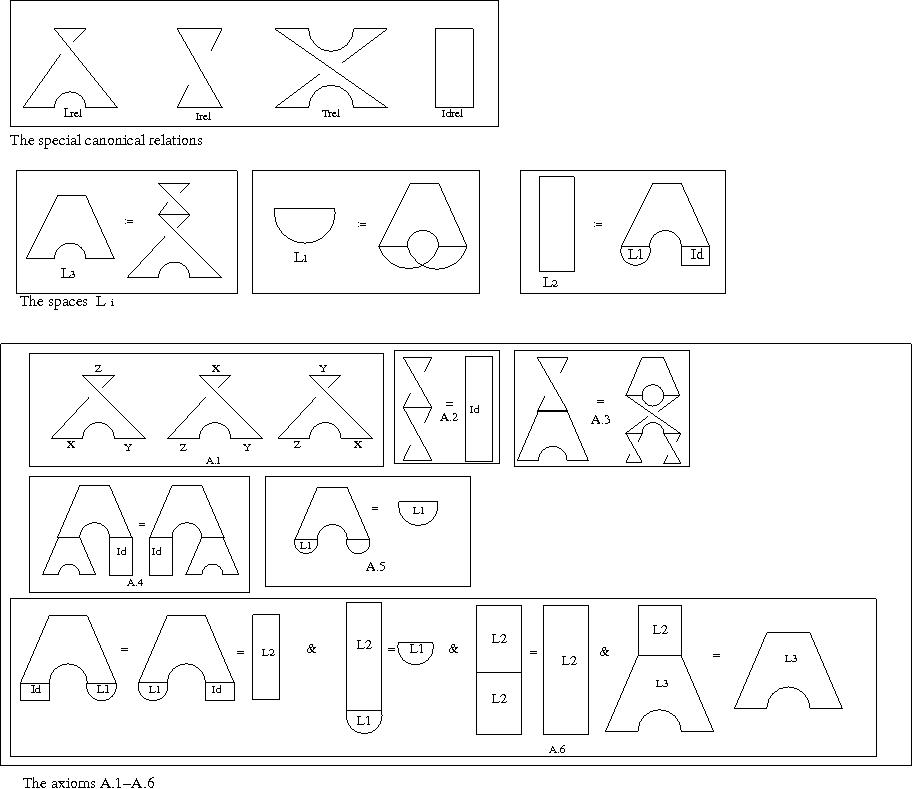}}
\caption{\underline{Relational symplectic groupoid: Diagrammatics.} The morphisms 
$I_{rel}$ and $L_{rel}$ are represented by a twisted stripe and pair of \emph{paper doll pants} respectively, 
and the induced immersed canonical relations $L_1, L_2$ and $L_3$ are constructed as compositions of $L$ and $I$. 
As it is shown in the Figure, they should satisfy  the previously defined compatibility axioms. The horizontal segments in the boundary 
of the surfaces represent the weak symplectic manifold $\mathcal G$. the non horizontal segments have no meaning.}
\label{fig:Ax}
\end{figure}

\newpage The next set of axioms defines a particular type of relational symplectic groupoids  which will allow us to relate the construction of relational symplectic groupoids for Poisson manifolds to the usual symplectic groupoids for the integrable case. Before this, we introduce the notion of immersed cosiotropic submanifolds for weak symplectic manifolds.
\begin{definition}\emph{ An immersed coisotropic submanifold of a weak symplectic manifold $M$ is a pair $(\phi, C)$ such that
\begin{enumerate}
 \item $C$ is a smooth Banach manifold.
 \item $\phi\colon C \to M$ is an immersion.
\item $T\phi_x$ applied to $T_xC$ is a coisotropic subspace of $T_{\phi(x)}M,\, \forall x \in C$.
\item  $C$ is a 
covering of $\mathfrak{Im}(\phi)$ (see Definition \ref{Extended}).
\end{enumerate}}
\end{definition}
\begin{definition}\emph{
A relational symplectic groupoid $(\mathcal G,\, L,\, I)$ is called \textbf{regular} if the following three axioms A.7, A.8 and A.9 are satisfied. Consider $\mathcal G$ 
as a relation $* \nrightarrow \mathcal G$ denoted by $\mathcal G_{rel}$.}
\end{definition}
\begin{itemize}
 \item \textbf{\underline{A.7}} 
 \begin{equation} \label{C}
 C:=L_2 \circ \mathcal G_{rel}
 \end{equation} 
 is an immersed submanifold of $\mathcal G$. 
\end{itemize}
\begin{corollary}\emph{
$C$ is an immersed coisotropic submanifold of $\mathcal G$.}
\end{corollary}
\begin{proof}
By Equation \ref{invariance1} we conclude that $L_1\subset C$, hence $C$ is coisotropic.
\end{proof}
\begin{corollary} \label{equi}\emph{
$L_2$ is an equivalence relation in $C$}.
\end{corollary}
\begin{proof}
By Equation \ref{invariance2}
\begin{equation}
L_2=L_2\circ L_2 \subset L_2\circ (\mathcal G\times \mathcal G)= C\times C\colon  *\nrightarrow \mathcal G \times \mathcal G,
\end{equation}
so $L_2$ is a relation on $C$. By Equation \ref{invariance2}, $L_2$ is transitive, by Equation \ref{symmetric} it is symmetric and, for any $x \in C$, by definition, there exists $y$ such that $(x,y)\in L_2$ and by symmetry and transitivity of $L_2$, we conclude that $(x,x)\in L_2$, hence, $L_2$ is an equivalence relation.
\end{proof}
The following Proposition allows us (in principle at the infinitesimal level), to regard the equivalence relation given by $L_2$ as the equivalence relation given by the characteristic 
foliation of $C$.
\begin{proposition}\emph{
Let 
$$R^C:=\{(x,y)\in C\times C \mid L_x=L_y\},$$
where $L_x$ is the leaf of the characteristic foliation through the point $x \in C$. Let $(x,y)\in R^C\cap L_2$.
Then
\begin{equation}\label{char}
T_{(x,x)}R^C=T_{(x,x)}L_2.
\end{equation}}
\end{proposition}
\begin{proof}
First we will prove that $T_{(x,x)}R^C\subset T_{(x,x)}L_2$. For this, consider $(X,Y)\in T_{(x,x)}R^C$, since $X-Y \in T_xC^{\perp}$, we get that 
\begin{equation}\label{E1}
(X-Y,0) \in T_{(x,x)}C^{\perp}\oplus TC^{\perp}.
\end{equation}
Since $L_2\subset C\times C$ and $L_2$ is Lagrangian
\begin{equation}\label{E2}
T_{(x,x)}C^{\perp}\oplus T_{(x,x)}C^{\perp} \subset T_{(x,x)}L_2\subset T_xC \oplus T_xC. 
\end{equation} 
Combining Equations \ref{E1} and \ref{E2}, we get that
\begin{equation}\label{E3}
(X-Y,0)\in T_{(x,x)}L_2.
\end{equation}
Since $\triangle C \subset L_2$ (from Corollary \ref{equi}), then 
\begin{equation}\label{E4}
(Y,Y)\in T_{(x,x)}L_2, \forall Y \in C.
\end{equation}
From equations \ref{E1} and \ref{E4}, we conclude that 
\begin{equation}
(X,Y)=(X-Y,0)+(Y,Y)\in T_{(x,x)}L_2,
\end{equation}
as we wanted.
Now we prove that $T_{(x,x)}R^C$ is a Lagrangian subspace of $T_xC\oplus T_xC$. For this, first observe that 
\begin{equation}
T_xC^{\perp} \oplus T_x^{\perp}C \subset T_{(x,x)}R^C \subset T_xC \oplus T_xC
\end{equation}
and that the canonical projection of $T_{(x,x)}R^C$ in the symplectic reduction $\underline C \oplus \underline C$ is $\triangle C$, that is Lagrangian. Then we can apply the result in Proposition \ref{Coisotropic} and conclude that $T_{(x,x)}R^C$ is Lagrangian. Now, since $T_{(x,x)}L_2$ is also Lagrangian by the axioms above and it contains $T_{(x,x)}R^C$ as a subspace, it follows that 
\begin{equation}
T_{(x,x)}L_2=T_{(x,x)}L_2^{\perp}\subset T_{(x,x)}R^C=T_{(x,x)}R^C,
\end{equation}
hence $T_{(x,x)}L_2=T_{(x,x)}R^C$, as we wanted.
\end{proof}


\begin{itemize}
\item \textbf{\underline{A.8}} The partial reduction $\underline{L_1}= L_1 / (L_2\cap L_1\times L_1)$ is a finite dimensional smooth manifold. 
We will denote $\underline{L_1}$ by $M$.
\item \textbf{\underline{A.9}} $S:= \{(c,[l]) \in C \times M: \exists l \in [l], g \in \mathcal G \vert (l,c,g) \in L_3\}$ is an immersed submanifold of $\mathcal G \times M $ satisfying that
\begin{enumerate}
\item
\begin{equation} \label{Source}
(S\times S)\circ L_2^{rel}= \triangle_M,
\end{equation}
where $L_2^{rel}: pt \nrightarrow C\times C$ is the induced relation from $L_2$.
\item 
The induced relation
\begin{equation}\label{subm}
dS:= TS: T\mathcal G \nrightarrow TM
\end{equation}
is surjective.
\end{enumerate}
It is easy to check that the first condition implies the following 
\begin{corollary}\label{source}\emph{
\begin{enumerate}
\item The relation \[T:= \{(c,[l]) \in C \times M\colon  \exists l \in [l], g \in \mathcal G \vert (c,l,g) \in L_3\}=I\circ S\] is an immersed submanifold of $\mathcal G \times M$.
\item $S$ and $T$ regarded as relations from $C$ to $M$ are surjective submersions.
\end{enumerate}
}
\end{corollary}
\begin{proof}
(1) follows from the cyclicity condition in A.1.  (2) follows by definition of $T$ and from the fact that, since Equation \ref{Source} holds, then if
$(x,[l])$ and $(x, [l]^{'})$ belong to $S\colon  C\to M$, then $([l], [l]^{'})\in \triangle_M$, which implies that $S$ is a surjective map and is clearly a submersion by Axiom 9, part 2.
\end{proof}
\begin{remark}\emph{
The condition  given by Equation \ref{char} is at the level of tangent spaces. If  we want that $R^C=L_2$, we should impose a connectedness condintion on the leaves of the characteristic foliation and the classes of $L_2$. The following is a modification of the example given in Remark \ref{counter1} of a structure satisfying all the axioms except the global version of Equation \ref{char}.
\begin{enumerate}
\item $\mathcal G = \mathbb Z$ 
\item $L=\{(n,m,-n-m-2k-1)\mid (m,n,k) \in \mathbb Z ^3\}$
\item $I: n \mapsto -n$
\end{enumerate}
For this example, the spaces $L_i$ and $C$ are given by}
\begin{eqnarray}
L_1&=& \{ 2\mathbb Z +1\}\\
L_2&=& \{ (m,n)\mid m-n \in 2\mathbb Z\}\\
L_3&=&\{(m,n,m+n+2k+1)\mid n,m,k  \in \mathbb Z ^3 \}\\
C&=& \mathbb Z.
\end{eqnarray}
\emph{Since $\mathbb Z$ is zero dimensional, we get that $C$ is also Lagrangian and for the symplectic reduction, $\underline{C}=*$. In the other hand,}
$$C / L_2= \mathbb Z/ 2 \mathbb Z \neq *.$$
\end{remark}
The following theorem connects the construction of relational symplectic groupoids in the regular case with the usual symplectic groupoids.
\begin{theorem}\label{TheoSym}\emph{ Let $(G,L,I)$ be a regular relational symplectic groupoid. Then $G:= C/L_2\rightrightarrows M$ is a topological groupoid over $M$. Moreover, if $G$ is a smooth manifold, then $G\rightrightarrows M$ is a symplectic groupoid over $M:= L_1/L_2$.
}
\end{theorem}
\begin{proof}
The following are the data that define the groupoid structure, if 
\begin{eqnarray*}
p\colon  \mathcal G &\to& G\\
g&\mapsto& [g]
\end{eqnarray*}
denotes the canonical projection with respect to the symplectic reduction of $C$.
\begin{eqnarray*}
G_0&=& L_1/L_2 \\
G_1&=& C/L_2\\
G_2&=&C/L_2\times_{L_1/L_2} C/L_2
\\
m&=& p^3(L_3)\colon  G_2 \to G_1\\
\varepsilon&:& G_0 \to G_1:= \underline{\varepsilon}:L_1/L_2 \to C / L_2 \\
s&:&G_1 \to G_0:=\underline s\colon  C/L_2 \to M  \\
t&:&G_1 \to G_0:= \underline t\colon  C/L_2 \to M  \\
\iota&:& G_1 \to G_1:= \underline {I}\colon  C/L_2 \to C/L_2.
\end{eqnarray*}
Under the smoothness assumption for $\underline C$ and also assuming that the characteristic foliation of $C$ has finite codimension, $G_1$ is a finite dimensional symplectic manifold and due to Corollary \ref{source}, the map $\underline s$ is a surjective submersion, hence, the fiber product $G_2$ is a  finite dimensional (topological) manifold. It is easy to check that the groupoid axioms are satisfied.
For the symplectic structure on $G\rightrightarrows M$, note that the projection of $L_3$ in $G$ is Lagrangian (see Definition \ref{canonicalproj}) and restricted to $G_2$ is a map (due to Corollary \ref{source}). 

\end{proof}
\subsection{Poisson structure on $M$}
In this section, the objective is to relate the construction of the relational symplectic groupoids in the regular case  with 
Poisson structures in the space $M$. More precisely, we prove the existence and uniqueness of a Poisson bracket on 
$M$ compatible with a given regular relational symplectic groupoid $\mathcal G$. 
This theorem is the analog of the existence and uniqueness of a Poisson structure in the space of objects of usual 
symplectic groupoids \cite{Weinstein}. Namely,\begin{theorem}\label{Poisson base}
\emph{ \cite{Weinstein}. Let ${(G,\omega) \rightrightarrows M}$ be a symplectic groupoid over $M$ . Then there exists a unique Poisson structure $\Pi$ on $M$ such that the source map $s$ is a Poisson map (or equivalently the target map $t$ is an anti-Poisson map).}
\end{theorem}
One possible way to prove this theorem is by the use of what is known in the literature as \textit{Liberman's lemma}, that is stated in a slightly different formulation by Paulette Liberman in \cite{Liberman}. 
Before stating the result, we need some definition that will be used in the sequel.
\begin{definition}
\emph{Let $(G, \omega)$ be a symplectic manifold and $\mathcal F$ a foliation on $G$. $\mathcal F$ is called a \emph{symplectically complete foliation} if the symplectically orthogonal distribution $(T \mathcal F)^{\perp}$ is an integrable distribution.
}
\end{definition}
This is equivalent to say that $\mathcal F$ is symplectically complete if there exists another foliation 
$\mathcal F^{'}$ such that $$T_x(\mathcal F)= (T_x(\mathcal F^{'}))^{\perp}.$$

After this definition, Liberman's lemma reads as follows

\begin{lemma}\label{Liberman} \emph{(Liberman). 
Let $\phi\colon  (G,\omega) \to M$ be a surjective submersion from a symplectic manifold $G$ to a manifold $M$ such that the fibers are connected. Denote by $(G,\mathcal F)$ the foliation on $G$ induced by the fibers of $\phi$.   Then, there exists a unique Poisson structure $\Pi$ on $M$ such that the map $\phi\colon  G \to M$ is a Poisson map if and only if the foliation $\mathcal F$ is symplectically complete.}
\end{lemma}
\begin{proof}
Here we present a sketch of the proof. It can be checked that the distribution $(T\mathcal F)^{\perp}$ is Hamiltonian, 
i.e. is generated by the vector fields $X_{\phi^*f}$, for which $\omega(X_{\phi^*f}, \cdot)= d \phi^*f$, 
with $f \in \mathcal C^{\infty}(M)$. The fact that such distribution is integrable is equivalent, due to Frobenius Theorem, 
that $[X_{\phi^*f}, X_{\phi^*g}]$, with $f,g \in \mathcal C^{\infty}(M)$, is tangent to the distribution $(T\mathcal F)^{\perp}$. 
This is also equivalent to say that $X_{\{\phi^*f, \phi^*g\}}$ is tangent to 
$(T\mathcal F)^{\perp}$, therefore the bracket $\{\phi^*f, \phi^*g \}$ is constant along the $\phi-$ fibers and this implies that 
there exists a function $h\in\mathcal C^{\infty}(M)$ such that $\{\phi^*f, \phi^*g\}= \phi^*h$. This functions defines uniquely the Poisson bracket on 
$M$ given by 
$$\{f,g\}_{\Pi}=h.$$
\end{proof}

By applying Lemma \ref{Liberman} to the case of a symplectic groupoid $G \rightrightarrows M$ with $\mathcal F$ being the foliation described by the 
distribution $\ker(ds)$, Theorem \ref{Poisson base} holds.
The generalization of this result in the case of regular relational symplectic groupoids is the following Theorem, in which we have to introduce the language of Dirac structures 
(for more details see Appendix \ref{Dirac})
\begin{theorem}\label{Theo1Poisson1}\emph{ Let $(\mathcal G, L, I)$ be a regular relational symplectic groupoid, with $M= L_1/ L_2$.  Then, assuming that the $s$-fibers are connected, there exists a unique Poisson structure $\Pi$ on $M$ such that the map $s\colon  C\to M$ (or $t$) is a forward-Dirac map.
}
\end{theorem}
\begin{proof}
For this proof we present a more general version of Liberman's lemma for the case when $G$ is presymplectic. 
\begin{definition}\emph{ Let $M$ be a Banach manifold. A $2-$ form $\omega \in \Omega^2(M)$ is called \emph{presymplectic} if $d\omega=0$. In this case $M$ is called a presymplectic manifold. 
}
\end{definition}
\begin{lemma}
\emph{(Liberman's lemma for presymplectic manifolds). Let $G$ be a presymplectic manifold and 
$s, t : G \to M$ be smooth surjective submersions of $G$ onto a smooth manifold $M$ such that the fibers $s^{-1}(x)$ and $t^{-1}(x)$ are presymplectic orthogonal, for all $x \in M$. 
Then, if the $s$-fibers are connected, there exists a unique Poisson structure $\Pi$ on $M$ such that the map $s$ is a forward-Dirac map.}
\end{lemma} 
\begin{proof}
Let $\omega$ be the presymplectic structure on $G$ and $\omega^{\sharp}\colon  TG \to T^*G$ its induced bundle map. Now consider the functions $f \in \mathcal C^{\infty}(G)$ such that $df \in \Gamma(\mathfrak {Im}(\omega^{\sharp}))$. Then, for these functions, there exist uniquely determined Hamiltonian vector fields $X_f$, i.e. 
$$\omega(X_f, \cdot)= df$$
Now, let $W$ be a subspace of $\Gamma(TG)$ and we define pointwise
$$W_x^{\omega}:= \{v \in T_xG \mid \omega(v,w)=0, \forall w \in W_x \}$$
and we also define
$$W^0= \{\xi \in T^*G \mid \xi(V)=0, \forall V \in T_xG.\}$$
\end{proof}
Now, let $W= \ker(t_*) \subset T_xG$ and $Y \in W$. We have that 
$$\xi:= \omega(Y,\cdot) \in \frak{Im}(\omega^{\sharp}) \cap (\ker(t_*))^0$$
Since the $t-$ fibers are connected, we could write $\xi$ as
\begin{equation}
\xi= \sum_i \alpha_i dh_i,
\end{equation}
with $h_i \in \mathcal C^{\infty}(M)$.
Now, given two functions $f$ and $g$ in $\mathcal C^{\infty}(M)$, we construct the bracket $\{f,g\}_{M}$ as follows. Since 
$$Y\{s^*f, s^*g\}_G=0,\, \forall Y \in \Gamma(\ker (s_*)),$$
then, from the discussion above, there exists $\alpha \in \Gamma(\ker t_*)^0 \cap \frak{Im}(\omega^{\sharp})$ 
such that $Y= (\omega^{\sharp})^{-1}(\alpha)$. Since $Y$ can be written as
$$ Y= \sum_i X_{t^*h_i}, \, h_i \in \mathcal C^{\infty}(M),$$ we can define
\begin{equation}
\{ f,g \}_{M}:=h= \sum_i \alpha_i h_i.
\end{equation}
The fact that $s$ is a forward-Dirac map with respect to $\{,\}_M$ is equivalent to the following equation
\begin{equation}
\{t^*h, \{s^*f, s^*g \}\}_{G}=0
\end{equation}
that it holds since $\{,\}_G$ satisfies the Jacobi identity.
\end{proof}


\end{itemize}
Conjecturally, the proof of this theorem could be adapted in order to drop the connectedness assumption of the $s$-fiber. Then, we have the following
\begin{conjecture}\label{Theo1PoissonC}\emph{
Theorem \ref{Theo1Poisson1} holds also when the $s$-fibers are not connected.}
\end{conjecture}





\subsection{The categroid RSGpd}
We have defined so far relational symplectic groupoids in the extended symplectic category and we have related this construction with the usual notion of symplectic groupoids. These objects have a natural notion of morphism that is also defined in the context of canonical relations. Hence, as before, the composition of morphisms is only partially defined but it allows us to describe the categroid  \textbf{RSGpd} of relational symplectic groupoids with suitable morphisms.
\begin{definition} \emph{
A morphism between two relational symplectic  groupoids $(\mathcal G,L_{\mathcal G},I_{\mathcal G})$ 
and $(\mathcal H,L_{\mathcal H},I_{\mathcal H})$ 
is a relation $F\colon  G\nrightarrow H$ satisfying the following properties:
\begin{enumerate}
 \item $F$ is an immersed Lagrangian submanifold of $\mathcal G\times \bar{\mathcal H}$.
 \item $F\circ I_{\mathcal G}=I_{\mathcal H} \circ F$. 
 \item $L_{\mathcal H}\circ (F\times F)=F\circ L_{\mathcal G}.$
\end{enumerate}
}
\end{definition}
\begin{definition}\emph{
A morphism of relational symplectic groupoids $F\colon  \mathcal G \to \mathcal H $ is called an \textbf{equivalence} 
if the transpose canonical relation $F^{\dagger}$ is also a morphism.}
\end{definition}
\begin{remark}\label{equiv}\emph{From the definition, it follows that an equivalence $F$ satisfies the following compatibility conditions with respect to $L_1$ and $L_2$:
\begin{eqnarray}
F\circ (L_1)_{\mathcal G}&=& (L_1)_{\mathcal H}\\
F\circ (L_1)_{\mathcal H}&=& (L_1)_{\mathcal G}\\
F^{\dagger}\circ F&=& (L_2)_{\mathcal G} \label{1}\\
F\circ F^{\dagger}&=& (L_2)_{\mathcal H}. \label{2}.
\end{eqnarray}
This implies that, in the case where $\mathcal G$ and $\mathcal H$ are both regular, the equivalence $F$ induces relations 
$F_M\colon M_{\mathcal G} \nrightarrow M_{\mathcal H}$ and $F_M^{\dagger}\colon M_{\mathcal H} \nrightarrow M_{\mathcal G}$ satisfying that
\begin{eqnarray*}
F_M^{\dagger}\circ F_M&=& \Id_{M_{\mathcal G}}\\ 
F_M\circ F_M^{\dagger}&=& \Id_{M_{\mathcal H}}.
\end{eqnarray*}
Therefore, the induced relation $F_M$ is the graph of a diffeomorphism between $M_{\mathcal G}$ and $M_{\mathcal{H}}$, and since the following diagram commutes:
\[
\xymatrixrowsep{2pc} \xymatrixcolsep{2pc} \xymatrix{C_{\mathcal G} \ar@{->}[r]^{F} |-{/}\ar@{->}[d]_{s}&C_{\mathcal H}\ar@{->}[d]_{s}\\M_{\mathcal G}\ar@{->}[r]^{F_M} &
M_{\mathcal H}}
\]
from Theorem \ref{Theo1Poisson1} it follows that the map $F_M$ is a Poisson diffeomorphism.
} \end{remark}

The following are some examples of equivalences.
\begin{example}\emph{Let $(\mathcal G, L, I)$ be a relational symplectic groupoid. Then $L_2$ is an equivalence between $\mathcal G$ and itself. \\
To check that $L_2$ is a morphism of relational symplectic groupoids, we observe that, by Equation \ref{symmetric} $L_2$ commutes with $I$ and by Equation \ref{invariance3} we get that 
\begin{eqnarray*}
L_{\mathcal G}\circ (L_2 \times L_2)&=& I\circ L_3\circ (L_2\times L_2)\\
                                                                &=& I \circ L_3= I \circ L_2 \circ L_3\\
                                                                &=& L_2 \circ I \circ L_3= L_2 \circ L_{\mathcal G}. 
                                                                \end{eqnarray*} 
Since $L_2$ is self transposed, it follows that $L_2$ is an equivalence.                                                                
}
\end{example}
\begin{example}\emph{
For a relational symplectic groupoid $(\mathcal G, L, I)$ the map $I$ is an equivalence from $(\mathcal G, L,I)$ and $(\overline {\mathcal G}, L^{'}, I)$, where $L^{'}= L \circ T_{rel}$}.
\end{example}
In the next section we give additional examples of equivalences, after giving some natural cases of relational symplectic groupoids.




\section{Examples of relational symplectic groupoids}

 


\subsection {Symplectic groupoids} \label{groupo}
Given a symplectic groupoid $G$ over $M$, we can endow it naturally with a relational symplectic structure:
\begin{eqnarray*}
\mathcal{G}&=&G.\\
L&=&\{(g_1,g_2,g_3) \vert (g_1,g_2) \in G\times_{(s,t)}G,\, g_3^{-1}=\mu(g_1,g_2)\}.\\
I &=& g \mapsto \iota(g),\, g \in G.
\end{eqnarray*}
In this case, it is an easy check that the immersed canonical relations $L_i$ are given by
\begin{eqnarray*}
L_1&=& \varepsilon(M)\\
L_2&=& \triangle(G)\\
L_3&=& Gr(\mu),
\end{eqnarray*} 
and also we observe that in this case $(\mathcal G, L, I)$ is regular.
According to Theorem \ref{TheoSym}, given a regular relational symplectic groupoid $(\mathcal G, L,I)$ which admits a smooth symplectic reduction, we can associate a usual symplectic groupoid $G\rightrightarrows M$. By definition of such groupoid, we obtain the following
\begin{proposition}\label{pro}\emph{
The projection $p\colon \mathcal G \to G$ is an equivalence of relational symplectic groupoids.}
\end{proposition}
\begin{proof}
By Proposition \ref{projection}, $p$ is a canonical immersed relation $p\colon  \mathcal G \nrightarrow G$ and by definition of the groupoid structure on $G$ (see Theorem \ref{TheoSym}), it follows that $p$ commutes with $I$ and $L$ respectively, hence, $p$ is a morphism of relational symplectic groupoids.
The fact that $p^{\dagger}$ is also a morphism follows from the following facts. By definition, $I_G$ can be written as 
\begin{equation}
I_G= p\circ I_{\mathcal G} \circ p^{\dagger}
\end{equation}
and $L_G$ can be written as
\begin{equation}
p\circ L \circ (p^{\dagger} \times p^{\dagger}).
\end{equation}
Again, by proposition \ref{projection} we have that
\begin{equation} \label{Q1}
p^{\dagger} \circ p=(L_2)_{\mathcal G}
\end{equation}
and that
\begin{equation} \label{Q2}
p\circ p^{\dagger}= Id_G.
\end{equation}
Therefore, we get the following equalities
\begin{eqnarray}
p^{\dagger}\circ I_G&=& p^{\dagger}\circ p \circ I_{\mathcal G}\circ p^{\dagger}\\
&\stackrel{\ref{Q1}}{=}& L_2
\end{eqnarray}

\end{proof}
Another finite dimensional example is the following.
\subsection{Symplectic manifolds with a Lagrangian submanifold} Let $(G,\omega)$ be a symplectic manifold, $\phi$ an antisymplectomorphism and $\mathcal L$ an immersed 
Lagrangian submanifold of $G$ such that $\phi(\mathcal L)=\mathcal L$. We define
\begin{eqnarray}
\mathcal{G}&=&G.\label{p}\\
L&=&\mathcal L \times \mathcal L\times \mathcal L.\label{q}\\
I&=& \phi\label{r}
\end{eqnarray}
It is an easy check that this construction satisfies the relational axioms and that the spaces $L_i$ are given by 

\begin{eqnarray*}
L_1&=& \mathcal L \\
L_2&=& \mathcal L \times \mathcal L\\
L_3&=&\mathcal L\times  \mathcal L\times  \mathcal L.
\end{eqnarray*}

This example is a regular relational symplectic groupoid and furthermore we can prove the following
\begin{proposition} \emph{
The previous relational symplectic groupoid is equivalent to the zero dimensional symplectic groupoid (a point with zero symplectic structure and empty relations).}
\end{proposition}
\begin{proof}We prove that $\mathcal L$ is an equivalence from the zero dimensional manifold $p$ to $\mathcal{G}$ . This comes from the fact that, for this example, $C$ (as defined in Equation \ref{C}) is precisely $\mathcal L$, hence, its symplectic reduction is just a point. By Proposition \ref{pro} it follows that $\mathcal L$, being the canonical projection, is an equivalence.
\end{proof}
\begin{remark}\emph{More generally, following the definition of equivalence of relational symplectic groupoids and remark \ref{equiv}, we can check that $(\mathcal G,L,I)$ 
is equivalent to the zero dimensional symplectic groupoid if and only if there exists a Lagrangian submanifold $\mathcal L_{eq}$ of $\mathcal G$ satisfying the following two properties
\begin{itemize}
 \item $I\circ \mathcal L_{eq}=\mathcal L_{eq}$.
\item $\mathcal L_{eq}= L\circ (\mathcal L_{eq} \times \mathcal L_{eq})$. 
\end{itemize}
This implies that the only relational symplectic groupoids that are equivalent to the zero dimensional one are the ones described by Equations \ref{p}, \ref{q} and \ref{r}.}
 
\end{remark}

\subsection{Powers of symplectic groupoids}\label{power}
The following are two (a priori) different constructions of relational symplectic groupoids for the powers of a given symplectic groupoids. Let $G\rightrightarrows M$ be a symplectic groupoid and $(G,L,I)$ its associated relational symplectic groupoid as in Example \ref{groupo}. It is easy to check that
\begin{proposition}
\emph{
$(G^n, L^n, I^n)$ is a relational symplectic groupoid, for all $n\geq 1$.}
\end{proposition}
Now, let us denote $G_{(1)}=G$, $G_{(2)}$ the fiber product $G\times_{(s,t)}G$, $G_{(3)}=G \times_{(s,t)}(G \times_{(s,t)}G)$ and so on. We will use the following
\begin{lemma} \emph{\cite{Xu}. Let $G\rightrightarrows M$ be a symplectic groupoid.
\begin{enumerate}
 \item $G_{(n)}$ is a coisotropic submanifold of $G^n$.
\item The reduced spaces $\underline{G_{(n)}}$ are symplectomorphic to $G$. Furthermore, there exists a natural symplectic groupoid structure on $\underline{G_{(n)}}\rightrightarrows M$ coming from the symplectic quotient, isomorphic to the symplectic groupoid structure on $G\rightrightarrows M$. 
\end{enumerate}
}
\end{lemma}
Having, this lemma at hand, and considering the canonical relations 
$$p_n: G^n \to \underline{G_{(n)}}\equiv G,$$ we define the regular relational symplectic groupoid $(G^{(n)}, L^{(n)}, I^{(n)})$, given by 
\begin{eqnarray*}
G^{(n)}&:=& G^n\\
I^{(n)}&:=& p_n^{\dagger} \circ I \circ p_n\\
L^{(n)}&:=& p^{\dagger} \circ L \circ (p \times p)
\end{eqnarray*}

It can be checked that this example satisfies the relational axioms and  that the diagonal $\triangle_{n+1}(G)$ is a morphism of relational symplectic groupoids.\\


\section{The main example: $T^*(PM)$}
The objective of this section is to prove the following Theorem
\begin{theorem}\label{Relational} Given a Poisson manifold $(M, \Pi)$ there exists a regular relational symplectic groupoid $(\mathcal G,L,I)$ that integrates it.
\end{theorem}
As we mentioned in the introduction, integration in this setting means the following 
\begin{enumerate}
\item Such relational symplectic groupoid satisfies that $L_1/ L_2=M$ and the symplectic structure on $\mathcal G$ is compatible with the Poisson structure on $M$ according to Theorem \ref{Theo1Poisson1}
\item In the case that the Lia algebroid $T^*M$ is integrable, such relational symplectic groupoid is equivalent to a symplectic groupoid integrating it.
\end{enumerate}
The structure of the proof of this Theorem is as follows. First, we describe the defining data for the relational symplectic groupoid in terms of the PSM and $A-$ homotopy for Lie algebroids specialized in the Poisson case. Then we verify that such data in fact satisfy the relational axioms. In order to do this, we need to prove the smoothness and Lagrangianity of the canonical relations $L_i$, which deserves special attention since we are dealing with infinite dimensional spaces.\\

\textbf{Proof of Theorem \ref{Relational}}
We will prove that the relational symplectic groupoid $(\mathcal G, L, I)$ associated to $(M, \Pi)$ is given by 
\begin{enumerate}
\item $\mathcal G := T^*(PM)$, the cotangent bundle of the path space of $M$.
\item $L=\{(\gamma_1, \gamma_2, \gamma_3) \in (T^*(PM))\}$ is such that
\begin{itemize}
\item $\gamma_i$, with $1 \leq i \leq 3$ are $T^*M$-paths.
\item The concatentation $\gamma_1 * \gamma_2$ is $T^*M$- homotopic to the inverse path $\gamma_3^{-1}$, or equivalently, $\gamma_1*\gamma_2*\gamma_3$ is $T^*M$-homotopic to a constant path.
\end{itemize}
\item \begin{eqnarray*}
I\colon  T^* PM &\to& T^* PM\\
\gamma &\mapsto& \gamma^{-1}.
\end{eqnarray*}

\end{enumerate}

First, we describe the defining spaces $L_i$ of the relational symplectic groupoid 
set theorically, proving that they satisfy the algebraic relational axioms and then we prove that they are in fact immersed canonical relations.\\

\textbf{\underline{A.1.}} To prove the cyclicity property, we use the following remark, that is easy to check.
\begin{remark}
\emph{ Let $\gamma_1, \, \gamma_2, \, \gamma_1^{'}$ and $\gamma_2^{'}$ be $T^*M$- paths such that $\gamma_1 \sim \gamma_1^{'}$ and $\gamma_2 \sim \gamma_2^{'}$, where $\sim$ denotes the equivalence by $T^*M$- homotopy. Then 
$$\gamma_1 * \gamma_2 \sim \gamma_1^{'} * \gamma_2^{'}.$$ 
}
\end{remark}
Now, consider $(x,y,z) \in L$. Since $x * y \sim z^{-1}$, we get that
\begin{equation*}
x * y \sim z^{-1}\Leftrightarrow (x*y)*y^{-1} \sim z^{-1} * y^{-1} \Leftrightarrow z*(x*y)*y^{-1} \sim z*z^{-1} * y^{-1}\Leftrightarrow z*x\sim y^{-1},
\end{equation*}
hence, $(z,x,y)$ (and similarly $(y,z,x)$) belongs to $L$.
\qed
\\

\textbf{\underline{A.2.}} 
If we define
\begin{eqnarray}\label{phi}
\phi\colon  [0,1] &\to& [0,1]\\
t&\mapsto& 1-t
\end{eqnarray}
Then we get that
\begin{eqnarray*}
I\colon  T^*(PM) &\to& T^*(PM)\\
\gamma &\mapsto& \phi^*\circ \gamma,
\end{eqnarray*}
hence,
$$I_{*}\delta \gamma =I_{*}(\delta X, \delta \eta)=\delta X(\phi(t)), -\delta \eta(t))$$
and therefore, using Equation \ref{symplectic},
\begin{equation*}
I^*\omega_{\gamma}(\delta_1\gamma, \delta_2\gamma)= -\int_0^1\delta_1X^i(t)\delta_2\eta_i(t)-\delta_2X^i(t)\delta_1\eta_i(t)dt=-\omega_{\gamma}(\delta_1\gamma, \delta_2\gamma)
\end{equation*}
and this proves that $I$ is an anti-symplectomorphism.
\qed
\\

\textbf{\underline{A.3.}} First, we observe that, from the definition,
\begin{equation}\label{ELE3}L_3=\{(\gamma_1, \gamma_2, \gamma_3)\in T^*(PM)^3 \mid \gamma_1* \gamma_2 \sim \gamma_3 \}.
\end{equation}
In Subsection \ref{smoothness} we will prove that $L_3$ is an immersed canonical relation.
\\

\textbf{\underline{A.4.}} We have that 
\begin{eqnarray*}
L_3\circ (L_3\times Id)&=& \{(\gamma_1, \gamma_2,\gamma_3, \gamma_4) \in (T^*(PM))^3 \mid \exists (\gamma_5, \gamma_6)\in T^*(PM)^2\\
 &\mid& (\gamma_1, \gamma_2, \gamma_5)\in L_3, (\gamma_3, \gamma_6) \in Id, (\gamma_5, \gamma_6, \gamma_4) \in L_3.\}
\end{eqnarray*}
Given the restrictions
\begin{eqnarray*}
\gamma_3&=& \gamma_6\\
\gamma_5&\sim& \gamma_1 *\gamma_2\\
\gamma_5* \gamma_3 &\sim& \gamma_4,  
\end{eqnarray*}
which implies that
$$L_3\circ (L_3\times Id)=\{(\gamma_1, \gamma_2, \gamma_3) \mid (\gamma_1*\gamma_2)*\gamma_3 \sim \gamma4\}$$
and since $(\gamma_1*\gamma_2)*\gamma_3 \sim \gamma_1* (\gamma_2* \gamma_3)$ we get that $L_3\circ (L_3\times Id)=L_3\circ (Id\times L_3)$, as we wanted.
\qed
\\

\textbf{\underline{A.5.}} From the definition, we get that
\begin{eqnarray}\label{l1}
L_1&=& \{\gamma \in T^*(PM) \mid \exists \alpha \in T^*PM , \gamma \sim \alpha*\alpha^{-1} \sim \alpha^{-1}*\alpha \}\\
&=&\{ \gamma \in T^*(PM)\mid \gamma \sim (X\equiv x_0, \eta \equiv  0)\}.
\end{eqnarray}
\\

\textbf{\underline{A.6.}}
For the case of $L_2$ it follows from the definition, that
\begin{equation}
L_2=\{T^*M\mbox{-paths } (\gamma_1, \gamma_2) \in T^*(PM)^2 \mid \gamma_1 \sim \gamma_2 \}.
\end{equation}
The smoothness for $L_1$ and $L_2$ will be proved in Section \ref{smoothness}.
$\qed$\\

Assuming Theorem \ref{Relational}, it is possible to prove the following

\begin{proposition}
\emph{The relational symplectic groupoid $(G,L,I)$ is regular.}
\end{proposition}
\begin{proof}
It is easy to observe for $C= C_{\Pi}$, the space of $T^*M$-paths, that by Proposition \ref{Coi}, $C$ is a Banach submanifold of finite codimension, therefore, axiom \textbf{\underline{A.7.}} holds.
To check \textbf{\underline{A.8.}}, observe that
\begin{equation*}
\underline {L_1}= L_1/ L_2= \{(X,\eta)\in T^*(PM)\mid \exists x_0\in M: (X\equiv x_0, \eta\equiv 0) \}\cong M.
\end{equation*}
We can define the map
\begin{eqnarray*}
s\colon  C &\to& M\\
\gamma=(X,\eta)&\mapsto& X(0) 
\end{eqnarray*}
It follows that $S$, defined in \textbf{\underline{A.9}} corresponds to Graph($s$). The following Lemmata ensure the fact that $dS$ is surjective.
\begin{lemma} \label{mani}\emph{
Let $X$ be a metric space and $PX$ the space of continuous maps from $I$ to $X$. We define the evaluation map
\begin{eqnarray*}
ev_t\colon  PX &\to& X\\
\gamma&\mapsto& \gamma(t).
\end{eqnarray*} 
Then $ev_t$ is a continuous map, provided that $PX$ is equipped with the uniform convergence topology.
}
\end{lemma}
\begin{proof} We fix a path $\gamma \in PX$, a time $t \in T$ and $\varepsilon \in \mathbb R^{>0}$. Consider an open ball $\mathcal U_{\varepsilon}(ev_t(\gamma))$, centered at  $ev_t(\gamma)$ with radius $\varepsilon$. Let $\mathcal V(\gamma):= ev_t^{-1}(U_{\varepsilon}(ev_t(\gamma)))$ and let $\tilde{\gamma} \in \mathcal V(\gamma)$.  The open neighborhood of $\tilde{\gamma}$ defined by 
$$ \mathcal V_{\varepsilon / 2}(\tilde{\gamma}):= \{\xi \in PX \mid d(\tilde{\gamma}, \xi)< \varepsilon/ 2 \}$$
is contained in $V(\gamma)$, therefore 
$$V(\gamma)= \bigcup_{\tilde{\gamma}\in V(\gamma)}V_{\varepsilon / 2}(\tilde{\gamma}), $$
hence, $\mathcal V(\gamma)$ is an open in $PX$, which implies that $ev_t$ is continuous.
\end{proof}
Setting $X=M$, where $M$ is our given smooth manifold, this Lemma proves that the map $s: C\to M$ is continuous, where $C$ is equipped with the subspace topology. This implies that 
Graph$(s)$ is a submanifold of $C\times M$.
To check that it corresponds to a submersion, we will prove the following
\begin{lemma}\emph{ The differential $\delta s$ of the map $s\colon C\to M$ is a well defined surjective map from $TC$ to $TM$}
\end{lemma}
\begin{proof}
Let $\gamma =(X,\eta)\in C$. A vector $\delta \gamma \in T_{\gamma}C$ is described by
$$\delta \gamma = \{ (\delta X, \delta \eta) \mid \delta X \in \Gamma (X^* TM), \, \delta \eta \in \Gamma (X^*T^*M)\}.$$
The map $\delta s$ corresponds to 
\begin{eqnarray*}
\delta s\colon  TC &\to& TM\\
\delta \gamma &\mapsto& \delta X (0),
\end{eqnarray*}
that is the evaluation of $\delta X$ at 0, which is a well defined surjective map, as we wanted.
\end{proof}
\end{proof}
The rest of the section is devoted to prove the smoothness and Lagrangianity of the spaces $L_i$ defining the relational symplectic groupoid. 
 
\subsection{Smoothness of $L_i$}\label{smoothness}
In this subsection, we develop the notion of \emph{path holonomy} for the foliated manifold $(T^*PM, \mathcal F)$, where $\mathcal F$ is the characteristic foliation associated to the submanifold $C_{\Pi}$, which has codimension $n$, where $n= \dim (M)$. Following the construction in the case of finite dimensional foliations \cite{Mor,Brown}, it is possible to give a smooth manifold structure to the holonomy and monodromy groupoids associated to $(T^*PM, \mathcal F)$. These constructions will allow us to give smoothness conditions to the defining relations $L_i$. First, we recall some basic definitions we will use throughout the proofs.\\

\subsubsection{Foliations for Banach manifolds }
\begin{definition}\label{fol}
\emph{Let $M$ be a connected Banach manifold. Let $$\mathcal F= \{ \mathcal L_{\alpha} \mid \alpha \in A\}$$ be a family of path connected subsets of $M$. Then $(M, \mathcal F)$ is a foliation of codimension $p$ if the following conditions hold:
\begin{enumerate}
\item $\mathcal L_{\alpha} \cap \mathcal L_{\beta}= \emptyset,$ for $\alpha, \beta \in A, \alpha \neq \beta.$
\item $\bigcup_{\alpha\in A} \mathcal L_{\alpha}=M.$
\item For every $x \in M$, there exists a coordinate chart $(\mathcal U_{\lambda}, \phi_{\lambda})$ for $M$ around $x$ such that for $\alpha \in A$ with $\mathcal U_{\lambda}\cap \mathcal L_{\alpha}\neq \emptyset$, each path connected component of $\phi_{\lambda}(\mathcal U_{\lambda} \cap \mathcal L_{\alpha}) \subset B \times \mathbb R ^p$, where $B$ is a Banach space,  has the form
$$(B \times \{ c \} )\cap \phi (\mathcal U_{\lambda}),$$
where $c \in \mathbb R ^p$ is determined by the path connected component $\mathcal L_{\alpha}$, called a \emph{leaf} of the foliation. If $U$ is a subset of $M$, a path component of the intersection of $U$ with a leaf is called a \emph{plaque} of $U$.
\end{enumerate}}
\end{definition}
Besides the usual finite dimensional examples of foliations, the following proposition gives us characteristic distributions as examples of foliations at the infinite dimensional level.
\begin{proposition} \emph{Let $(M, \omega)$ be a weak symplectic Banach manifold and let $C$ be a coisotropic submanifold such that $TC^ {\perp}$ has finite codimension. Then $TC^{\perp}$ induces a foliation of finite codimension of $C$.}
\end{proposition}
\begin{proof}
We will check first that the distribution $TC^ {\perp}$ is involutive, that is,
$$\omega([X,Y], Z)=0, \forall X,Y \in TC^{\perp}, Z \in TC.$$ 
We know that
\begin{eqnarray*}
d\omega (X,Y,Z)&=& \omega(X, [Y,Z])-\omega(Y, [X,Z])+ \omega (Z, [X,Y])\\
&+& X\omega(Y,Z)-Y\omega(X,Z)+ Z\omega(X,Y)\\
&=&-\omega([X,Y], Z)=0.
\end{eqnarray*}
By the use of Frobenius Theorem for Banach manifolds (for references see \cite{Lang}), this distribution is integrable and it induces a foliation on $C$ of finite codimension.
\end{proof}
In our case of interest the Banach manifold is $\mathcal G= T^*(PM)$ and $C= C_{\Pi}$. In \cite{Cat} it is proven that $C^{\perp}$ has finite codimension. 
Now, we describe the monodromy and holonomy groupoids for foliations.

 \subsubsection{Monodromy groupoid over a foliated manifold}
 
  Let $(M, \mathcal{F})$ be a foliation. The monodromy groupoid, denoted 
by Mon$(M, \mathcal F)$, has as space of objects the manifold $M$ and the space of morphisms is defined as follows:
\begin{itemize}
 \item If $x,y \in M$ belong to the same leaf in the foliation, the morphisms between $x$ and $y$ are homotopy classes, relative to the 
end points, of paths between $x$ and $y$ along the same leaf.
\item If $x$ and $y$ are not in the same leaf, there are no morphims between them.
\end{itemize}
\subsubsection{Holonomy groupoid over a foliated manifold}

We introduce the notion of holonomy for a foliation, that will be useful for our purposes. From now on, $\mathcal L_p$ will denote the leaf on $\mathcal F$ through the point $p$; in this case $p$ should not be confused with the index $\alpha$ in Definition  \ref{fol}, we introduce this new notation for simplicity. 


Given $p \in \mathcal{L}_p$, with $\mathcal{L}_p$ a leaf on $\mathcal{F}$, we consider a path $\alpha_0$ in $\mathcal{L}_p$ such that
$\alpha_0([0,1])\subset U_0$, with $U_0$ given by the foliation chart $(U_0, \phi_0)$. 
Consider $q_0 \in \mathcal{L}_p$ such that $\phi_0(p)$ and $\phi_0(q_0)$ 
lie on the same plaque (i.e in the same leaf with respect to the chart $(U_0, \phi_0)$) and let $T_{p}$ and $T_{q_0}$ be transversals to 
$\mathcal{F}$ through $p$ and $q_0$ respectively. A \emph{local holonomy} from $p$ to $q_0$, denoted by $Hol^{T_p,T_{q_0}}(\alpha_0)$
is defined as a germ of a diffeomorphism $f\colon  T_p \to T_{q_0}$, in such a way that there exists an open neighborhood $A$ in $T_p$ where $f$ is a 
leaf preserving diffeomorphism (i.e $a$ and $f(a)$ belong to the same leaf, for $a \in A$). 

Given a foliation and a transversal $T$ through $x$, using the fact that
\[\mbox{Diff}_{x}(T) \cong \mbox{Diff}_0(\mathbb{R}^q) \] 
where $\mbox{Diff}_0$ denotes the group of the germs of diffeomorphisms at 0, $q$ being the codimension of $\mathcal{F}$ and that the holonomy is independent of the homotopy class of the path (up to conjugation with an element in $\mbox{Diff}_0(\mathbb{R}^q$), we can see the holonomy as a 
group homomorphism

\[\mbox{hol:} \pi_1(L,x) \to \mbox{Diff}_0(\mathbb{R}^q),\]

The image of this map is denoted by $\mbox{Hol}(L,x)$.\\
Based on this notion, we define the holonomy groupoid of $\mathcal{F}$ in the natural way: the space of objects is the foliated manifold and the 
space of morphisms is the classes of holonomy of paths along the leaves of $\mathcal{F}$. Observe that the isotropy groups of this groupoid are precisely the holonomy 
groups $\mbox{Hol}(L,x)$.\\

\subsubsection{Smoothness of $L_2$}
It can be checked (see \cite{Brown}) that, given a foliated manifold $(M,\mathcal F)$, the equivalence relation $R: M \nrightarrow M$ 
of being in the same leaf, is not necessarily  a smooth submanifold of the cartesian product of the 
foliated manifold with itself. 

Fortunately, there is a way to ``resolve" the singularities, by using the holonomy groupoid associated to what are called \emph{locally Lie groupoids}. Following \cite{Brown, Brown2} we construct the holonomy groupoid associated to the equivalence relation $L_2$, denoted by $\mbox{Hol}(L_2, W)$, where the pair $(L_2,W)$ is the locally Lie groupoid associated to $L_2$ \cite{Brown}.
First, some definitions.
\begin{definition}\emph{
Let $G\rightrightarrows M$ be a groupoid. The \emph{difference} map $\delta: G\times_{(s,s)}G\to G$ is given by $\delta(g,h)=\mu(g, \iota(h))$.}
\end{definition}
\begin{definition}\emph{Let $G\rightrightarrows M$ be a (topological) groupoid. An admissible local section of $G$ is a map $\gamma: U \to G$ from an open set $U$ of $M$ satisfying the following properties:
\begin{enumerate}
\item $(s\circ \gamma) (x)=x, \forall x\in M$.
\item $(t\circ \gamma)(U)$ is an open in $M$.
\item $(t\circ \gamma): U \to \mathfrak(t\circ \gamma)$ is a homeomorphism.
\end{enumerate}}
\end{definition}
Now, consider  a subspace $M\subset W \subset G$. The triple $(s,t, W)$ is said to have \emph{enough  smooth admissible local sections} \cite{Brown}, if for each $w \in W$ there is an admissible local section $\gamma$ of $G$ satisfying that:
\begin{itemize}
\item $(\gamma \circ s)(w)=w$.
\item $\mathfrak{Im}(\gamma) \subset W$.
\item $\gamma$ is smooth.
\end{itemize}
Now we are able to introduce the notion of locally Lie groupoid:
\begin{definition}\label{locally}\emph{\cite{Brown}. A \emph{locally Lie groupoid} is a pair $(G,W)$, where $G\rightrightarrows M$ is a groupoid and a manifold $W$ such that:
\begin{enumerate}
\item $M\subset W \subset G$.
\item $W= \iota(W)$.
\item The set $$W_{\delta}:= (W\times_{(s,s)}W)\cap \delta^{-1}(W)$$
is open in $W\times_{(s,s)}W$ and $\delta$ restricted to $W_{\delta}$ is smooth.
\item $s$ and $t$ restricted to $W$ are smooth and $(s,t, W)$ has enough admissible local sections.
\item $W$ generates $G$ as a groupoid.
\end{enumerate}
}
\end{definition}
We will show how $L_2$ can be regarded as a locally Lie groupoid and its associated holonomy groupoid will be the covering manifold which allows us to regard $L_2$ as a morphism in $\mbox{\textbf{Symp}}^{Ext}$.

First, consider the foliated manifold $(M, \mathcal F)$ and a subset $U$ of $M$. We denote $L_2(U)$ the equivalence relation on $U$ defined by 
\[x \sim y \Longleftrightarrow  x \mbox{ and } y \mbox{ are in the same plaque}.\] 
Now, we consider $\Lambda=\{ (\mathcal U_{\lambda}, \phi_{\lambda})\}$ a foliation atlas for $(M, \mathcal F)$ and we define
$$W(\Lambda):= \bigcup _{\mathcal U_{\lambda}} L_2(\mathcal U_{\lambda}),$$
for all domains $\mathcal U_{\lambda}$ of the atlas $\Lambda$.

We prove the following
\begin{proposition} \emph{\cite{Brown}. $W(\Lambda)$, endowed, with the subspace topology with respect to $L_2$ (and hence regarded as a topological subspace of $M \times M$), has the structure of a smooth manifold, coming from the foliated atlas $\Lambda$.}
\end{proposition}
\begin{proof}
The same argument explained in \cite{Brown} works in the case of a foliation on Banach manifold
with finite codimension. There is an induced equivalence relation on $\phi_{\lambda}(\mathcal U_{\lambda}),$ that is determined by the connected components of $\phi_{\lambda}(\mathcal U_{\lambda}) \cap B \times \{ c\} \subset B  \times \mathbb R ^q$
and by using the coordinate function $\phi_{\lambda}$ we induce coordinate charts for $W(\Lambda)$.
\end{proof}
Moreover, it is proven (Theorem 1.3 in \cite{Brown}) that
\begin{theorem}\label{Locally}
Let $(M, \mathcal F)$ be a foliated manifold. Then an atlas $\Lambda$ can be chosen such that $(L_2, W(\Lambda))$ is a locally Lie groupoid. 
\end{theorem}
\begin{remark}\emph{
In \cite{Brown} the construction of the locally Lie groupoid structure on $(L_2, W(\Lambda))$ is done for finite dimensional foliations but it can be naturally extended to the case where the leaf is a Banach manifold and $\mathcal F$ has finite codimension. The only non-trivial step is to check that the property (3) in Definition \ref{locally} is satisfied. For this case, thanks to The Lebesgue Covering Lemma, that can be applied in the Banach case,  there is always a decomposition of a path $a$ from $x$ to $y$ on a leaf $\mathcal L$ in smaller paths $a_i$ such that $a_i$ is a path from $x_i$ to $x_{i+1}$, with $x_0=x, x_{n+1}=y$,  with the property that $(x_i,x_{i+1}) \in W(\Lambda)$.}
\end{remark}

In \cite{Brown2}, the holonomy groupoid for a locally topological groupoid is constructed through a universal property, namely:
\begin{theorem}\label{holo}\emph{(Globalisation Theorem)}\cite{Brown2}. Let $(G,W)$ be a locally topological groupoid. Then there is a topological groupoid $H\rightrightarrows N$, a morphism $\phi: H \to G$ of groupoids, and an embedding $i: W \to H$ of $W$ to an open neighborhood pf $N$ satisfying the following:
\begin{enumerate}
\item $\phi$ is the identity on objects, $\phi\circ i(w)=w, \forall w \in W$, $\phi^{-1}(W)$ is open in $H$ and $\phi \mid_{W}: \phi^{-1}(W) \to W$ is continuous.

\item (Universal property). If $A$ is a topological groupoid and $\xi: A \to G$ is a morphism of groupoids satisfying:
\begin{itemize}
\item $\xi$ is the identity on objects.
\item $\xi\mid_{W}: \xi(W) \to W$ is continuous and $\xi^{-1}(W)$ is an open in $A$ and generates $A$.
\item The triple $(s_A,t_A,A)$ has enough continuous admissible local sections,
\end{itemize}
then there is a unique morphism $\xi^{'}: A \to H$ of topological groupoids such that $\phi \xi^{'}=\xi$ and $\xi^{'}a=i\xi a, \, \forall a \in \xi^{-1}(W).$

\end{enumerate}
\end{theorem}
The groupoid $H$ is called the holonomy groupoid of the locally topological groupoid $(G,W)$ and is denoted by $Hol(G,W)$. In the smooth setting, due to Theorem \ref{Locally} , we can prove that
\begin{proposition}\emph{$Hol(L_2, W(\Lambda))$ is a Lie groupoid.}
\end{proposition}


Thus, the immersed canonical relation associated to the equivalence relation $L_2$ is the triple
$(L_2, Hol(L_2, W(\Lambda)), \phi)$, where $\phi$ is the natural projection from the holonomy groupoid to $L_2$. In fact, we get that $\phi^{-1}(x,y)= Hol(x,\gamma, y)$, that is, the holonomies of paths $\gamma$ between $x$ and $y$.\\
The next step is to adapt the argument to show that $L_1$ and $L_3$ induce immersed canonical relations. 

\subsubsection{Smoothness of $L_1$}
First of all,  we can see $L_1$ 
as a subspace of the characteristic foliation associated to $\mathcal{C}_{\Pi}$. Namely, we can think the elements of $L_1$ as the 
Lie algebroid morphisms connected to the trivial algebroid morphisms by a path along the distribution. More precisely, if we denote 
by $C \subset \mathcal{C}_{\Pi}$ the submanifold corresponding to the trivial Lie algebroid morphisms ($X$ is constant and $\eta$ is 0), then
\begin{equation}
L_1= \{ \sqcup_{\mathcal{L}\in \mathcal{F}}\mathcal{L} \vert \mathcal{L} \cap C \neq \emptyset  \}. 
\end{equation}
The characteristic foliation can be understood as the space of orbits of a gauge group $H$ acting on $\mathcal{C}_{\Pi}$, 
where $H$ corresponds to the group of local diffeomorphisms generated by the flows of the Hamiltonian vector fields associated to the 
Hamiltonian functions:
\[H_{\beta}(X,\eta)=\int_I \langle dX(u)+ \pi^{\#}(X(u))\eta(u), \beta(X(u),u) \rangle , \]
where $\beta\colon  I \to \Omega^1(M)$ and $\beta(0)=\beta(1)=0.$
This action can be written in local coordinates as follows:
  

\begin{eqnarray}
\delta_{\beta} X^i(u) &=& -\pi^{ij}(X(u))\beta_j(X(u),u)    \\
 \delta_{\beta}\eta_i(u) &=& d_u\beta_i(X(u),u)+\partial_i \pi^{jk}(X(u))\eta_j(u)\beta_k(X(u),u). \nonumber \\
\end{eqnarray}\label{for}

With this prescription, it is easy to check that the submanifold \[C:=\{(X,\eta)\vert X= X_0,\, \eta=0  \},\] which is an $n$-dimensional submanifold of $\mathcal C_{\Pi}$, where $n= \dim M$,  intersects  the foliation neatly, i.e. 
\[ T_x C \cap T \mathcal L_x = \{ 0\}, \forall x \in C\cap L_x.
\]This comes from the fact, that, after the prescribed gauge transformation,
the points of $C$ are trivially stabilized: the gauge transformation preserves fixed the intial and final points of the path, 
and the fact that the space $\mathcal{C}_\Pi$ is invariant under this gauge transformation implies that there is a unique point 
for each leaf and that the tangent to the orbit (that is given precisely by the gauge) and the tangent to $C$ are independent. 
Choosing a transversal $C \subset T$ to the foliation $\mathcal F$, the restriction of the holonomy of $\mathcal F$ to $C$, induces the covering 
\[p\colon  Hol(L_2, W(\Lambda))\mid_{L_1} \to L_1,\]
with fibers the holonomy of paths along the fibers over $C$. Thus, the induced immersed canonical relation for $L_1$ is given by $(L_1, Hol(L_2, W(\Lambda))\mid_{L_1}, p)$.

\subsubsection{Smoothness of $L_3$.}
Here, we describe $L_3$ in a suitable way so we find a smooth covering for it.
The idea of the proof is to use the holonomy groupoid for an equivalence relation, understanding the space $L_3$ in terms of an equivalence homotopy relation. 
First of all, a remark:
\begin{remark}\emph{ The $s$ and $t$ fibers are saturated by the leaves of $\mathcal{F}$ restricted to $\mathcal{C}_{\Pi}$.}
\end{remark}

In other words, given the fact that the characteristic foliation can be understood as the space of orbits of gauge transformations, leaving invariant the initial 
and final points of the paths, the equivalence relation determined by $\mathcal{F}$ is finer than the one determined by $s$ or $t$.\\

In a similar way:
\begin{remark}\emph{
The fibers of the the fibered product of maps:
\[(s\times t)\colon \mathcal{C}_{\Pi} \times \mathcal{C}_{\Pi} \to M \times M \] 
are saturated by the leaves of the product foliation $\mathcal{F} \times \mathcal{F}.$
}\end{remark}
In this way, $\mathcal{F} \times \mathcal{F}$ restricts to a foliation $\mathcal{F}_{(s,t)}$ in 
\[\mathcal{C}_{\Pi}\times_{(s,t)}\mathcal{C}_{\Pi}\subset \mathcal{C}_{\Pi}\times \mathcal{C}_{\Pi}:= (s\times t)^{-1}\Delta\]
This restricted foliation has finite codimension, more precisely
\[\mbox{codim}_{\mathcal{C}_{\Pi}\times_{(s,t)}\mathcal{C}_{\Pi}}\mathcal{F}_{(s,t)}=\mbox{codim}_{\mathcal{C}_{\Pi}\times \mathcal{C}_{\Pi}}\mathcal{F}\times\mathcal{F}-\mbox{codim}_{\mathcal{C}_{\Pi}\times \mathcal{C}_{\Pi}}\mathcal{C}_{\Pi}\times_{(s,t)}\mathcal{C}_{\Pi}=2n. \]
In this way, for a triple $(a,b,c) \in L_3$, the pair $(a,b)$ is an element in $(\mathcal{C}_{\Pi}\times \mathcal{C}_{\Pi}, \mathcal{F}_{(s,t)})$. $c$ can be identified with an 
element in $\mathcal{C}_{\Pi}$ via the smooth map 
\begin{eqnarray}
\tilde{\beta}\colon  (\mathcal{C}_{\Pi}\times_{(s,t)}\mathcal{C}_{\Pi}) &\to& (\mathcal{C}_{\Pi}, \mathcal{F})\nonumber \\
(a,b) &\to& a\star b
\end{eqnarray}
where
\[ a\star b(t)= \left\{ \begin{array}{rl}
 a(\beta(2t))&, t\in [0,\frac 1 2] \\
 b(\beta(2t-1))&, t \in [\frac 1 2, 1]
       \end{array} \right.\]
and $\beta$ denotes a bump function $\beta: [0,1] \to [0,1]$.
Therefore, it is possible to characterize the space $L_3$ in the following way:
\[L_3=\{(a,b,c)\in (\mathcal{C}_{\Pi}\times_{(s,t)} \mathcal{C}_{\Pi})\times \mathcal{C}_{\Pi} \vert \mathcal{L}_{(\tilde{\beta}(a,b))}=\mathcal{L}_c\}\]
where $\mathcal{L}$ denotes (as before), the orbits of the $T^*M$-homotopy.
Hence, the induced immersed canonical relation for $L_3$ is $(L_3, \mathcal{C}_{\Pi}\times_{(s,t)} \mathcal{C}_{\Pi}, Hol(L_2, W(\Lambda))).$ 

\subsubsection{Embedding conditions and integrability of $T^*M$}
So far, the smooth immersed structure for the spaces $L_i$ has been given. In this section, the integrability conditions for the 
Lie algebroid given by the Poisson structure on $M$ are introduced, following the work of M. Crainic and R.Loja Fernandes \cite{Crai}. 
There, the integrability conditions are described in terms of the \textit{monodromy groups} associated to the characteristic folitation.
\begin{theorem} A Lie algebroid $\mathcal{A}$ is integrable if and only if the following conditions hold:

\begin{enumerate}
\item The associated monodromy subgroups $N_x(\mathcal{A})$ are discrete ($r(x)=0$).

\item $N_x(\mathcal{A})$ are locally uniform discrete, that is: \[\liminf_{y \to x} r(y) > 0.\]

\end{enumerate}

\end{theorem}

The objective of this section is to give a different interpretation of the integrability 
conditions in terms of the previously defined canonical relations $L_i$. The monodromy groups are defined as follows:

\begin{definition}

For any $x \in M $the subset $N_x(\mathcal{A})$ is the subset of the center of 
$\mathfrak{g}_x$ formed by elements  $v \in Z(\mathfrak{g}_x)$  such that the constant $\mathcal{A}-$path $v$ is $\mathcal{A}-$ homotopic equivalent to the trivial $\mathcal{A}-$ path.
 \end{definition}
It can be proven that the subspaces $N_x(A)$ are in fact, subgroups of $Z(\mathfrak{g}_{x})$. Considering a metric on the bundle $A$ and its associated distance $d$, we define the \textit{size} of the monodromy group as:
$$
r(x) = \left\{ \begin{array}{rl}
 d(0_x, N_x(A)-{0_x}) &\mbox{if } N_x(A)-{0_x}\neq \emptyset  \\
  +\infty &\mbox{if } N_x(A)-{0_x}=\emptyset 
       \end{array} \right.
$$
We restrict ourselves to the case when $A=T^*M$.

Consider the infinite dimensional bundle $\pi\colon  T^*PM\to PM$ and $\sigma\colon  PM\to T^*PM$ its zero section. The monodromy group $N_{x_0}(T^*M)$ 
corresponds now to
\[N_{x_0}(T^*M)= T^*_{(\overline{X,\eta})}PM \cap L_1 \]
where $(\overline{X,\eta})$corresponds to $X=x_0$ and $ \eta \in Ker \pi^{\sharp}$. This characterization guarantees that the $T^*M-$homotopies preserve the base path and it transforms elements in the 
kernel of $\pi^{\sharp}\subset T^*_{x_0}M$. Let $\overline{L_1}:= \cup_{x_0 \in M} T^*_{(\overline{X,\eta})}PM \cap L_1$.
In this way, the \textit{size} of the monodromy group seems natural. Giving a metric $g$ on $T^*PM$, induced by a metric on $T^*M$ and its corresponding norm we get
\[r(x)=\liminf \vert v \vert, v \in T^*_{(\overline{X,\eta})}PM .\]
Now, the integrability conditions in Theorem 1 together imply the following\\
\textit{(Locally uniform discreteness):} $\forall x_0 \in M,\exists \,\varepsilon > 0$ and an open neighborhood $U$ containing $x_0$ and contained in
$\overline{L_1}$ such that $\forall v\in T_{U}PM \setminus \sigma (U)$, we have that $\vert v \vert >0.$ This is precisely the property of $\overline{L_1}$  being an embedding.
Therefore, we have proven the following 
\begin{theorem}
If the Poisson manifold $M$ is integrable, then, there exists a tubular neigborhood of the zero section of $T^*PM$, denoted by $N(\Gamma_0(T^*PM))$ such 
that $\overline{L_1}\cap N(\Gamma_0(T^*PM))$ is an embedded submanifold of $T^*PM$.
\end{theorem}
\subsubsection{Proofs of smoothness of $L_i$ in the integrable case.}
It is possible to check that, in the integrable case, $L_1$ is an immeresd submanifold of $\mathcal{G}$, in an easier way than in the general case. This comes from the following lemma:

\begin{lemma} \cite{Nash}\emph{
Let $X$ and $Y$ be Banach manifolds and let 
$r:X\to Y$ be a smooth submersion. Let $P$ be an embedded submanifold of $Y$. Then $r^{-1}(P)$ is an embedded submanifold of $X$.}
\end{lemma}

Appling the lemma for $X=C_{\Pi}$, $Y=, \underline{C_{\Pi}}$, $r$ the quotient map and $P=M$, where $M$ is identified with the space of units of the symplectic groupoid integrating $M$, we obtained the desired result.\\
In order to prove that $L_2$ is a submanifold in the case where $T^*M$ is integrable, let us use the following result:

\begin{lemma} \cite{Nash}\label{pull}\emph{
Let $M_1,M_2$ and $M$ be Banach manifolds and let 
\[P_1\colon M_1 \to M,\,\, P_2\colon M_2 \to M \]
be smooth submersions. Then $M_1 \times_{P_1,P_2} M_2$ is a closed embedded submanifold of $M_1\times M_2$.}
\end{lemma}

\begin{proof} The idea is to construct explicit charts for $M_1 \times_{P_1,P_2} M_2$. We denote $\Delta _M$ the diagonal in $M\times M$. By definition, $M_1 \times_{P_1,P_2} M_2:= (P_1 \times P_2)^{-1} 
\Delta_M $ and by continuity reasons the space is closed. Now, observe that, because $P_1$ and $P_2$ are submersions, 
there exists coordinate charts $\phi_{M_1}\colon  U_1 \times V_1 \to M_1, \phi_{M_2}\colon  U_2 \times V_2 \to M_2$ and $\phi_1\colon  V_1 \to M,\, 
\phi_2: V_2 \to M$ in such a way that the following diagram commute

\[\begin{array}[c]{ccc}
U_i\times V_i&\stackrel{\phi_{M_i}}{\rightarrow}&M_i\\
\downarrow\scriptstyle{\pi_2}&&\downarrow\scriptstyle{P_i}\\
V_i&\stackrel{\phi_i}{\rightarrow}&M
\end{array}\]
for $i\in \{1,2\}$ and $\pi_2$ denotes the projection in the second component. Denoting $V:= V_1 \cap V_2$, restricting the previous charts to 
$V$ we obtain coordinate charts $$\phi_{M_1}\vert_{V} \times \phi_{M_2}\vert_{V}\colon  U_1 \times V \times U_2 \times V \to M_1 \times M_2 $$
and restricting to the diagonal of $V\times V$ we obtain 
\[\phi_{M_1}\vert_{V} \times \phi_{M_2}\vert_{V}(U_1 \times U_2 \times \Delta _{V\times V})= M_1 \times_{P_1,P_2} M_2 \]
as we wanted. 
\end{proof}

With this lemma in mind, we observe that $L_2$ can be seen as a fibered product in the following way:

\[\begin{array}[c]{ccc}
L_2 &\stackrel{\pi_1}{\rightarrow}&X\\
\downarrow\scriptstyle{\pi_2}&&\downarrow\scriptstyle{p}\\
X &\stackrel{p}{\rightarrow}&X / \mathcal{F}
\end{array}\]
where $L_2$ is precisely $X\times_{\mathcal{F}}X$ and $p$ corresponds to the quotient map.\\

To prove that $L_3$ is an immersed submanifold, the main observation is that $L_3$ is given by a fiber product satisfying the conditions 
of the lemma \ref{pull} and using the same argument as in the proof for $L_2$ the result holds.

Therefore, to summarize,  we have just proved the following fact:
\begin{theorem}\label{Embeddability} If $T^*M$ is integrable, then $L_1\cap N(\Gamma_0(T^*PM))$ is an embeded submanifold of $T^*(PM)$ and the spaces $L_1, L_2$ and $L_3$ are immersed submanifolds of $T^*(PM), T^*(PM)^2$ and $T^*(PM)^3$, respectively.
\end{theorem}

The next step is to connect the construction of the relational symplectic groupoid for $T^*PM$, which is infinite dimensional, with the s-fiber simply connected symplectic Lie groupoid integrating a Poisson manifold.  The connection is given by the following
\begin{theorem} \label{reduction}Let $(M,\Pi)$ be an integrable Poisson manifold. Let $\mathcal{G}$ be the relational symplectic groupoid associated to $T^*PM$ described above and let $G= \underline{C_{\Pi}}$ be the symplectic Lie groupoid associated to the characteristic foliation on $C_{\Pi}$. Then $\mathcal{G}$ and $G$ are equivalent as relational groupoids.
\end{theorem}
\begin{proof}
This is a direct consequence of Proposition \ref{pro}, since the previously described relational symplectic groupoid is regular.
\end{proof}
Another fact that results useful with the introduction of relational symplectic groupoids is the comparison of different integrations of Poisson manifolds, i.e. we do not restrict only to the case where the symplectic groupoid is $s$-fiber simply connected. The following Proposition (for more details see \cite{Mor} for the more general case of Lie algebroids) relates different symplectic groupoids integrating a given Poisson manifold $(M, \Pi)$.
\begin{proposition}\emph{
Let $G_{ssc} \rightrightarrows M$ be the s-fiber simply connected symplectic groupoid integrating $(M, \Pi)$ and let $G^{'}\rightrightarrows M$ be another s-fiber connected symplectic groupoid integrating $(M, \Pi)$. Then there exists a discrete group $H$ acting on $G_{ssc}$ such that $G= G_{ssc}/ H$ and the quotient map $p\colon  G_{ssc}\to G$ is the unique groupoid morphism that integrates the identity map 
$id: T^*M \to T^*M$.}
\end{proposition}
With this Proposition in mind, we observe that the projection map $p$, being a local diffeomorphism, is naturally compatible with the symplectic structures of $G_{ssc}$ and $G$, therefore, it corresponds to a morphism of symplectic groupoids and by definition, it corresponds to a morphism of relational symplectic groupoids. Moreover, since locally $p^{-1}$ is also a diffeomorphism the adjoint relation $p^{\dagger}$ is also a morphism. Therefore we have the following
\begin{proposition} \label{Integrations}\emph{Let $G\rightrightarrows M$ and $G^{'}\rightrightarrows M$ be two s-fiber connected symplectic groupoid integrating the same Poisson manifold $(M, \Pi)$. Then $(G,L,I)$ and $(G^{'}, L^{'}, I^{'})$ are equivalent as relational symplectic groupoids.}
\end{proposition}
As a result of this proposition we obtain the following
\begin{corollary} \emph{
If $M$ is an integrable Poisson manifold, then the relational symplectic groupoid on $T^*PM$ is equivalent to every s-fiber connected symplectic groupoid integrating $M$.}
\end{corollary}

\subsection{Lagrangianity of $L_i$}
First we prove the following 

\begin{proposition}\emph{
The tangent space $TL_2$ is a Lagrangian subspace of $T(T^*(PM)) \oplus \overline{T(T^*(PM))}.$
}
\end{proposition}

\begin{proof}
First we prove the following 
\begin{lemma}
$TC_{\Pi}^{\perp} \oplus TC_{\Pi}^{\perp} \subset TL_2$.
\end{lemma}
\begin{proof}
To prove this lemma, we observe first that, according to \cite{Cat}, the leaves of the characteristic foliation of $C_{\Pi}$ are precisely the orbits of the gauge equivalence relation given by $L_2$ in $C_{\Pi}$. Therefore we get that 
\[TL_2= R^{C} \]
as in Equation \ref{char} and therefore we get that
\[ TC_{\Pi}^{\perp} \oplus TC_{\Pi}^{\perp} \subset TL_2 \subset TC_{\Pi} \oplus TC_{\Pi}\]
Observe now that the projection of $TL_2$ with respect to the coisotropic reduction of $C_{\Pi}$ is precisely the diagonal of $C_{\Pi}$ that is a Lagrangian subspace of $TC_{\Pi} \oplus \overline{TC_{\Pi}}$.
Now, the space $TL_2$ satisfies the conditions of Proposition \ref{Coisotropic} and therefore $TL_2$ is Lagrangian, as we wanted.
\end{proof}
\end{proof}


\begin{proposition} \label{L1Lag}\emph{The tangent space $TL_1$ is a Lagrangian subspace of  $T(T^*(PM))$.}
\end{proposition}
\begin{proof}
First, we prove the following 
\begin{lemma} \emph{
$TL_1$ is an isotropic subspace.}
\end{lemma}
\begin{proof}
The direct computation of the tangent space $TL_1$
 yields 
$$T_{\gamma}L_1= (\delta X(t)+ v, \delta \eta(t)) \mid (\delta X(t), \delta \eta(t) \in TC^{\perp}, v \in T_{\gamma(0)}M).$$
Now, considering two vectors in $T_{\gamma}L_1$ denoted by $(\delta_1 X(t)+ v_1, \delta_1 \eta(t))$  and $(\delta_2 X(t)+ v_2, \delta_2 \eta(t))$ we compute in local coordinates
\begin{eqnarray*}
&\omega&((\delta_1 X^i(t)+ v^i_1, \delta_1 \eta_i(t)), (\delta_2 X^i(t)+ v^i_2, \delta_2 \eta_i(t)))\\ &=&\int_0^1(\delta_1X^i(t)+v^i_1)\delta_2\eta_i(t)-(\delta_2X^i(t)+v^i_2)\delta_1\eta_i(t))dt\\
&=&\int_0^1((\delta_1X^i(t)\delta_2\eta_i(t)-\delta_2X^i(t)\delta_1\eta_i(t))dt+\int_0^1v^i_1\delta_2 \eta_i(t)dt-\int_0^1v^i_2\delta_1 \eta_i(t)dt.
\end{eqnarray*}
The first integral vanishes since $C$ is coisotropic. The second and third integrals vanish since
\begin{equation*}
\int_0^1v^i_1\delta_2 \eta_i(t)dt=\int_0^1v^i_2\delta_1 \eta_i(t)dt= \eta_1(1)-\eta_1(0)= \eta_2(1)-\eta_2(0)=0.
\end{equation*}
\end{proof}
Now, since $TL_1$ is isotropic, after reduction we get that
\begin{eqnarray*}
 &\underline{\omega}&([(\delta_1 X^i(t)+ v^i_1, \delta_1 \eta_i(t)],[\delta_2 X^i(t)+ v^i_1, \delta_2 \eta_i(t)])\\
&=&\omega((\delta_1 X^i(t)+ v^i_1, \delta_1 \eta_i(t)), (\delta_2 X^i(t)+ v^i_2, \delta_2 \eta_i(t)))=0.
\end{eqnarray*}
Therefore $\underline{TL_1}$ is isotropic. Now, since 
\[\underline{T_{\gamma}L_1}= \{ v \in T_{\gamma_{0}}M \} \sim T_{\gamma_0}M,\]
we get that
\[\dim \underline{TL_1}= \dim T_xM= n= 1/2 \dim \underline{C_{\Pi}}.\]
This implies that $\underline{TL_1}$ is Lagrangian and then, by applying Proposition \ref{Coisotropic}, we conclude that $TL_1$ is Lagrangian, as we wanted.
\end{proof}
Now, we prove that
\begin{proposition}\emph{
The space $TL_3$ is a Lagrangian subspace of  $$T(T^*(PM)) \oplus T(T^*(PM)) \oplus \overline{T(T^*(PM))}.$$
}
\end{proposition}
\begin{proof}
In order to prove this Proposition, we first prove the following
\begin{lemma}\emph{ Let $\delta \gamma_1$ and $\delta \gamma_2$ be two vectors in $TC_{\gamma_1}^{\perp}$ and $TC_{\gamma_2}^{\perp}$ that are composable. Then $\delta \gamma_1*\delta \gamma_2 \in TC_{\gamma_1*\gamma_2}^{\perp}. $
}
\end{lemma}
\begin{proof}
This follows immediately from the additive property of $\omega$ with respect to concatenation, namely, if $\delta \gamma$ is a vector in $T_{\gamma_1*\gamma_2}C$, then
\begin{eqnarray*}
\omega(\delta \gamma_1*\delta \gamma_2, \delta \gamma)= \alpha_1 \omega(\delta \gamma_1, \delta_\gamma)+ \alpha_2 \omega(\delta \gamma_2, \delta_\gamma)=0,
\end{eqnarray*}
where $\alpha_i$ are factors due to reparametrizations for $\gamma_i$.
\end{proof}
With this Lemma at hand, we can conclude, from Equation \ref{ELE3} that
\[TC^{\perp} \oplus TC^{\perp} \oplus TC^{\perp} \subset TL_3 \subset TC\oplus TC \oplus TC.\]
Now, after reduction we get that
\begin{eqnarray*}
\underline{\omega}\oplus \underline{\omega}\oplus -\underline{\omega}&\,& ([\delta_1\gamma_1] \oplus [\delta_1\gamma_2]\oplus [\delta_1\gamma_3], [\delta_2\gamma_1] \oplus [\delta_2\gamma_2]\oplus [\delta_2\gamma_3])\\
&=& \underline{\omega}([\delta_1\gamma_1, \delta_1\gamma_2])+\underline{\omega}([\delta_2\gamma_1, \delta_2\gamma_2])-\underline{\omega}([\delta_1\gamma_1*\delta_1\gamma_2], [\delta_2\gamma_1*\delta_2\gamma_2]),
\end{eqnarray*}
that is zero by the additivity property for $\omega$.
This implies that $\underline{L_3}$ is isotropic. Now, by counting dimensions, we get that
the compatibility condition for $\gamma_1$ and $\gamma_2$ give $3\dim (M)$ independent equations (for the initial, final and coinciding point of $\gamma_1, \gamma_2$ and $\gamma_3$).
Hence,
\[\dim(\underline{L_3})= 6\dim(M)-3\dim (M)= 1/2 \dim (\underline{TC}\oplus \underline{TC}\oplus \underline{TC}).\]
This implies that $\underline{TL_3}$ is Lagrangian. By Proposition \ref{Coisotropic} we conclude that $TL_3$ is Lagrangian, as we wanted.

\end{proof}

\section{Relational symplectic groupoids for Poisson manifolds} 
This section is devoted to illustrate with example the construction previously described. In the integrable case, we show how they give rise to the usual s-fiber simply connected symplectic groupoid integrating $(M,\Pi)$.
\subsection{The zero Poisson case}
In this section we consider $M$ an $n$-dimensional manifold and $\Pi$ is the zero Poisson bivector. The space $\mathcal G$ can be described as 
$$\mathcal G= \{(X,\eta) \mid X\equiv x_0 \in M, \eta \in \Omega^1(I, T^*_{x_0}M). \}$$
In this case, according to the definition of 
$C_{\Pi}$ in Equation \ref{Cois} we obtain

\[C_{\Pi}=\{(X,\eta) \vert X\equiv x_0, \eta \in \Gamma (T^{*}I \otimes T^{*}_{x_0}M)\}, \forall x_0 \in M. \]
and its linearized version gives rise to the following description for $TC_{\Pi}$ in local coordinates:
$$TC_{\Pi}=\{(\delta X^i, \delta \eta_i)\mid \delta X^i=0, \delta \eta_i \in \mathcal C^{k}(I, \mathbb R^n).\}$$
Now, it is possible to compute (componentwise) the symplectic orthogonal complement of the space $TC_\Pi$:
\[(TC_\Pi^{\perp})_i= \{ (\tilde{\delta X^i}, \delta \tilde{\eta_i}) \vert \int_I \delta \tilde{X^i}\eta_i + \int_I X^i \delta \tilde{\eta_i}=0, \forall \delta X^i, \delta \eta_i \in (C_{\Pi})_i\}.\]
that is,
\begin{equation}\label{cperp}(TC_\Pi^{\perp})_i=\{(\delta X^i, \delta \eta_i)\mid \delta X^i\equiv 0, \delta \eta_i= d \beta,  \beta: I \to T_{x_0}M, \beta(0)=\beta(1)=0\}.
\end{equation}

that is contained in $TC_{\Pi}$, as expected.
As a corollary, we obtain that the reduced space is the expected one:
\[\underline{C_\Pi}= T^*M.\]

Observe that the space $L_1$ is obtained by gauge equivalent paths to the space
\[C= \{(X,\eta)\vert X\equiv x_0, \eta\equiv 0\}\]
and since the gauge transformations leave $C_\Pi$ invariant and the $\eta$ component change modulo an exact form, 
it is possible to conclude that
\[L_1= \{(X,\eta) \in C_{\Pi}\mid X\equiv x_0, \eta=d\beta, \beta: I \to T^*_{x_0}M, \beta(0)=\beta(1)=0\}. \]
Here, we can observe that $C_{\Pi}^{\perp} \subset L_1 \subset C_{\Pi}$ and in addition, the reduced space $\underline{L_1}$ 
corresponds to the zero section of the cotangent bundle $T^{*}M$ that is Lagrangian.\\
It is possible to prove the previous statement directly: Computing the symplectic orthogonal space to $TL_1$ we get:
\begin{eqnarray}
TL_1^{\perp}&=& \{(\delta \tilde{X},\delta \tilde{\eta})\vert \int_I \delta \tilde {X} \delta \eta + \int_I \delta X \delta \tilde{\eta}=0, \forall (X,\eta) \in L_1 \\
&=& \{(\delta \tilde{X},\delta \tilde{\eta})\vert \int_I \delta \tilde{X} \delta \eta + \delta x\int_I \delta \tilde{\eta}=0, \forall x \in \mathbb{R}^n, \eta= df, f\in \mathcal{C}^{\infty}(\mathbb{R}^n)\}.  
\end{eqnarray}
To check that $L_1$ is isotropic, we observe that 
\[\int_I\delta \tilde{X}\eta + \int_I \delta X\delta \tilde{\eta}= \delta \tilde{x}\int_I \delta \eta+ \delta x\int_I \delta \tilde{\eta}= 0+0, \forall x,\delta \tilde{x} \in \mathbb{R}^n,\] hence $TL_1 \subset TL_1^{\perp}$.
To check that $L_1$ is coisotropic, let $f\equiv 0$. This implies that $\int_I \delta \tilde{\eta}=0$ and by a similar argument as before, 
$\delta \tilde{\eta}=d\beta, \beta(0)=\beta(1)=0$. Hence, $\int_I\delta \tilde{X} df=0$ and by integration by parts, $\int_I d\delta \tilde{X}f, 
\forall f \in \mathcal{C}^{\infty}(\mathbb{R}^n)$, which implies that $d\delta \tilde{X}=0$, therefore $\tilde{X}$ is constant and this 
proves that $L_1^{\perp} \subset L_1$, as we wanted.\\
It is also possible to prove that $L_2$ is Lagrangian. We have that
\begin{eqnarray}\label{ele2}
L_2=\{(X_1,\eta_1),(X_2,\eta_2) \in C_{\Pi}\times C_{\Pi} &\mid& X_1=X_2=x \in \mathbb{R}^n,\\&& \eta_1-\eta_2= d\beta,\\&& \beta: I \to T^*_{x_0}M, \beta(0)=\beta(1)=0\}.
\end{eqnarray}
From Equations \ref{cperp} and \ref{ele2} we obtain that
\[(TC_{\Pi} \oplus TC_{\Pi})^{\perp} \subset L_2 \subset (TC_{\Pi}\oplus TC_{\Pi}).\]
Using the fact that the reduced space $\underline{L_2}$ is precisely $\bigtriangleup (T^*M)$ that is Lagrangian subspace and by 
lemma \ref{Coisotropic} the space $L_2$ is Lagrangian.\\
To describe the space $L_3$, observe that the two composable $T^*M-$ paths have to be constant and the product is taken modulo gauge transformation, i.e. 
modulo an exact form. Hence,
\[L_3=\{(X_1,\eta_1),(X_2,\eta_2), (X_3,\eta_3)\mid X_1=X_2=X_3, \int_I\eta_3-\eta_2-\eta_1=0\}, \]
A similar observation as before proves that $L_3$ is Lagrangian.
To summarize, the relational symplectic groupoid associated to the zero Poisson structure is given by the following data:
\begin{eqnarray*}
\mathcal{G}&=& T^*PM  \\
L&=&\{(X_1,\eta_1),(X_2,\eta_2), (X_3,\eta_3)\mid X_1=X_2=X_3, \int_I\eta_1+\eta_2+\eta_3=0\}\\
I&=&\{(X,\eta)\mapsto (X\circ \phi, -\eta)\},
\end{eqnarray*}
with $\phi$ as in Equation \ref{phi}.

In this example it is straightforward to check that the induced inclusion maps of the spaces $L_i$ are smooth embeddings. 
In addition, it is possible to study the image of these spaces under reduction. $\underline{L_1}$ is precisely the zero section of the cotangent bundle, 
that is the quotient space $L_1 / L_2$ as well. $\underline{L_2}$ by definition corresponds to $\triangle(T^*M)$ that is precisely the 
graph of the identity in the reduced space $G=T^*M.$ In the case of $L_3$, we obtain that the reduced space $\underline{L_3}=L_3 / L_2$ 
corresponds to
\[(x,\alpha, \beta) \in M\times \Omega^1(M)\times \Omega^1(M)\]
that is identified with
\[\underline{L_3}= \{(x,\alpha),(x,\beta),(x,\alpha+ \beta)\in T^*M \times \overline{ T^*M\times T^*M} \},\]
that is the graph of fiber wise multiplication in $G$ as expected.\\
Therefore, to summarize, the groupoid can be recovered out of the relational groupoid construction:
\begin{eqnarray*}
G&=& \underline{C_{\Pi}}= T^*M. \\
\underline{L_1}&=& L_1 / L_2= \underline{L_1}=M.\\
\underline{L_2}&=& \triangle T^*M = Gr(id_{T^*M}).\\
\underline{L_3}&=& L_3 /L_2 = Gr(m_{T^*M}).
\end{eqnarray*}
Here $m_{T^*M}$ denotes the fiber wise sum in the cotangent bundle and the structure maps $s,t,\varepsilon, \iota$ factor naturally through the quotient.

\subsection{The symplectic case} In this example we assume for simplicity that
$M= \mathbb{R}^{2n}$ equipped with the canonical symplectic structure $\omega=dx^i \wedge dy_i.$ \footnote{The argument can be extended to general symplectic manifolds.}From equation (\ref{Coi}), we obtain, 
since $\Pi^{ij}=\delta^{ij}$ , that $\eta_i = \dot{X}_i.$ Therefore the space $C_{\pi}$ is diffeomorphic to the path space $PM$. In order to show that this 
space is coisotropic, observe that the symplectic orthogonal space to $C_{\Pi}$ is given by:

\[C_{\Pi}^{\perp}=\{(\tilde{X},\tilde{\eta})\vert \int_I \langle X, \tilde{\eta} \rangle + \langle \tilde{X}, \dot{X} \rangle =0,\, \forall X \in PM\}.\]
Now, substituting $\tilde{\eta}= \dot{\tilde{X}} + \gamma$ we get that integral can be rewritten as:
\[ \int_I d \langle X, \dot{\tilde{X}} \rangle + \int_I \langle X, \gamma \rangle,\]
implying that $\int_I \langle X, \gamma \rangle=0, \forall X \in PM$ and from here we can deduce that $\gamma \equiv 0.$ In this way, 
$(\tilde{X},\tilde{\eta})=(\tilde{X},\dot{\tilde{X}}) $ and hence, $C_{\Pi}^{\perp} \subset C_{\Pi}$.
It is easy to observe that, given a path $X$, the gauge transformations corresponds in this case to 
diffeomorphisms of $\mathbb{R}^{2n}$ leaving $X(0)$ and $X(1)$ constant.
With this observation in mind, we compute the space $L_1$:
\[L_1= \{(X, \eta)\vert X= X_0 + \beta,\, \eta= d\beta, \beta \in \mathcal{C}^{\infty}_0(\mathbb{R}^{2n}), X_0 \in \mathbb{R}^{2n}\},\]
which implies that $L_1$ is diffeomorphic to the loop space $L(\mathbb{R}^{2n})$. It is possible to compute the symplectic orthogonal to $L_1$ and to check that in 
fact is a Lagrangian subspace:
\[L_1^{\perp}= \{(X,\eta) \vert \int_I \langle X, d\beta \rangle + \langle \beta, \eta  \rangle\}\]
and using again the substitution $\eta= dX + \gamma$ we conclude that $\gamma= 0$ and then $\eta= dX$, 
with $X(0)=X(1)=X_0$, as we wanted.\\
To characterize the space $L_2$, we consider the quotient map  $\pi\colon  P(\mathbb{R}^{2n})\to \mathbb{R}^{2n} \times \mathbb{R}^{2n}$
and given the description of the gauge transformations before, we conclude that:
\[L_2=P(\mathbb{R}^{2n}) \times_{(\pi,\pi)} P(\mathbb{R}^{2n}) .\]\\
The space $L_3$ can be written as a fibered product in the following way:
\[L_3= (P(\mathbb{R}^{2n}) \times_{(t,s)} P(\mathbb{R}^{2n}))\times_{(s\times t)\times(s,t)}P(\mathbb{R}^{2n}). \]\\
The relational symplectic groupoid associated to a symplectic space is therefore given by:
\begin{eqnarray*}
\mathcal{G}&=& T^*P\mathbb{R}^{2n}.  \\
L&\cong& (P(\mathbb{R}^{2n}) \times_{(t,s)} P(\mathbb{R}^{2n}))\times_{(s\times t)\times(s,t)}P(\mathbb{R}^{2n}). \\
I&=&\{(X,\eta)\mapsto (X\circ \phi, -\eta)\},
\end{eqnarray*}
and the same structure maps $\varepsilon, \iota, s$ and $t$ as in the previous example. Furthermore, the spaces $L_i$ 
are embedded Lagrangians as expected.

After reduction, we observe that the space $L_1/L_2$ is isomorphic to $\mathbb{R}^{2n}$. The image of $L_2$ 
after reduction is, as expected, the graph of the identity in the pair groupoid $M\times M$ and the image of $L_3$ is precisely the graph of 
the pair multiplication $m\colon \triangle_3 (M) \subset (M\times M)\times (M\times M) \to M\times M$ given by $m((a,b),(b,c))=(a,c)$. Hence,

\begin{eqnarray*}
G&=& \underline{C_{\Pi}}= M \times M. \\
\underline{L_1}&=& L_1 / L_2= \mathbb{R}^{2n}=M.\\
\underline{L_2}&=& \triangle (M \times M) = Gr(id_{M\times M}).\\
\underline{L_3}&=& L_3 /L_2 = Gr(m_{M\times M}).
\end{eqnarray*}

\subsection{The constant case} In this case, we consider $M=\mathbb{R}^n$ and coordinates $x_1,\ldots, x_{2k},\ldots, x_n$,  in such way that $\Pi$ can be written in the form

$$\Pi^{ij} = \begin{pmatrix} \omega^{ij}&0\\ 0&0 \end{pmatrix},$$
where $\omega$ is a $2k\times 2k$ invertible constant matrix. Combining the previous examples we obtain that 
\[C_{\Pi}= C_{sym}(\mathbb{R}^{2k}) \times C_{0}(\mathbb{R}^{n-2k}) ,\]
where the first component is the space of solutions of the constraint equation for the symplectic structure $\omega$ 
and the second component correspond to the space of solutions for the zero Poisson structure. Therefore, If $(\mathcal G, L, I)_{ct}$ denotes the relational symplectic groupoid associated to a constant Poisson structure, 
$$(\mathcal G,L,I)_{ct}=(\mathcal G,L,I)_{sym}\times (\mathcal G,L,I)_0.$$

\subsection{The linear case}\label{linear}
We start with this example, where $M= \mathfrak{g}^{*}$ is the dual of a finite dimensional Lie algebra  
equipped with a natural Poisson structure, also known as the Kirillov-Kostant structure. To be more precise, the Poisson structure in this case is given by the 
structure constant for the Lie algebra  $\mathfrak{g}$:

$$\Pi= c^{ij}_k x^k \frac{\partial}{\partial x_i}\wedge \frac{\partial}{\partial x_j} $$
and the constraint equation looks like:
\[dX^i= c^{ij}_k X^k \eta_j.\]
Considering $\eta$ as a connection in coadjoint bundle on $I$ associated to $G$:$E=\mathfrak{g}^{*}\times I$ 
($G$ is the Lie group integrating $\mathfrak{g}$ 
and the action is the coadjoint one), the previous equation can be rewritten as: 

\begin{equation}\label{K}
\triangledown_{\eta}X=0 
\end{equation}

and the gauge transformations as:                          
\begin{eqnarray}                        
\delta X&=& ad^{*}_{\beta} X \label{Ca}\\ 
\delta \eta&=& \triangledown_{\eta}\beta,
\end{eqnarray}
with $X \in \Gamma(E), \beta \in \mathcal{C}^{\infty}(I, \mathfrak{g}).$\\
From equation (\ref{Coi}), observe that a solution is completely characterized with the choice of the 1-form $\eta$ and the initial condition $X(0)$. 
Therefore,
\[C_{\Pi}= \Omega^1(I, \mathfrak{g})\otimes \mathfrak{g}^{*}.\]
Now, in order to describe the spaces $L_i$, we observe that the infinitesimal gauge transformations leave invariant the initial point 
of the paths in $\mathfrak{g}^{*}$. Therefore, the gauge transformations only change the 1-form $\eta$ by an exact form and hence
$L_1$ corresponds to the space of connections gauge equivalent to the trivial connection on $E$. 
The idea is to prove the following:
\begin{proposition} \emph{Let $L_e^c(G)$ denote the space of connected loops starting and ending at the identity $e$. Then $L_e^c(G)\cong L_1.$}
\end{proposition}
\begin{proof} Consider the map $\phi\colon L_e^c(G) \to L_1 $ defined by \[\phi(\gamma)= \gamma^{-1} d \gamma.\] where 
$\gamma\colon  I \to G$ and $\gamma^{-1}(t)= (\gamma(t))^{-1}.$ First of all, observe that this map is surjective since for this example, 
the global symmetries come from paths in the Lie group $G$ in the connected component of the identity. The change of the 
connection $\eta$ along $\gamma$ is given by:
\[ \eta(\gamma)= \gamma^{-1} d\gamma + Ad_{\gamma} \eta.\]
Setting $\eta=0,$ we obtain all the gauge equivalent connections to the trivial one.
To prove that $\phi$ is injective let consider to paths $\gamma$ and $\tilde{\gamma}$ 
in such way that $\gamma^{-1} d \gamma = \tilde{\gamma}^{-1} d \tilde{\gamma}.$ Define $\tau:= \tilde{\gamma}^{-1} \gamma.$ Differentiating this equation   we get:
\begin{eqnarray*}
d\tau&=& d\tilde{\gamma}^{-1} \gamma + \tilde{\gamma}^{-1} d\gamma \\
     &=& -\tilde{\gamma}^{-1}d\tilde{\gamma}\tilde{\gamma}^{-1} \gamma + \tilde{\gamma}^{-1}d\gamma\\
     &=& -\gamma^{-1}d\gamma \tilde{\gamma}^{-1} \gamma + \tilde{\gamma}^{-1} d\gamma\\
     &=& \gamma^{-1} d\gamma (e- \tilde{\gamma}^{-1}\gamma)\\
     &=& \gamma^{-1}d\gamma (e- \tau).
\end{eqnarray*}
We have obtained a linear ODE, where $\gamma$ is independent of $\tau$ and it satisfies the Liptchitz condition of uniqueness. Since $\tau \equiv e$ 
satisfies the equation it is possible to conclude that $\gamma= \tilde{\gamma}$.  \\
\end{proof}
To describe $L_2$, note that the gauge transformations are given by paths $\gamma$ on the connected component of the identity and it can be proven 
(see \cite{Cat}) that the transformation only depends on the final point of $\gamma$. More precisely, defining the holonomy map 
\[\mbox{Hol}\colon \Omega^1(I, \mathfrak{g})\to G \]
as the the parallel transport of $e$ on the connection given by $\eta$ (using the flat section $X$), we get that
\[L_2\cong \{(g_1,g_2,\eta)\in \mathfrak{g}^{*}\otimes \mathfrak{g}^{*} \otimes  \Omega^1(I, \mathfrak{g})\mid \eta=d\beta, \beta(0)=\beta(1)=0\}. \] \\
In a similar way, it is possible to describe $L_3$ in terms of the holonomy map and, therefore we get  the following relational symplectic groupoid:

\begin{eqnarray*}
\mathcal{G}&=& T^*P\mathfrak{g}^{*}.  \\
L&=& \{(g_1,\eta_1), (g_2,\eta_2),(g_3,\eta_3)\mid g_i\in \mathfrak g^{*}, \eta_i \in  \Omega^1(I,\mathfrak{g}), \eta_3= d\beta, \beta(0)=\beta(1)=0\}.\\
I&=&\{(X,\eta)\mapsto (X\circ \phi, -\eta)\}.
\end{eqnarray*}
After reduction we obtain:
\begin{eqnarray*}
G&=& T^{*}\mbox{\textbf{G}}\simeq \mathfrak{g}^{*}\times \mbox{\textbf{G}}.   \\
\underline{L_1}&=& \mathfrak{g}^{*}.\\
\underline{L_2}&=& Gr(id_{T^*G}).\\
\underline{L_3}&=& Gr(m_{T^*G}), 
\end{eqnarray*}
where $m_{T^*G}$ is defined as 
\[m_{T^{*}G}((\alpha, g),(Ad^*_{g^{-1}} (\alpha), h))=(\alpha, gh).\]

\subsection{A non integrable example}\label{nonintegrable}
This is an example (see \cite{Weinstein}) of a non integrable Poisson manifold. The idea is to exhibit explicitly 
the corresponding immersed canonical relations in this example.
We consider $M= \mathbb{R}^3 / \{ 0 \}$ equipped with Poisson structure

\[\Pi^{ij}(x)= f(\lvert x \rvert) \varepsilon ^{ij}_k x^k \frac{\partial} {\partial x^i} \wedge \frac{\partial} {\partial x^j} \]

and $f$ is a function such that $f(R)> 0, \forall R >0$. It is possible to write the constraint equation as:
\[X^{'}+ f(\vert X \vert) \eta \times X =0,\]
where $X$ and $\eta$ are considered as smooth paths in $\mathbb{R}^3$ 
(after identifying $\mathbb{R}^3$ and $\mathbb{R}^{*3}$ via the Euclidean metric) and $\times$ denotes the cross product. The infinitesimal gauge 
transformations are described by (see \cite{Cat}):
\begin{eqnarray*}
\delta X&=& -f(\vert X\vert) \beta \times X.\\
\delta \eta&=& \dot{\beta} + f(\vert X \vert) \eta \times \beta + \frac{f^{'}(\vert X \vert )}{\vert X \vert} (X \cdot \eta \times \beta)X, 
\end{eqnarray*}
where $\beta \in \mathcal{C}^{\infty} (I, \mathbb{R}^3)$ and $\cdot$ is the Euclidean dot product. 
Considering the radial and tangent components of $\eta$ and $\beta$ with respect to $X$ (denoted by $\eta_r,\eta_t, \beta_r,\beta_t$ respectively), the constraint equation and the symmetries can be rewritten as:
\begin{eqnarray*}
X^{'}&+&f(\vert X \vert)\eta_t \times X=0.\\
\delta X &=& -f(\vert X \vert)\beta_t \times X.\\
\delta \eta_r&=& \dot{\beta}^{'}_r-\frac{f(\vert X \vert)}{\vert X \vert} (1- \frac{\vert X \vert \cdot f^{'}(\vert X \vert)}{f\vert X \vert})(X \cdot \eta_t \times \beta_t).\\
\delta \eta_t&=& \dot{\beta}^{'}_t+ f(\vert X \vert)\eta_t \times \beta_t + \frac{f(\vert X \vert)}{\vert X \vert ^2 }(X \cdot \eta_t \times \beta_t) X.
\end{eqnarray*}
Setting $f\equiv 1$, the constraint equations and symmetries corresponds to the linear case where, $M=\mathfrak{s}\mathfrak{u}(2)^{*}\setminus \{0\}.$ Therefore, $C_{\Pi}=\Omega^1(I,\mathfrak{su}(2))\otimes \mathfrak{su}(2)^{*}$ and in addition, 
the subspaces $L_i$ are immersed Lagrangian submanifolds of $\mathcal{G}^i=(T^{*}P\mathbb{R}^3)^i$.

Denoting $\vert X \vert$ by R and setting $C(R):= \frac{R\cdot f^{'}(R)}{f(R)},\, A(R)=\frac{4\pi R}{f(R)}$, we can redefine coordinates 
(for the case where $C(R)\neq 1$ or equivalently $A^{'}(R)\neq 0$) with tangential and radial components as follows:
\begin{eqnarray*}
a_r&=&\frac{f(\vert X \vert)}{1-C(\vert X \vert)}\eta_r,\,\, a_t=f(\vert X \vert )\eta_t,\\
b_r&=&\frac{f(\vert X \vert)}{1-C(\vert X \vert)}\beta_r,\,\, b_t=f(\vert X \vert )\beta_t.
\end{eqnarray*}
The constraints and the gauge transformations look like:
\begin{eqnarray*}
X^{'}&+&a_t \times X=0,\\
a_r^{'}&=& b^{'}_t- \frac{1}{\vert X \vert}(X \cdot a_t \times b_t),\\
a_t^{'}&=& b^{'}_t+ a_t \times b_t + \frac{1}{\vert X \vert ^2}(X\cdot a_t \times b_t)X.
\end{eqnarray*}
that is precisely, the ones controlling the linear case, where $\mathfrak g= \mathfrak {su} (2)$. Following \cite{Cat}, in the case where $A(R)=0$, since we are not allowed to use the previously defined change of coordinates, we restrict ourselves to the invariant functions 
\[x\colon =X(0), y\colon = \frac{X(1)}{\vert x \vert}, \pi:= \int_I \eta_r .\]
Thus, we conclude that for the singular points, the gauge equivalent solutions of the E-L equation cannot differ by the value of $\eta_r$ and they only depend on the initial and final values of $X$.  
Therefore, we have the following description for the spaces $L_i$
\begin{eqnarray*}
\mathcal{G}&=& T^*P(\mathbb R ^3 / \{ 0 \}).  \\
L&\cong& \{(X_1,X_2,X_3,\eta_1,\eta_2,\eta_3) \in \mathbb R ^3 \setminus \{ 0 \})^3 \times \Omega^1(I, \mathfrak{su}(2))^3\mid \eta_3=d\beta, \beta(0)=\beta(1)=0\\
&\sqcup&(X_1,X_2,X_3^{'}, \eta^{'})\in \mathbb R ^3 \setminus \{ 0 \})^3 \times \Omega^1(I, \mathfrak{su}(2))\mid \eta^{'}=d\beta, \beta(0)=\beta(1)=0\}.\\
I&=&\{(X,\eta)\mapsto (X\circ \phi, -\eta)\}.
 \end{eqnarray*}
 The reduced phase space in this case corresponds to a singular space, if $A$ vanishes at some point \cite{Cat}, therefore one obtains a topological groupoid over $M$ that is singular on the critical points of $A$.
\subsection{A non regular example}\label{torus}
The following corresponds to an example of a relational symplectic groupoid that fails to be regular. Consider the manifold
$$M= T^*(S^1 \times S^1)\equiv S^1 \times S^1\times \mathbb R \times \mathbb R$$
with local coordinates $(\theta_1, \theta_2, p_1, p_2)$ and canonical symplectic structure. Now consider the space $$C_{\alpha}:= \{(\theta_1, \theta_2, p_1, p_2)\mid p_2 = \alpha p_1\},$$
for some $\alpha$ in $\mathbb R$.
It can be easily checked that $C_{\alpha}$ is a coisotropic submanifold of $M$ and that
$$\underline  {C_{\alpha}} \mbox{ is smooth if and only if } \alpha \in \mathbb Q.$$
Now, regarding $M$ as a Poisson manifold and by choosing the  set of coordinates $(\Theta^i, P_i, \eta_i, \xi^i)$ for $T^*(PM)$, the constraint equations look like 
\begin{eqnarray}\label{constraint1}
d\theta^i&=& \xi^{i}\\
dP_i&=& -\eta_i. \label{constraint2}
\end{eqnarray}
Now, we restrict to the solutions of Equations \ref{constraint1} and \ref{constraint2} which are paths on the submanifold $C_{\alpha}$. This in coordinates reads as
\begin{eqnarray} \label{Ccons1}
P_2&=& \alpha P_1\\
\eta_2&=& \alpha \eta_1 \label{Ccons2}
\end{eqnarray}
 Therefore we can define the subspace $L_1^{\alpha}$ of $T^*(PM)$ as the space of solutions of Equations \ref{constraint1} and\ref{constraint2} such that at time $t=0,1$ the Equations \ref{Ccons1} and \ref{Ccons2} are satisfied. 
 This allows us to parametrize $L_1^{\alpha}$ as follows. We have the following relations between the coordinates we chose for $T^*(PM)$:
 \begin{eqnarray}
 \Theta^I&=& \theta^i + \int_0^t \xi^i\\
 P_i&=& p_i - \int_0^t \eta_i.
 \end{eqnarray}
 The relations
 \begin{eqnarray}
 p_2&=& \alpha p_1\\
 \int (\eta_2 - \alpha \eta_1)dt&=&0
  \end{eqnarray}
  imply that $P_2(1)= \alpha P_1(1)$
  Then we choose the parametrization $(\theta^1, p_i, \xi^i, \eta_i)$ such that for the manifold
  \[ S_1^{\alpha}= C_{\alpha} \times \Omega^1(I)\otimes \mathbb R ^2 \times \Gamma^{\alpha},\]
  where
  \[\Gamma^{\alpha}= \{(\eta_1, \eta_2)\in \Omega^1(I)\otimes \mathbb R ^2 \mid \int_I \eta_2 - \alpha \eta_1=0\}\]
  is a closed subspace of $\Omega^1(I)\otimes \mathbb R ^2$. This implies that $S_1^{\alpha}$ is a manifold and the inclusion map $\iota\colon  S_1^{\alpha} \to T^*(PM)$ is an embedding with image $L_1^{\alpha}$.
  We have the following
  \begin{proposition}\label{Lag}
  \emph{
  $L_1^{\alpha}$ is a Lagrangian submanifold of $T^*(PM)$.}
  \end{proposition}
  \begin{proof}
  The proof of this Proposition is equivalent to Proposition \ref{L1Lag}; it follows from the construction that $L_1^{\alpha}$ is isotropic and we observe that
  $$\dim \underline {TL_1^{\alpha}}= \dim \underline {C_{\alpha}}= 2= 1/2 \dim M.$$
  \end{proof}
 Now, the description of $L_2^{\alpha}$ is slightly different. We consider $T^*M$- paths$\gamma_1$ and $\gamma_2$ in $C_{\Pi}$ such that $\gamma_i(0)$ and $\gamma_i(1)$ belong to $C$ and the $T^*M$- homotopy connecting $\gamma_i$ satisfy that, in the boundary, $\gamma_t(0)$ and $\gamma_t(1)$ are Lie algebroid morphisms from $TI$ to $N^*C$, where $N^*C$ is the conormal bundle of $C$.
 For this equivalence relation $L_2^{\alpha}$ we have the following parametrization:
 \[S_2^{\alpha}= S^1\times S^1 \times \mathbb R \times V_{\alpha}\times V_{\alpha}\times \Gamma_{\alpha}\times \Gamma_{\alpha} \times \Lambda_{\alpha},\]
 where 
 $$V_{\alpha}= \{(p_1,p_2) \in \mathbb R ^2 \mid p_2=\alpha p_1 \}$$
 and
 $$\Lambda_{\alpha}= \{(\xi ^1, \xi ^2, \xi^3, \xi^4)\in \Omega^1(I)\otimes \mathbb R ^2 \mid \int_I \xi^2+ \alpha \xi^1= \int_I \xi^4+ \alpha \xi^3 \}.$$
 We get that $S_2^{\alpha}$ is an embedded submanifold of $T^*(PM)\times T^*(PM)$ whose image under the inclusion is $L_2^{\alpha}$. An argument similar to the one used in Proposition \ref{Lag} proves that $L_2^{\alpha}$ is Lagrangian. For $L_3^{\alpha}$, we obtain the following parametrization
 $$S_3^{\alpha}= S^1\times S^1 \times S^1 \times \mathbb R \times \mathbb R \times S_2^{\alpha}$$
 that is a submanifold of $T^*(PM)\times T^*(PM)\times T^*(PM)$ with image $L_3^{\alpha}$.
 This shows that, if we pick $\alpha$ an irrational number, the relational symplectic groupoid $(\mathcal G^{\alpha}, L^{\alpha}, I^{\alpha})$ given by
 \begin{eqnarray}
 \mathcal G^{\alpha}&=& T^*(PM)\\
L^{\alpha}&=& I^{\alpha}\circ L_3^{\alpha}\\
I^{\alpha}&\colon & \gamma \mapsto \gamma^{-1} 
 \end{eqnarray} 
 is non regular since $L_1^{\alpha} / L_2^{\alpha}= \underline{C_{\alpha}}$.

\section{The extension to the Fr\'echet category}
This section is devoted to discuss further generalization of the construction of relational symplectic groupoids in the context of Fr\'echet manifolds. The complications here are of analytical type, i.e. the absence of the inverse function theorem for Fr\'echet manifolds and the change of behavior of the ODE describing the E-L equation. However, studying the geometrical features of the relational symplectic groupoid for Poisson manifolds it is possible to describe the smoothness of the defining spaces of $(\mathcal G, L, I)$.
The aim is to prove the following

\begin{proposition}
\emph{The construction on $T^*PM$ of the relational  symplectic groupoid can be extended to 
the category of Fr\'echet manifolds.}
\end{proposition}
Before proving this statement, some preliminary work on Fr\'echet manifolds has to be done.

\begin{definition} (See \cite{Michor}) \emph{A Fr\'echet space is a locally convex space with a countable basis of seminorms.}
\end{definition}
An example to have in mind is the following:

\begin{example} \label{Fre} \emph{The space $\mathcal{C}^{\infty}(I, \mathbb{R})$ of $\mathcal{C}^\infty$- maps from the interval $I=[0,1]$ to the real line is equipped with a natural 
Fr\'echet structure given by the following basis of seminorms:
\[ \lVert f \rVert_k:= \sup_{x\in I}f^{(k)}(x) \]
where $f^{(k)}$ denotes the $k$-th derivative.}
\end{example}
 Observe that in this example, by truncating the sequence of seminorms, we obtain a filtration of Banach spaces 
 $$\cdots \mathcal{C}^{k}(I, \mathbb{R})\subset \mathcal{C}^{k-1}(I, \mathbb{R}) \subset \mathcal{C}^{k-2}(I, \mathbb{R})\subset \cdots \subset \mathcal{C}^{0}(I, \mathbb{R}).$$

\begin{definition} \emph{A Fr\'echet manifold corresponds to a topological Hausdorff space $X$ equipped to an atlas $(X,U)$ in such way 
that there are homeomorphims $\phi_{\alpha}\colon  U_\alpha \to F$ where $F$ is a Fr\'echet space and the transition functions $F_{\alpha \beta}$ are smooth mappings with respect to the Gateaux derivative \cite{Michor}.}

\end{definition}
\begin{definition}\emph{
We define the categroid $\mbox{\textbf{Symp}}^{Ext}_{F}$ which objects are Fr\'echet manifolds equipped with a weak symplectic structure and a morphism between two objects $(M, \omega_m)$ and $(N, \omega_N)$ is a pair $(L, \{(L_i, \phi_i)\}_{i\in \mathbb N})$ such that
\begin{enumerate}
\item If we truncate the sequence of seminorms for $M$ and $N$ up to the index $i$, the underlined sets denoted by $M_i$ and $N_i$ respectively can be regarded as Banach manifolds.
\item $(L_i, \phi_i)\colon M_i \nrightarrow N_i$ is a morphism for $\mbox{\textbf{Sym}}^{Ext}$ as in Definition \ref{Extended}.
\item As sets, 
\[ L= \bigcap_0^{\infty}L_i\]
can be equipped with an Inverse Limit Banach (ILB) structure (for further details on ILB manifolds see \cite{Omori}),
and 
\[T L= \bigcap_0^{\infty}TL_i\]
is a Lagrangian subspace of $T(\overline M\times N)$.
\end{enumerate}
}
\end{definition}
The spaces of smooth mappings between finite dimensional manifolds are examples of objects in $\mbox{\textbf{Symp}}^{Ext}_{F}$ due to the following
\begin{theorem} (see \cite{Michor}). 
Let $M$ and $N$ be smooth finite dimensional manifolds. 
The space $\mathcal{C}^{\infty}(M,N)$ 
of all smooth mappings from $M$ to $N$ is a smooth (Fr\'echet) manifold modeled on spaces of smooth sections with compact support of pullback bundles along 
$f\colon M \to N $ over $M$, denoted by $\mathcal{C}_c^{\infty}(M \leftarrow f^*TN)$
\end{theorem} 

\begin{corollary} \emph{Let $M$ be a smooth finite dimensional (compact) manifold. Then the path space $PM$ is equipped with a Fr\'echet manifold structure. }
 
\end{corollary}

We prove the following
\begin{proposition} \label{Path}\emph{
The space $T^{*}PM$ as defined above, but now considering $X$ and $\eta$ as $\mathcal{C}^{\infty}$- maps, 
is equipped with a Fr\'echet manifold structure.}
\end{proposition}
\begin{proof}
First of all, some observations. Given the coordinate description of $T^{*}PM$, this space is the same as $PT^{*}M$, the path space 
of the cotangent bundle of $M$ and the change of coordinates is given by
\[(X,\eta) \to (t\to (X(t),\eta(t))) \]
where in the right hand side, $\eta(t) \in T^{*}_{X(t)}M$. Now, for $PT^{*}M= \mathcal{C}^{\infty}(I,T^{*}M)$, we pick a Riemannian 
metric on $T^{*}M$ and for any $\gamma \in PT^{*}M$ we construct the pull-back bundle $\gamma^{*}(TT^{*}M) \to I$. 
Now, consider $\varepsilon $ sufficiently small in such way that the exponential map, exp:  $T_{\varepsilon} \to T^{*}M$, 
where $T_{\varepsilon}$ denotes the (open) set of vectors in $\gamma^{*}(TT^{*}M)$ with norm less than $\varepsilon$ (with respect to the induced metric), 
is a diffeomorphism in its image.
This map induces an identification of $\mathcal{C}^{\infty}(I, T_{\varepsilon})$ with an open set in $PT^{*}M$.
Since $\gamma^{*}(TT^{*}M)$ is a bundle with base the unit interval, 
it is a trivializable bundle. Let $\phi\colon \gamma^{*}(TT^{*}M) \to I \times \mathbb{R}^n $ be a trivialization. 
Therefore $\phi$, restricted to $T_{\varepsilon}$ induces a map $\overline{\phi}\colon \mathcal{C}^{\infty}
(I,T_{\varepsilon}) \to \mathcal{C}^{\infty}(I,\mathbb{R})$, where the right hand side, due to Example \ref{Fre}, is a Fr\'echet space and 
therefore, $(T_{\varepsilon}, \overline{\phi})$ corresponds to a Fr\'echet chart for $T^{*}PM$ as we wanted.
\end{proof}
The next step is to establish the suitable notion of smoothness for the space of solutions of the constraint equation in the Fr\'echet context. Since we lack the inverse function Theorem to determine whether such space is equipped with a Fr\'echet manifold structure, we construct a compatible family of Banach manifolds which projects to the space of smooth solutions and are compatible with the characteristic distribution.

\begin{remark}\emph{
Let $C_{\Pi}^k$ be the Banach manifold corresponding to the space of solutions $(X, \eta)$ of regularity type $(k+1,k)$ of the equation
\begin{equation}\label{Constraint}
dX_i=\Pi^{\#}(X)^{ij}\eta_j
\end{equation}
and let $C_{\Pi}^{\infty}$of $T^{*}PM$ be the space of smooth solutions of Equation \ref{Constraint}. Then
\[ C_{\Pi}^{\infty}= \bigcap_{k=0}^{\infty} C_{\Pi}^k.\]
}
\end{remark}
The following Proposition allows us to decompose $C_{\Pi}^{\infty}$ into "leaves" coming from the characteristic foliations for $C_{\Pi}^k$.
\begin{proposition}\label{decom}\emph{
The subspace $C_{\Pi}^{\infty}$ is coisotropic in $T^*PM$ and it carries a decomposition into disjoint subspaces coming from the intersections of that leaves of the characteristic foliations, for each $k\geq 0$.}
\end{proposition}
\begin{proof}

First, observe that the coisotropicity, being a local condition, can be proven directly by using the same argument to show that $C_{\Pi}^k$ is coisotropic. However, since the use of Frobenius Theorem for involutive distributions is out of reach, we have to construct the leaf decomposition of  
$C_{\Pi}^{\infty}$ as an intersection of leaves of the underlined Banach manifolds.

As we mentioned before, the foliation charts for $C_{\Pi}$ used in the previous proofs are given by 
the integrability of the coisotropic distribution.  Here, the leaves of the characteristic foliation are precisely the orbits of the infinitesimal action of  a Lie algebra on $T^*PM$, 
following the same argument as stated in \cite{Cat}.
\subsubsection{The Hamiltonian foliation }
In the previous constructions, the foliation charts for $T^*PM$ were used in order to prove the smoothness conditions for immersed canonical relations
$L_i$. In this subsection, we discuss the geometry of the intersection of the characteristic leaves, that in general does not have to be a smooth manifold, since Frobenius 
theorem does not hold in the usual way for Fr\'echet manifolds, i.e. there are involutive distributions which are not integrable. \\

The construction goes as follows. Let $\beta\colon  I \to T^{*}M$ be a $\mathcal{C}^{\infty}$ function such that $\beta (0)= \beta (1)=0$. 
This defines a Hamiltonian function $H_\beta$ defined in the following way:
\[H_\beta= \int_0^1\langle dX+ \Pi^{\#}\eta, \beta \rangle \]

The Hamiltonian vector fields associated to these functions are given in local coordinates by:

\[ \xi_{\beta}X^i(u)= \Pi^{ij}(X(u))\beta_j(X(u),u)\]
for the coordinate in the direction of $X$ and

\[\xi_{\beta}\eta_i(u)= d_u\beta_i(X(u),u) + \partial_i \Pi^{jk}(X(u))\eta_j(u)\beta_k(X(u),u)\]

for the direction along $\eta$.\\
In \cite{Cat} it is proven that this is a distribution of closed subspaces of codimension $2n$, where each subspace leaves in $T_{\tilde{X}}T^{*}PM$, with 
$\tilde{X} \in \mathcal{C}_{\Pi}$. The idea is to prove that this argument can be extended in the case where $\beta \in \Gamma ^{\infty}(X^{*}T^{*}M)$ and such that 
$\beta(0)=\beta(1)=0.$ In order to do this, consider the linear map 
\[\beta \to \xi_{\beta}.\]
This map is injective: its kernel corresponds to the solution of the following homogeneous ODE:
\[ d_u \beta(u)= A(u) \beta(u) \]
where $A(u)_{ij}= \partial_i \Pi ^{kj}\eta_j$ and with initial conditions $\beta(0)=\beta(1)=0$, which implies that the map $\beta$ is identically zero.\\
If a vector $(\dot{X},\dot{\eta})$ belongs to the distribution generated by $\xi$, the $\eta$ direction satisfies the equation:
\begin{equation} \label{K}
\dot{\eta}(u)= d_u\beta (u)+ A(u)\beta(u). 
\end{equation}
By using the auxiliary equation
\[d_uV(u)= V(u)A(u),\]
equation (\ref{K}) can be rewritten as
\[V(u)\dot{\eta}(u)= d_u(V(u)\beta(u))\]
with initial conditions $V(0)=1$ and integrating in both sides (and using the vanishing conditions for $\beta$) we obtain that $(\dot{X},\dot{\eta})$ satisfies
\begin{enumerate} 
 \item 
\begin{equation}
\dot{X}(0)=0.
\end{equation} 
\item 
\begin{equation}\label{P}
\int_I V(u)\dot{\eta}(u)=0.
\end{equation}

\end{enumerate}
which corresponds to a system of $2n$ independent equations. Conversely, for any solution $(\dot{X}, \dot{\eta})$ 
 to the previous system, it is possible to find a $\beta$ such that $(\dot{X}, \dot{\eta})= \xi_\beta$. 
Now, here it is important to make a distinction: The distributions spanned by the vector fields $\xi_{\beta}$ depend on the regularity of 
$\beta$. Being more precise, let $D^0$ be the distribution generated by the vector fields $\xi_{\beta^0}$, where $\beta^0 \in \mathcal{C}_0^1(I, X^{*}(T^{*}M))$ and 
$D^{\infty}$ the one generated by the vector fields $\xi_{\beta^{\infty}}$, 
where $\beta^{\infty} \in \mathcal{C}^{\infty}(I, X^{*}(T^{*}M))$. We distinguish in a similar way between $\mathcal{C}^0_{\Pi}$ and 
$\mathcal{C}_{\Pi}^{\infty}$ depending wether we consider smooth solutions of the constraint equation or not.\\
We prove the following
\begin{proposition} \label{inter} \emph{With the previous notation,
\[D^{\infty}= D^0 \cap TC^{\infty}_{\Pi}\]}
\end{proposition}
\begin{proof}
The fact that $D^{\infty} \subset D^0 \cap TC^{\infty}_{\Pi} $ 
is easy to check. For the other direction, it is sufficient to show that $D^{\infty}$ is tangent to $D^{0}$. In order to do this, we will use 
the following lemma, proven in \cite{Cat}:

\begin{lemma} \label{gauge}\emph{ Every leaf in the characteristic foliation of $C^{k}_{\Pi}$ admits a smooth representative, i.e, for every solution 
$(X,\eta) \in C^{k}_{\Pi}$ there exists a solution $(X,\eta)^{\infty} \in C^{\infty}_{\Pi}$ gauge equivalent to it.}
\end{lemma}

For the proof of this Lemma, a sequence of smooth gauge transformations is constructed in a $\mathcal{C}^{0}$ 
neighborhood of $(X,\eta) \in C^{0}_{\Pi}$, dividing the path $X$ in an odd number of subintervals and constructing in each one of them a
parameter $\beta$ which generates the gauge transformation.\\
Now, by means of this Lemma, it is possible to show that

\begin{lemma}\label{path}\emph{
Let $(X_0,\eta_0)^{\infty}$ and  $(X_1,\eta_1)^{\infty}$ be to points of $C^{0}_{\Pi}$ in the same 
leaf of the characteristic foliation.Then, there exists a path $\gamma\colon  I \to C^{0}_{\Pi} $ such that $\gamma(0)=(X_0,\eta_0)^{\infty},\, 
\gamma(1)=(X_1,\eta_1)^{\infty}$ and $\gamma(t) \in C_{\Pi}^{\infty}, \forall t \in I.$}
\end{lemma}

\begin{proof}First, we pick a path $\tilde{\gamma}$ in $C_{\Pi}^{0}$ with initial and final points the given solutions. 
Without loss of generality, we assume that $(X_1,\eta_1)^{\infty}$ belongs to the same $\mathcal{C}^0$- neighborhood 
of $(X_0,\eta_0)^{\infty}$, otherwise we partition the path into subintervals, 
each of them living in a $\mathcal{C}^{0}$- neighborhood. Using the notation of lemma \ref{gauge}, the path $\gamma$ is constructed 
in the following way:
\[\gamma(t):= \beta (\tilde{\gamma}(t))\]

where $\beta (\cdot)$ denotes the smooth gauge transformation which is defined in \cite{Cat}. In this way we prove that $D^0 \cap TC^{\infty}_{\Pi} \subset D^{\infty}$, as we 
wanted.
\end{proof} 
Thus, Proposition \ref{inter} follows from Lemma \ref{gauge} and Lemma \ref{path}.
\end{proof}
It also follows that the tangent spaces to the `leaves" of $C_{\Pi}^{\infty}$ are generated by $D^{\infty}$, which completes the proof of Proposition \ref{decom}.
\end{proof}


From the previous facts, we can easily check the following 
\begin{proposition}\emph{
Denoting $(\mathcal G_k, L_k, I_k)$ the regular relational symplectic groupoids defined fixing the regularity type $(k+1,k)$ for $(X,\eta)$, then the data $(\mathcal G_{\infty}, L_{\infty}, I_{\infty})$ where 
\begin{enumerate}
\item $\mathcal G_{\infty}:= T^*PM$
\item $L_{\infty}:= \bigcap _0^{\infty} L_k$
\item $I_{\infty}:= \bigcap_0^{\infty} I_k$
\end{enumerate}
corresponds to a relational symplectic groupoid over $\mbox{\textbf{Sym}}^{Ext}_{F}$.
}
\end{proposition}

\chapter{Categorical and algebraic features}\label{Frobenius}

The objective on this chapter is to introduce several constructions for more general categories (not just $\mbox{\textbf{Symp}}^{ext}$ or $\mbox{\textbf{Symp}}^{ext}_F$), which resemble the construction of relational symplectic groupoids. It turns out that some of the algebraic axioms defining $(\mathcal G, L,I)$ appear as natural generalizations of the axioms defining Frobenius algebras, monoid structures or H*-algebras, and they correspond, in some way, to the version `before reduction" of such structures, in the same way that relational symplectic groupoids appear as symplectic groupoids before symplectic reduction.\\

This discussion yields to the case of relational symplectic groupoids over linear spaces that is an intermediate space towards the quantization of Poisson manifolds via relational symplectic groupoids. (See Section \ref{Perspectives}).\\

In the sequel we consider a category $\mathcal C$ which admits products and  with a final object denoted by $pt$. 
\section{Weak monoids}
\begin{definition}\emph{ A weak monoid in $\mathcal C$ corresponds to the following data:
\begin{enumerate}
\item An object $X$.
\item A morphism $L_1\colon  pt \to X$
\item A morphism $L_3\colon  X\times X \to X,$
\end{enumerate}
satisfying the following axioms
\begin{itemize}
\item (Associativity). $$L_3 \circ (L_3 \times Id)=L_3 \circ (Id \times L_3)$$
\item (Weak unitality). $$L_3\circ(L_1 \times Id)= L_3\circ (Id \times L_1)=:L_2$$
and $L_2\circ L_2=L_2$.
\end{itemize}
We call $L_1$ a weak unit and $L_2$ a projector.}
\end{definition}
\begin{example}\emph{
Any monoid object in $\mathcal C$ is a weak monoid with $L_1$ being the unit and $L_2$ being the identity morphism.}
\end{example}
\begin{example} \emph{
As we will se later, any relative Frobenius algebra $X$ in \textbf{Rel} (see Section \ref{Frobe}) is by definition a weak monoid.}
\end{example}

\begin{example}\emph{
A commutative monoid $(X,m,1)$ equipped with a special element  $p$ such that $m(p,p)=1$, can be made into a weak monoid. In this case,  $L_3=m$, $L_1=p$ and $L_2\colon  x \mapsto m(p,x)$. Since in general $L_2$ is not the identity morphism, this is not an example of an usual monoid, but for a commutative monoid in \textbf{Set} it can be checked that the quotient $X / L_2$ is a monoid.}
\end{example}
\begin{remark} \emph{
The last example does not yield in general a monoid if we start with a commutative monoid in a category different from \textbf{Set}. For instance, if we take the monoid $(\mathbb R, \cdot, 1)$ and the special element $p=-1$, the quotient space $\underline X=[0, \infty)$ is a monoid object in \textbf{Set} but it is not an object in \textbf{Man}, the category of smooth manifolds and smooth maps. However the weak monoid $(\mathbb R, \cdot, 1)$ is a smooth manifold.}

\end{remark}

\section{Weak *-monoids}
As a special case of weak monoids are those which are compatible with adjoints. 

\begin{definition}\emph{
Let $\mathcal C$ be a dagger category, that is, a category equipped with an involutive, identity on objects functor $\dagger: \mathcal C^{op}\to \mathcal C$,  which has also products, adjoints and a special object $pt$.
A weak *-monoid in $\mathcal C$ consists of the following data:
\begin{enumerate}
\item An object $X$
\item A morphism $L_3\colon  X \times X \to X$
\item A morphism $\psi\colon  X \to X^{\dagger}$
\end{enumerate}
such that the following axioms hold
\begin{itemize}
\item (Associativity). $$L_3 \circ (L_3 \times Id)=L_3 \circ (Id \times L_3)$$
\item (Involutivity). $\psi^{\dagger}\psi =\psi=\psi^{\dagger}= \Id$
\item  Defining $\psi_R$ the (unique) induced morphism $\psi_R\colon  pt \to X \times X$, then
$$L_1:=L_3\circ \psi_R$$ determines a weak monoid $(X,L_1,L_3)$
\end{itemize}
}
\end{definition}

\begin{example}
\emph{
Consider $\mathcal C$ the category $\mbox{\textbf{Vect}} ^{Ext}$ of vector spaces (possibly infinite dimensional) whose morphisms are linear subspaces. The dagger structure is the identity in objects and the relational converse for morphisms. Let $\phi$ be an involutive diffeomorphism of $M$. If $X = \mathcal C^{\infty} (M)$, then $(X, +, \phi^*)$ is a weak *-monoid.
To check this, first observe that 
\begin{eqnarray*}
L_1&=&\{f + \phi^*(f), \, f \in X\}\\
L_2&=&\{(g, g+h+\phi^*h), \, g,h \in X \}\\
L_2 \circ L_2 &=&\{(g, g+h+h^{'}+\phi^*h+\phi^*h^{''}), \, g,h, h^{'} \in X \}.
\end{eqnarray*}
Setting $h^{'}\equiv 0$ we get that $L_2 \subset L_2 \circ L_2$ and by linearity of $\phi$ $L_2 \circ L_2 \subset L_2$. Associativity and unitality follow from the additive structure of $X$.
}
\end{example}

\begin{example}
\emph{(Deformation quantization).\label{Def}
Let $\mathcal C=\mbox{\textbf{Vect}} ^{Ext}$ and consider a Poisson manifold $M$. Let $X= \mathcal C ^{\infty}(M, \mathbb C)$ be the algebra of smooth complex valued functions on $M$.  By deformation quantization for Poisson manifolds (see, for example, \cite{Bayen}), given a Poisson structure $\Pi$ on $M$, there exists an associative $\mathbb C [\varepsilon]]$- linear product in $X [[ \varepsilon]]$ \cite{Kontsevich1}, denoted by $\star$, such that \footnote{in this case that we are considering complex valued functions we set $\varepsilon= i\hbar / 2$}
\begin{enumerate}
\item $1\star f= f\star 1 =f, \forall f \in X[[ \varepsilon]] $
\item $$f\star g= fg + \varepsilon B_1(f,g)+ \varepsilon^2 B_2(f,g)+\cdots,$$
with $f,\, g \in X \subset X[[\varepsilon]] $ and $B_i$ are bidifferential operators, where 
$$\Pi(df,dg)= B_1(f,g).$$
\end{enumerate}
It can be easily checked that $(X[[ \varepsilon]], \star, \overline{\cdot})$ is a weak-* monoid, where $\overline{\cdot}$ denotes complex conjugation.
}
\end{example}
\section{Cyclic weak *-monoids}
This is the type of weak *-monoids which are compatible with the cyclicity axiom for relational symplectic groupoids. More precisely,
\begin{definition}\emph{
Let $\mathcal C$ be  a dagger category with products and adjoints and such that every object $X$ admits a \emph{canonical pairing} $\phi:= X\times X^{\dagger}\to pt$.
A cyclic weak *-monoid in $\mathcal C$ consists of the following data:
\begin{enumerate}
\item An object $X$
\item A morphism $\psi\colon  X \to X^{\dagger}$
\item A morphism $L\colon  X \times X \to X^{\dagger}$
\end{enumerate}
such that
\begin{itemize}
\item (Cyclicity). For the associated morphism $L_R: (\phi_X\circ (L\times \Id)^t)\colon  pt \to X^3$
$$L_R= \sigma \circ L_R= \sigma\circ \sigma \circ L_R$$
where 
\begin{eqnarray}
\sigma\colon  X^3 &\to& X^3\\
(a,b,c)&\mapsto& (c,a,b)
\end{eqnarray}
\item If $L_3:= \psi ^{\dagger} \circ L$, then $(X, \psi, L_3)$ is a weak *-monoid.
\end{itemize}
}
\end{definition}

\begin{remark}
\emph{
For the case of a unimodular Poisson manifold $(M, \Pi)$, it can be proven, following Example \ref{Def} that deformation quantization gives rise to a cyclic weak *-monoid, where $\psi= \phi^*$, with $\phi$ being a $\Pi$- invariant diffeomorphism. This is part of some work in progress and is related to work on 
traces in deformation quantization (Section \ref{Perspectives}).
}
\end{remark}
\begin{example}
\emph{(Relational symplectic groupoids). Following Chapter \ref{PSMmain}, we consider $\mathcal C= \mbox{\textbf{Symp}}^{ext}$ and $M$ an arbitrary Poisson manifold.
\begin{proposition} \emph{The following data
\begin{eqnarray*}
X&:&= T^*(PM)\\
\psi&\colon & (x, \eta) \mapsto (i^* \circ x,i^* \circ \eta)\\
&\mbox{                   }& i: t \mapsto 1-t\\
L&\colon &= \{(x_1, \eta_1), (x_1, \eta_1),(x_3, \eta_3) \vert (x_1* x_2, \eta_1* \eta_2) \sim \psi((x_3, \eta_3))\},
\end{eqnarray*}
where $\sim$ denotes the equivalence relation by $T^*M$- homotopy of $T^*M$-paths,
corresponds to a cyclic weak $*$- monoid.
In this case,
\begin{eqnarray*}
L_1&=&\{(x, \eta) \in X \vert (x,\eta)\sim (x\equiv x_0, \eta \equiv 0), x_0 \in M\}\\
L_2&=&\{(x_1, \eta_1), (x_2, \eta_2) \in X \times X \vert (x_1,\eta_1)\sim (x_2, \eta_2)\}.
\end{eqnarray*}
}
\end{proposition}
}
\end{example}

\begin{example} \label{functions}\emph{ (The space of functions on a manifold).
For this example, we consider $M$ a smooth manifold and $\phi$ a given diffeomorphism of $M$ such that $ \phi^2 =\Id $.
We consider the vector space $\mathcal G$ of smooth functions on $M$ equipped with the product
$$f *g := 1/4 (fg+f \phi^*g + \phi^* f g + \phi^* f \phi^* g)$$
and the special element 1.
In turns out that $(\mathcal G, *, \phi^*)$ has the structure of a cyclic weak *-monoid. Moreover, the induced morphism $L_2$ corresponds to the projection to the $\phi$-even functions. \\
We obtain also that $(\mathcal G, *, \phi^*)$ can be equipped with a Frobenius algebra structure, considering the inner product 
$$\langle f,g \rangle:= \int_M fg.$$
}
\end{example}

\section{ Frobenius and H*-algebras}\label{Frobe}
The rest of this Chapter is devoted to some results concerning the relation between groupoids and Frobenius algebras. The generalization of these results turns out to give 
an algebraic characterization of the relational symplectic groupoid and its connection to Frobenius and $H^*-$ algebras. \\

In \cite{FrobeniusChris}, the connection between groupoids and Frobenius algebras is made precise. Namely, there is a way to understand groupoids in the category \textbf{Set} as what we called \emph{Relative Frobenius algebras}, a special type of dagger Frobenius 
algebra in the category \textbf{Rel}, where the objects are sets and the morphisms are relations.\\
In addition, a correspondence between semigroupoids 
(a more relaxed version of groupoids where the identities or inverses do  not necessarily exist) and $H^*-$ algebras, a structure similar to Frobenius algebras 
but without unitality conditions and a more relaxed Frobenius relation. 
\section{Groupoids and relative Frobenius algebras}
In this section, we consider a groupoid in \textbf{Set} as a category internal to the category \textbf{Set} of sets as objects and maps as morphisms.
Now, consider the category \textbf{Rel} with sets and relations. In addition, this category carries an involution 
$\dagger\colon \mbox{\textbf{Rel}}^{op} \to \mbox{\textbf{Rel}}$ given by the transpose of relation;  this is a contravariant involution 
and is the identity in objects, therefore, \textbf{Rel} is a dagger symmetric monoidal category that contains \textbf{Set} as a subcategory.
In \textbf{Rel} we define what we call \textit{relative Frobenius algebra}, a special dagger Frobenius algebra.
\begin{definition}
\emph{A morphism $m\colon  X\times X \nrightarrow X$ in \textbf{Rel} is called a special dagger Frobenius algebra or shortly, relative Frobenius algebra, if it satisfies the following axioms
\begin{itemize}
 \item (F) $(1\times m) \circ (m^{\dagger} \times 1)= m^{\dagger} \circ m= (m\times 1) \circ (1 \times m^{\dagger}),$
\item (M) $m\circ m^{\dagger}=1,$
\item (A) $m\circ (1 \times m)= m \circ (m\times 1),$
\item (U) $\exists u\colon  1 \nrightarrow X \vert m\circ (u \times 1)=1= m\circ(1\times u)$.
\end{itemize}
} 
\end{definition}
\begin{remark}
\emph{If such $u$ exists, it is unique.} 
\end{remark}
\begin{figure}
\psfrag{Rq}{$\mathbb{R}^q$}
\centering%
\center{\includegraphics[scale=0.55]{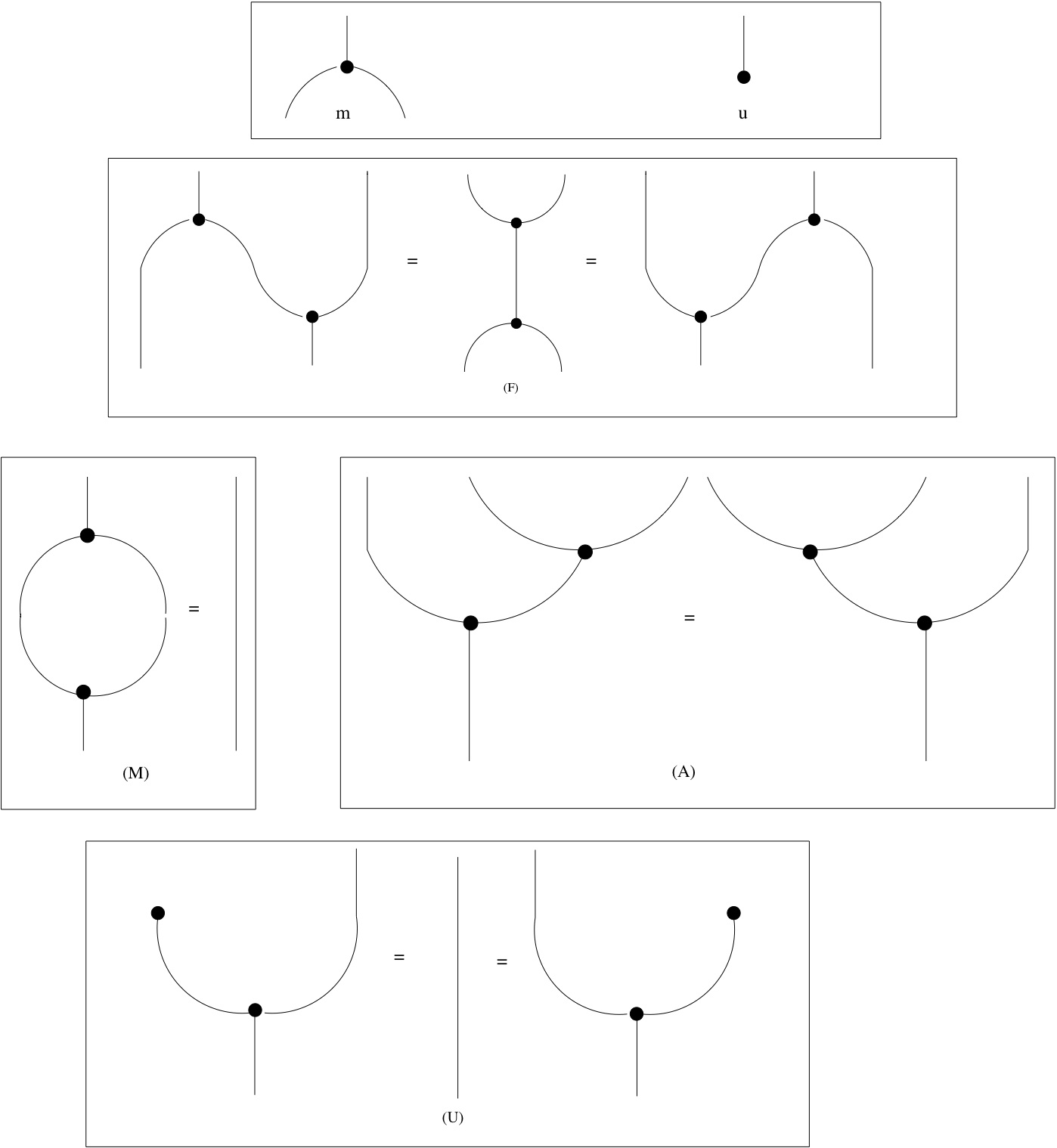}}
\caption{Relative Frobenius algebra: Diagrammatics}
\label{fig:FigureFrobenius}
\end{figure}

\subsection{From relative Frobenius algebras to groupoids}
Here, from a given relative Frobenius algebra we construct a groupoid, but first of all, we give precise meaning of the axioms defined above. We will use the notation $f=hg$ when $((h,g)f)\in m$. First of all, observe that axiom (M) implies that $m$ is single valued  and 
that 
\[\forall f \in X \, \exists g,\, h \in X \vert f=hg.\]
The axiom (F) means that for all $a,\,b,\,c,\,d\, \in X$
\[ab=cd \Longleftrightarrow \exists \, e \in X \vert b= ed,\, c= ae \Longleftrightarrow \exists \, e \in X \vert d= eb,\, a= ce.\]
The axiom (A) is associativity, i.e. $(fg)h= f(gh)$. For the last axiom, after identifying the morphism $u\colon  1 \nrightarrow X$ with a subset 
$U \subseteq X$, we get that (U) is equivalent to the following assertions
\begin{center}
\begin{eqnarray*}
\forall f \in X &\exists& \, u \in U \vert fu= f\\
\forall f \in X &\exists& \, u \in U \vert uf= f\\
\forall f \in X  &\forall& \,u \in U \vert f \mbox { and } u \mbox{ are composable} \Longrightarrow fu=f\\
\forall f \in X  &\forall& \,u \in U \vert u \mbox { and } f \mbox{ are composable} \Longrightarrow uf=f.
\end{eqnarray*}
\end{center}

From these data, we are able to give explicitly a groupoid in \textbf{Set}.
\begin{definition}\label{groupoid}
\emph{ Given a relative Frobenius algebra $(X,m)$, we define the following objects and morphisms in \textbf{Rel}:}
\begin{eqnarray*}
G_1&=& X,\\
G_2&=&\{(g,f) \in X^2\vert g \mbox { and } f \mbox{ are composable}\},\\
G_0&=&U,\\
\varepsilon&=& U \times U\colon  G_0 \nrightarrow G_1,\\
s&=&\{(f,x) \in G_1 \times G_0 \vert  f \mbox { and } x \mbox{ are composable}\}\colon  G_1 \nrightarrow G_0\\
t&=&\{(f,y) \in G_1 \times G_0 \vert  y \mbox { and } f \mbox{ are composable} \}\colon  G_1 \nrightarrow G_0\\
\iota&=& \{(g,f) \in G_2 \vert gf \in G_0, \, fg \in G_0 \}: G_1 \nrightarrow G_1.
\end{eqnarray*}
\end{definition}
We will prove that the following data 
\\
\xymatrixrowsep{4pc} \xymatrixcolsep{3pc} \xymatrix{
    &\,\,\,\,\;\;\,\,\;\,G_2  \ar[r]^{\,\,\,\,\,\;\;\;\;\;\;\;\;m}  & G_1\ar[r]^{\iota}   &G \ar@/_/[r]_t  \ar@/^/[r]^s & G_0 \ar[l]_{\varepsilon}  & 
    }
\\ 
correspond to a groupoid in \textbf{Set}.

First we prove the following lemmas
\begin{lemma}\emph{ For $f \in X$ and $u,v \in U:$
\begin{enumerate}
 \item if $f$ and $u$ are composable, then $u$ is composable with itself;
\item if $f$ and $u$ are composable and $f$ and $v$ are composable then $u$ and $v$ are composable;
\item if if $f$ and $u$ are composable and $f$ and $v$ are composable then $u=v$.
\end{enumerate}
}
\end{lemma}
\begin{proof} If $f$ and $u$ are composable, then, by axiom (U), $fu=f$, therefore, $(fu)u=f$. By (A), we get that $u$ is composable with itself, 
as we want in (1). For (2), we assume that $f$ and $u$ are composable, as well as $f$ and $v$. Then $fu=f=fv$ and by axiom (F) we have that $u= ev$, 
for some $e \in X$, hence $uv= ev^2$ (in particular, $e$ and $v^2$ are compatible). For (3), observe that if $f$ and $u$ are compatible, as well as $f$ and $v$, then, 
by axiom (U), $u=uv=v$.  
\end{proof}
\begin{corollary}
\emph{$s$ (and in a similar way $t$) is (the graph of) a function.} 
\end{corollary}
\begin{lemma}\emph{
The pullback $G\times_{(s,t)}G$ is isomorphic (as sets) to $G_2$.} 
\end{lemma}
\begin{proof}
We have that $G\times_{(s,t)}G= \{(g,f) \in X \vert s(g)=t(f)\}.$ Also, $s(g)$ is the unique $y \in U$ such that $g$ and $y$ 
are composable and similarly, $t(f)$ is the unique $y^{'}\in U$ such that $y^{'}$ and $f$ are composable.  
\end{proof}

\begin{lemma}
 \emph{The following diagram in \textbf{Rel} commutes
\[\xymatrix{
    G_1 \ar_-{\Delta}[d] \ar^-{s}[rr] && G_0 \ar^-{e}[d] \\
    G_1 \times G_1 \ar_-{1 \times i}[r] & G_1 \times G_1 \ar_-{m}[r] & G_1
  }\]
  Here, $\Delta$ is (the graph of) the diagonal function $x \mapsto (x,x)$.}
\begin{proof}
 We have that 
\begin{align*}
    e \circ s 
    &=& \{(f,g) \in G_1 \times G_1 \mid \exists u \in U .\, g=u, f\mbox{and }u\mbox{ are composable}  \}\\
    &=& \{ (f,u) \in X \times U \mid f\mbox{and }u\mbox{ are composable} \},
\end{align*}
  and
  \begin{align*}
    m \circ (1 \times i) \circ \Delta 
    & = \{ (f,h) \in G_1^2 \mid \exists g \in G_1 .\, fg \in U, gf \in U, h=gf \}\\
    & = \{ (f,gf) \in G_1^2 \mid g \in G_1 ,\, fg \in U \ni gf \}.
  \end{align*}
\end{proof}
\end{lemma}

\begin{lemma}
\emph{The relation $i$ is (the graph of)  a function.} 
\end{lemma}

\begin{proof}
  We need to prove that to each $f \in X$ there exists a unique $g \in X$ with $gf \in U \ni fg$.
  We already have existence of such a $g$ by the previous lemma, now we will prove uniqueness. 
Assume that $gf \in U \ni fg$ and $g'f \in U \ni fg'$. Then $f$ and $g$ are composable, as well as $g$ and $f$, 
so that by (A) also $f$, $g$ and $f$ are composable, similarly for  $f$, $g'$ and $f$. 
Now (U) implies $fgf=f=fg'f$, so that by the previous conjecture $gf=g'f$. But
  then $g=gfg=g'fg=g'$.  
\end{proof}
Finally, after some straightforward verification of the rest of the groupoid axioms, it is possible to state the following 
\begin{theorem}\label{Fro}
If  $(X,m)$ is a relative Frobenius algebra, then the object stated in Definition \ref{groupoid} is a groupoid in \textbf{Set}.
\end{theorem}

\subsection{From groupoids to relative Frobenius algebras}
Here we fix a groupoid 
\\
\xymatrixrowsep{4pc} \xymatrixcolsep{3pc} \xymatrix{
    &\,\,\,\,\;\;\,\,\;\,G_2  \ar[r]^{\,\,\,\,\,\;\;\;\;\;\;\;\;m}  & G_1\ar[r]^{\iota}   &G \ar@/_/[r]_t  \ar@/^/[r]^s & G_0 \ar[l]_{\varepsilon}  & 
    }
\\ 
in \textbf{Set}. 
\begin{definition}\emph{
  For a groupoid $G_1$, define $X=G_1$, and let $m \colon G_1 \times G_1 \nrightarrow G_1$ be
  the graph of the function $m$.}
\end{definition}

We will prove that $m$ is a relative Frobenius algebra. 
First of all,  it follows directly from associativity of composition in the groupoid that $m$ satisfies (A).

\begin{lemma}\emph{
  The morphism $m$ of \textbf{Rel} satisfies the axiom (U).}
\end{lemma}
\begin{proof}
  Define a relation $u \colon 1 \nrightarrow X$ by $u = \{ (*,e(x)) \mid x \in G_0 \}$.  Then
  \begin{align*}
    m \circ (u \times 1) 
    & = \{ (f,e(x)f) \mid f \in G_1, x=t(f) \in G_0 \} \\
    & = \{ (f,et(f)f) \mid f \in G_1 \} = 1.
 \end{align*}
  The symmetric condition also holds, and so (U) is satisfied.
\end{proof}

\begin{lemma}\emph{
  The morphism $m$ of $\Cat{Rel}$ satisfies (M).
 }
\end{lemma}
\begin{proof}
  We have
  \[
    m \after m^\dag = \{ (f,f) \in G_1^2 \mid \exists g,h \in G_2 .\, s(h)=t(g) ,\, f=hg \}.
  \]
  Because we can always take $g=f$ and $h=e(t(f))$, this relation is equal to $\{(f,f) \in
  G_1^2 \mid f \in G_1\}=1$. Thus (M) is satisfied.
\end{proof}

\begin{lemma}\emph{
  The morphism $m$ of $\Cat{Rel}$ satisfies (F).}
\end{lemma}
\begin{proof}
  First compute
  \begin{align*}
    m^\dag \after m & = \{ ((a,b),(c,d)) \in G_2^2 \mid ab=cd \}, \\
    (m \times 1) \after (1 \times m^\dag) & = \{ ((a,b),(c,d)) \in G_2^2 \mid \exists e \in G_1 .\, ed=b,\, ae=c \}.
 \end{align*}
  If $ed=b$ and $ae=c$, then $cd=aed=ab$. Hence $(m \times 1) \after (1 \times m^\dag)
  \subseteq m^\dag \after m$. Conversely, suppose that $((a,b),(c,d)) \in m^\dag \after m$.
  Taking $e=bd^{-1}$, then $ed=bdd^{-1}=b$, and $ae=abd^{-1}=cdd^{-1}=c$. Therefore
  $m^\dag \after m \subseteq (1 \times m^\dag) \after (m \times 1)$. The symmetric
  condition is verified analogously. Thus (F) is satisfied.
\end{proof}

Therefore, we have proven the following 
\begin{theorem}\label{gpdisfrob}
  If $\cat{G}$ is a groupoid, then $m$ is a relative Frobenius algebra.
  \qed
\end{theorem}

\subsection{Functoriality}\label{subsec:frob:functoriality}

Notice that the constructions $m \mapsto \cat{G}$ and $\cat{G} \mapsto
m$ of the previous two sections are each other's inverse. This subsection
proves that the assignments extend to an isomorphism of
categories under various choices of morphisms: one that is natural for
groupoids, one that is natural for relations, and one that is natural
for Frobenius algebras. 

Recall that the category $\cat{Rel}$ is rigid (with $\overline{X}=X$) and therefore in
particular closed
. Hence every
morphism $r \colon X \relto Y$ has a transpose $\name{r} \colon 1
\relto \overline{X} \times Y$. Explicitly, it is given by $\name{r}
= \{(*,(x,y)) \mid (x,y) \in r\}$. Furthermore $\cat{Rel}$ 
is dagger symmetric monoidal closed, giving a natural swap isomorphism $\swapmap
\colon X \times Y \to Y \times X$ with $\swapmap^{-1}=\swapmap^\dag$.

We start by considering a choice of morphisms that is natural from the
point of view of relations: namely, morphisms between groupoids should
be subgroupoids of the product.

\begin{definition}\emph{
  The category $\Cat{Frob}(\Cat{Rel})\ext$ has relative Frobenius
  algebras as objects. A morphism $(X,m_X) \to (Y,m_Y)$ is a morphism
  $r \colon X \relto Y$ in $\Cat{Rel}$ satisfying
  \begin{equation}
    (m_X \times m_Y) \after (1 \times \swapmap \times 1) \after
    (\name{r} \times \name{r}) = \name{r}.
    \tag{R}   
  \end{equation}
  }
\end{definition}

This gives a well-defined category. Identities $1_X = \{(x,x) \mid x
\in X \} \colon X \relto X$ satisfy (R), and if $r \colon X \relto Y$
and $s \colon Y \relto Z$ satisfy (R), then so does $s \after r$:
\begin{align*}
         \name{s \after r}
  & = \{(*, x'',z'') \in 1 \times X \times Z \mid \exists y'' \in Y
          .\, (x'',y'') \in R, (y'',z'') \in s \} \\
  & = \{(*, xx', zz') \mid x,x' \in X, z,z' \in Z, \exists y,y' \in Y . \\
  & \qquad\qquad\qquad\quad\; (x,y) \in r, (x',y') \in r, (y,z) \in s, (y',z') \in s \} \\
  & = (m_X \times m_Z) \after (1 \times \swapmap \times 1) \after
         (\name{s \after r} \times \name{s \after r}),
\end{align*}
since we may take $x=x'', x'=1$.

\begin{definition}\emph{
  The category $\Cat{Gpd}\ext$ has groupoids as objects. Morphisms $\cat{G}
  \to \cat{H}$ are subgroupoids of $\cat{G} \times \cat{H}$.}
\end{definition}

That this is a well-defined category will follow from the following
theorem. Identities are the diagonal subgroupoids, and composition of
subgroupoids $\cat{R} \subseteq \cat{G} \times \cat{G}'$ and $\cat{S}
\subseteq \cat{G}' \times \cat{G}''$ is the groupoid $S_1 \after R_1
\rightrightarrows S_0 \after R_0$.

\begin{theorem}\label{gpdfrobext}\emph{
  There is an isomorphism of categories $\Cat{Frob}(\Cat{Rel})\ext
  \cong \Cat{Gpd}\ext$.}
\end{theorem}
\begin{proof}
  Let $(X,m_X)$ and $(Y,m_y)$ be relative Frobenius algebras, inducing
  groupoids $\cat{G}$ and $\cat{H}$.
  First, notice that if $r \colon X \relto Y$ satisfies (R), then
  \begin{align*}
    m_r & = (m_X \times m_Y) \after (1 \times \swapmap \times 1) \\
    & = \{ (((a,b),(c,d)),(ac,bd)) \mid a,b,c,d \in X \}
    \colon r \times r \relto r
  \end{align*}
  is a well-defined morphism in $\Cat{Rel}$. In fact, since $(X,m_X)$ and
  $(Y,m_Y)$ are relative Frobenius algebras, so is $(r,m_r)$: one
  readily verifies that it satisfies (M), (A), and (F). Also, (U) is
  satisfied by the pullback $1 \relto R$ of $u_X \times u_Y \colon 1
  \relto X \times Y$ and $\name{r} \colon 1 \relto X \times Y$, that is,
  the intersection of $r$ with $U_X \times
  U_Y$. Theorem~\ref{Fro} thus shows that $r$ induces a groupoid
  $\cat{R}$. It is a subgroupoid of $\cat{G} \times \cat{H}$: we have
  $R_1 \subseteq (G \times H)_1$ by construction, and if $u \in U_R$,
  then $u=u^{-1}$, so $u \in (G \times H)_0$. The structure maps
  of $\cat{R}$ are easily seen to be restrictions of those of $\cat{G}
  \times \cat{H}$.

  Conversely, if $\cat{R}$ is a subgroupoid of $\cat{G} \times
  \cat{H}$, then clearly $R_1 \subseteq X \times Y$ is a morphism in
  $\cat{Rel}$ satisfying (R). It now suffices to observe that these
  constructions are inverses.
\end{proof}

Next, we consider a choice of morphisms that is natural to groupoids,
namely functors. The category $\cat{Rel}$ is in fact a 2-category,
where there is a single 2-cell $r \Rightarrow s$ when $r \subseteq s$. Hence it makes
sense to speak of adjoints of morphisms in $\Cat{Rel}$. A morphism has
a right adjoint if and only if it is (the graph) of a function.

\begin{definition}\emph{
  The category $\Cat{Frob}(\Cat{Rel})$ has relative Frobenius algebras
  as objects. Morphisms $(X,m_X) \to (Y,m_Y)$ are morphisms $R \colon
  X \relto Y$ that preserve both the multiplication and the
  comultiplication:
  \[
    r \after m_X = m_Y \after (r \times r), \quad
    m_Y^\dag \after r = (r \times r) \after m_X^\dag.
  \]
  We denote by $\Cat{Frob}(\Cat{Rel})\func$ the subcategory of all
  morphisms $r$ for which the pullback of $\name{r}$ and $u_X \times
  u_Y$ has a right adjoint.}
\end{definition}

\begin{lemma}\emph{
  Morphisms in $\Cat{Frob}(\Cat{Rel})$ satisfy (R). Hence we have
  subcategories
  \[
    \Cat{Frob}(\Cat{Rel})\func
    \hookrightarrow
    \Cat{Frob}(\Cat{Rel})
    \hookrightarrow
    \Cat{Frob}(\Cat{Rel})\ext.
  \]
  }
\end{lemma}
\begin{proof}
  Let $r \colon X \relto Y$ be a morphism in $\Cat{Frob}(\Cat{Rel})$. Then
  \begin{align*}
    & (m_X \times m_Y) \after (1 \times \swapmap \times 1) \after
    (\name{r} \times \name{r}) \\
    & = (m_X \times m_Y) \after (1 \times r \times r) \after (1 \times
    \swapmap \times 1) \after \{ (*, (x,x,y,y)) \mid x,y \in X \} \\
    & = (1 \times r) \after (m_X \times m_X) \after \{(*,(x,y,x,y))
    \mid x,y \in X \} \\
    & = (1 \times r) \after \{(*,(xy,xy)) \mid x\mbox{ and }y \mbox{ are composable. } \} \\
    & = (1 \times r) \after \{(*,(z,z)) \mid z \in X \} \\
    & = \name{r},
  \end{align*}
  because we can always choose $x=z$ and $y=1$.
\end{proof}

Write $\Cat{Gpd}$ for the category of groupoids and functors.

\begin{theorem}\label{gpdfrobfunc}
  There is an isomorphism of categories $\Cat{Frob}(\Cat{Rel})\func \cong \Cat{Gpd}$.
\end{theorem}
\begin{proof}
  Let $(X,m_X)$ and $(Y,m_Y)$ be relative Frobenius algebras, inducing
  groupoids $\cat{G}$ and $\cat{H}$. Let $r \colon m_X \to m_Y$ be a
  morphism in $\Cat{Frob}(\Cat{Rel})\func$; by Theorem~\ref{gpdfrobext} it
  induces a subgroupoid $\cat{R}$ of $\cat{G} \times \cat{H}$. By
  definition $r$ is (the graph of) a function $G_0 \to H_0$ when
  restricted to $U_X \times U_Y$.  

  We now show that if $(x,y) \in r$, then also $(x^{-1},y^{-1}) \in
  r$. Let $(x,y) \in r$. Then $((y,y^{-1}),1) \in r^\dag \after m_Y =
  m_X \after (r^\dag \times r^\dag)$, so there is $z \in X$ with
  $(z^{-1},y^{-1}) \in r$. But then $((x,z^{-1}),(y,y^{-1})) \in r
  \times r$ and of course $((y,y^{-1}),1) \in m_Y$. So there exists $w
  \in X$ with $xz^{-1}=w$ and $(w,1) \in r$. But the latter means
  $w=1$ because $r$ is a function on identities, giving $x=z$, and
  hence $(x^{-1},y^{-1}) \in r$. 

  Next, we show that $r$ is (the graph of) a function $G_1 \to
  H_1$. First notice that $((x,x^{-1}),1) \in r \after m_X = m_Y
  \after (r \times r)$ for any $x \in X$. Hence for each $x \in X$
  there is $y \in Y$ with $(x,y) \in r)$ (and $(x^{-1},y^{-1}) \in
  r$). Finally, assume that $(x,y) \in r$ and $(x,y') \in r$. Then
  also $(x^{-1},y^{-1}) \in r$. Hence $((y',y^{-1}),1) \in m_X \after
  (r^\dag \times r^\dag) = r^\dag \after m_Y$. So there exists $y''
  \in Y$ such that $(y'', 1) \in r^\dag$ and $y'y^{-1}=y''$. But then
  we must have $y''=1$. Thus $y'y^{-1}=1$, or in other words, $y'=y$.

  Relational composition of graphs coincides with composition
  of functions, so that this assignment is functorial. Conversely, it
  is easy to see that a functor between groupoids induces a morphism
  in $\Cat{Frob}(\Cat{Rel})\func$. Finally, these two constructions
  are inverse to each other. 
\end{proof}

Finally, we can consider a choice of morphisms that is natural from
the point of view of Frobenius algebras, namely the category
$\Cat{Frob}(\Cat{Rel})$. 
There is a category between $\Cat{Gpd}$ and $\Cat{Gpd}\ext$, 
corresponding to the middle subcategory in 
$
    \Cat{Frob}(\Cat{Rel})\func
    \hookrightarrow
    \Cat{Frob}(\Cat{Rel})
    \hookrightarrow
    \Cat{Frob}(\Cat{Rel})\ext
$, 
as follows.

\begin{definition}\emph{
  A \emph{multi-valued functor} $\cat{G} \to \cat{H}$ between
  categories is a multi-valued map $F \colon G_1 \to H_1$ that
  preserves identities and composition: 
  \begin{align*}
    &\mbox{for } g,f \in G_1 \times_{G_0} G_1:
    && g \after f \ni h \;\Rightarrow\; F(g) \after F(f) \ni F(h), \\
    &\mbox{for } x \in G_0: && F(e(x)) \ni H_0.
  \end{align*}
  We denote the category of groupoids and multi-valued functors by
  $\Cat{Gpd}\mfunc$.
  }
\end{definition}

\begin{theorem}\label{Chris1}
  There is an isomorphism of categories $\Cat{Frob}(\Cat{Rel}) \cong \Cat{Gpd}\mfunc$.
\end{theorem}
\begin{proof}
  Let $m_X$ and $m_Y$ be relative Frobenius algebras, inducing
  groupoids $\cat{G}$ and $\cat{H}$. Let $r \colon m_X \to m_Y$ be a
  morphism in $\Cat{Frob}(\Cat{Rel})$; by Theorem~\ref{gpdfrobext} it
  induces a subgroupoid of $\cat{G} \times \cat{H}$. By the
  argument of the proof of Theorem~\ref{gpdfrobfunc}, $r$ is a
  multi-valued map $G_1 \to H_1$. But then it is precisely a
  multi-valued functor.
\end{proof}

\section{Relative $H^*-$ algebras and semigroupoids}
Relative Frobenius algebras can be relaxed to
so-called H*-algebra structures, essentially by dropping
units. First, we define properly this relaxation: it turns out
that H*-algebras in the category of sets and relations correspond to
so-called locally cancellative regular semigroupoids. The
correspondence is functorial, but now it only gives an adjunction
instead of an isomorphism of categories. At the end  we discuss a universal way to pass from
H*-algebras to Frobenius algebras, \textit{i.e.}\ from semigroupoids to
groupoids. 

The generalization to H*-algebras is useful for an application to
geometric quantization that will be presented in subsequent work,
where one is forced to work with semigroupoids instead of
groupoids. Rather than in the category of sets and relations, this
plays out in the smooth setting of symplectic manifolds and canonical
relations, corresponding to Lie groupoids. One could imagine similar
applications in a topological or localic
setting \cite{etale}.

\subsection{From relative $H^*-$ algebras to semigroupoids}
\begin{definition}
 \emph{A relative H*-algebra is a morphism $m \colon X \times X
  \relto X$ in $\Cat{Rel}$ satisfying (M), (A), and the following axiom:
\begin{itemize} 
\item There is an involution $* \colon
      \Cat{Rel}(1,X) \to \Cat{Rel}(1,X)$  such that 
\[\tag{H}  m \after (1 \times x^*) = (1 \times x^\dag) \after m^\dag \text{  and  }
      m \after (x^* \times 1) = (x^\dag \times 1) \after m^\dag\]
       for all $ x \colon 1 \relto X$.
\end{itemize}
  }
\end{definition}

\begin{definition}\emph{
  Given a relative H*-algebra $m \colon X \times X \relto X$, define
  $\cat{G}$ by
  \begin{align*}
    G_0 & = \{ f \in X \mid m(f,f)=f \}, \\
    G_1 & = X, \\
    s & = \{ (f,f^*f) \mid f^* \emph{ is a pseudoinverse of } f \}
    \colon G_1 \relto G_0 \\
    t & = \{ (f,ff^*) \mid f^* \emph{ is a pseudoinverse of } f \}
    \colon G_1 \relto G_0. 
  \end{align*}}
\end{definition}

\begin{lemma}\label{regular}\emph{
   For each element $a$ in a relative H*-algebra there exists $a^* \in \{a\}^*$
   satisfying $a^*aa^*=a^*$ and $aa^*a=a$.}
\end{lemma}
\begin{proof}
  By (M), we have $\forall y \in X \,\exists a,x \in X .\,
  y=ax$. Applying (H) with $A=X$ gives $\forall a \in X
  \,\exists x \in X .\, x\mbox{ and }a \mbox{ are composable} $. 
  Now let $a \in X$. If we substitute $A=\{a\}$, then (H) becomes
  \begin{align*}
    \forall x,y \in X & \big[ xa=y \iff 
    \exists a^* \in \{a\}^* .\, ya^*=x \big] \\
    \forall x,y \in X & \big[ ax=y \iff 
    \exists a^* \in \{a\}^* .\, a^*y=x \big] 
  \end{align*}
  As above, there exists $x \in X$ with $x$ and $a$ composable. So 
  by the first condition above, there is $a' \in \{a\}^*$ with $a$ and $a'$ composable.  Hence, by the second condition, there is $a'' \in
  \{a\}^*$ with $a'' a a'=a'$. Applying the first condition again, now
  with $x=a'$ and $y=a''a$, gives $a'a=a''a$. Therefore we have
  $a^*=a^*aa^*$ for $a^*=a' \in \{a\}^*$. Finally, applying the first
  condition again, this time with $x=aa^*$ and $y=a$, we find that
  also $aa^*a=a$.
\end{proof}
\emph{
\begin{lemma}\label{semigroupoid}\emph{
  The data $\cat{G}$ form a well-defined semigroupoid.}
\end{lemma}
\begin{proof}
  By (A), the condition $m(m\times 1)=m(1 \times m)$ is clearly
  satisfied. It remains to prove that $m$, $s$ and $t$ are
  well-defined functions. 
 The former means that $(g,f) \in G_1 \times_{G_0} G_1$
  implies that $g$ and $f$ are composable. Assume $s(g)=t(f)$, \textit{i.e.}~$g^*g=ff^*$ for some
  pseudoinverses $g^*$ and $f^*$ of $g$ and $f$. Because $g^*g$ is
  idempotent, we have $g^*gff^* = g^*gg^*g = g^*g$, and
  therefore also $g$and $f$ are composable. Hence $m$ is well-defined.
  As for $t$, suppose that $f^*$ and $f'$ are both pseudoinverses of
  $f$, so that $(f,ff^*) \in s$ and $(f,ff') \in s$. Then $ff^*f = f =
  ff'f$. Set $A=\{f^*\}$, $a=f^*$, $x=f$, and $y=ff'$. By
  Lemma~\ref{regular}, we obtain $f \in A^*$, and so $ya^*=x$ for $a^*=f$.
  Now it follows from (H) that $ff^*=xa=y=ff'$. 
  Similarly, $s$ is a well-defined function. 
\end{proof}
}

\begin{theorem}\label{locallycancellative}
  If $m$ is a relative H*-algebra, then $\cat{G}$ is a locally
  cancellative regular semigroupoid.
\end{theorem}

\begin{proof}
  Regularity is precisely Lemma~\ref{regular}. Suppose that
  $fhh^*=gh^*$ for a pseudoinverse $h^*$ of $h$. Applying (H) to
  $A=\{h\}$, $x=fhh^*$, $y=g$, $a=h$ and $a^*=h^*$ yields
  $fh=fhh^*h=xa^*=y=g$. Hence $\cat{G}$ is locally cancellative.
\end{proof}

\subsection{From semigroupoids to relative $H^*$- algebras}
\begin{definition}\emph{
  Given a locally cancellative regular semigroupoid $\cat{G}$, define
  \begin{align*}
    X & = G_1, \\
    m & = \{ (g,f,gf) \mid s(g)=t(f) \} \colon G_1 \times G_1 \relto G_1, \\
    A^* & = \{a^* \in X \mid a^*aa^*=a^* \mbox{ and } aa^*a=a \mbox{ for all } a \in A \}.
  \end{align*}
  }
\end{definition}

\begin{theorem}
  If $\cat{G}$ is a locally cancellative regular semigroupoid, then
  $m$ is a relative H*-algebra.  
\end{theorem}
\begin{proof}
  Clearly, (A) is satisfied. Because
  \[
    m^\dag \after m 
   = \{ (f,f) \in G_1^2 \mid \exists (g,h) \in G_2 \,.\, f=hg \}
  \]
  we have $m^\dag \after m \subseteq 1$. Conversely, if $f \in G_1$,
  setting $g=f$ and $h=f^*f$ for some pseudoinverse $f^*$ of $f$, then
  $f=gh$. Hence (M) is satisfied. 

  Finally, we verify (H). Let $A \subseteq X$ be given, let $a \in
  A$ and $x \in X$, and suppose that $x$ and $a$ are composable. That means that
  $s(f)=t(a)$. By regularity, $a$ has a pseudoinverse $a^* \in A^*$,
  and we have $xa=xaa^*a$. Setting $f=xa$, $g=x$, $h=a$ and $h^*=a^*$
  in the definition of local cancellativity yields $xaa^*=x$. The
  symmetric condition is verified similarly. Hence (H) is satisfied. 
\end{proof}
\subsection{Functoriality}

This subsection proves that the assignments $m \mapsto \cat{G}$ and
$\cat{G} \mapsto m$ extend functorially to an adjunction. We only
consider the relative choice of morphisms, because the lack of units
make the other two choices of morphisms from
Subsection~\ref{subsec:frob:functoriality} very difficult to work 
with. The following two definitions give well-defined categories, just
as in Subsection~\ref{subsec:frob:functoriality}. 

\begin{definition}\emph{
  The category $\Cat{Hstar}(\Cat{Rel})\ext$ has relative H*-algebras
  as objects. A morphism $(X,m_X) \to (Y,m_Y)$ is a morphism $r \colon
  X \relto Y$ in $\Cat{Rel}$ satisfying (R).
  This gives a well-defined category.}
\end{definition}

\begin{definition}\emph{
  The category $\Cat{LRSgpd}\ext$ has locally cancellative regular
  semigroupoids as objects. Morphisms $\cat{G} \to \cat{H}$ are
  locally cancellative regular subsemigroupoids of $\cat{G} \times \cat{H}$.}
\end{definition}

\begin{proposition}\emph{
  The assignments $$m \mapsto \cat{G}$$ and $$\cat{G} \mapsto m$$ extend to functors
  $$\Cat{Hstar}(\Cat{Rel})\ext \to \Cat{LRSgpd}\ext$$ and
  $$\Cat{LRSgpd}\ext \to \Cat{Hstar}(\Cat{Rel})\ext$$  respectively.}
\end{proposition}
\begin{proof}
  Let $(X,m_X)$ and $(Y,m_Y)$ be relative H*-algebras, inducing
  locally cancellative regular semigroupoids $\cat{G}$ and
  $\cat{H}$. Given $r \colon m_X \to m_Y$, define $m_r \colon r \times
  r \relto r$ as in the proof of Theorem~\ref{gpdfrobext}; it
  satisfies (A) and (M). It also satisfies (H), as we now verify. For
  $A \subseteq r$, take
  $
    A^* = \{ (x^*,y^*) \mid (x,y) \in A, x^* \in \{x\}^*, y^* \in
    \{y\}^* \}
  $.
  \begin{align*}
    & (1 \times A) \after m_r^\dag \\
    & \: = \{((x,y),(a,b)) \in r \times r \mid \exists (c,d) \in A \,.\, y=bd, x=ac\} \\
    & \stackrel{\mbox{\tiny{(H)}}}{=} \{ ((x,y),(a,b)) \in r \times r \mid \exists
    (c,d) \in A, c^* \in \{c\}^*, d^* \in \{d\}^* \,.\, a=xc^*, b=yd^*
    \} \\
    & \: = m_r \after (1 \times A^*).
  \end{align*}
  Theorem~\ref{locallycancellative} now shows that $m_r$ induces a
  subsemigroupoid of $\cat{G} \times \cat{H}$.
  Conversely, if $\cat{R}$ is a subsemigroupoid of $\cat{G} \times
  \cat{H}$, then $R_1 \colon G_1 \relto H_1$ clearly satisfies
  (R). Finally, the identity relation $r \colon m_X \relto m_Y$
  corresponds to the diagonal subsemigroupoid, which is indeed regular
  and locally cancellative. 
\end{proof}

\begin{theorem}\label{Chris2}
  The functors from the previous proposition form an adjunction.
  \[\xymatrix@C+3ex{
    \Cat{LRSgpd}\ext \ar@<1ex>[r] \ar@{}|-{\perp}[r]
    & \Cat{Hstar}(\Cat{Rel})\ext \ar@<1ex>[l]
  }\]
\end{theorem}
\begin{proof}
  Starting with a relative
  H*-algebra $m \colon X \times X \relto X$, we end up with 
  \[
  \{(g,f,gf) \mid \exists g^* \in \{g\}^* \exists f^* \in \{f\}^*
  \,.\, g^*g=ff^* \} \colon X \times X
  \relto X.
  \]
  Clearly this is a subrelation of $m$, and the inclusion forms the
  unit of the adjunction.

  Starting with a locally cancellative regular semigroupoid $\cat{G}$,
  we end up with
  \[\xymatrix@C+3ex{
    \{ f \in G_1 \mid f^2=f \}
    & G_1 \ar@<-.6ex>|-{s'}[l] \ar@<.6ex>|-{t'}[l] 
    & G_1 \times_{s',t'} G_1 \ar|-{m}[l]
  }\]
  where $s'(f)=f^*f$ and $t'(f)=ff^*$. Clearly, the original $\cat{G}$
  maps into this, giving the counit of the adjunction.
  Naturality and the triangle equations are easy.
\end{proof}

Notice that we merely get an adjunction, and not an isomorphism as in
Theorem~\ref{gpdfrobext}. Indeed, $gf \defined$ need not imply
$g^*g=ff^*$ for some pseudoinverses $f^* \in \{f\}^*$ and $g^* \in
\{g\}^*$, and $G_0$ need not coincide with the idempotents of $G_1$ at
all. 

\subsection{Semigroupoids Vs groupoids: reduction}
\label{sec:quotient}

The forgetful functor $\Cat{Groupoid} \to \Cat{Cat}$ has a left
adjoint, that freely adds inverses. Similarly, the forgetful functor
$\Cat{Cat} \to \Cat{Semigroupoid}$ has a left adjoint, that freely
adds identities. The image of the latter left adjoint consists precisely
of those categories in which the only isomorphisms are identities.
Hence there is a functor $\Cat{Semigroupoid} \to \Cat{Groupoid}$
giving the free groupoid on a semigroupoid. Restricting it gives a
functor that turns a locally cancellative regular semigroupoid into a
groupoid.

\section{Extensions and perspectives}\label{Perspectives}

This section is devoted to discuss possible applications and perspectives of the construction of relational symplectic groupoids. Some of this work is still in progress and it gives rise to many questions and possible developments.
\subsection{Quantization}
A natural questions appearing in the study of symplectic groupoids is their quantization. In our setting, we have studied relational symplectic groupoids on the categories of vector spaces, the hope is to obtain the relational symplectic groupoid $(\mathcal G, L, I)$ as a dequantized version of the relational symplectic groupoid in the category \textbf{Hilb} of Hilbert spaces.

\subsubsection{Linear relational symplectic groupoids Versus Quantization}
in Chapter 4, we already studied some cases of relational symplectic groupoids where $\mathcal G$ is a vector space and also we considered more general structures in fagger categories. In this subsection we  discuss the connection between the linear version of relational symplectic groupoids and quantization.


As we showed in Example {\ref{functions}}, the space of smooth functions on a manifold $M$ can be equipped with a weak *-monoid structure and with a suitable inner product, it is a Frobenius algebra.


The dequantized version of such Frobenius algebra would correspond to the symplectic manifold
$T^*M$ with some induced Lagrangian subspaces that in general fail to be submanifolds.  It turns out that one get immersed submanifolds when the diffeomorphism $\phi$ has no fixed points. More precisely, the subspace $L_2$  is the union of two Lagrangian submanifolds $L_2^{Id}$ and $L_2^{\phi}$, where $L_2^{Id}$ is the graph of the identity and $L_2^{\phi}$ would correspond to the quadruples of the form $(p,\phi(x),\phi^*p,x)$, for all $x \in M, p \in \mathcal C ^{\infty}(M)$. One example to illustrate this situation is the following. Take $M=\mathbb R$ and 
\begin{eqnarray*}
\phi: \mathbb R &\to& \mathbb R\\
x&\mapsto& -x.
\end{eqnarray*}
In this case the spaces $L_i$ are embedded submanifolds and correspond to
\begin{eqnarray*}
L_1&=&\{\mbox{even positive functions in } \mathbb R \}. \\
L_2&=&\{(f(x), h(x)\frac{(f(x)+f(-x))}{2})\mid h(x) \in L_1\}\\
&=&\{\mbox{projection to even functions in } \mathbb R \}.\\
L_3&=&\{f(x),g(x), f*g(x)\}.
\end{eqnarray*}

However, this relational symplectic groupoid is not regular, since the space of objects is $\mathbb R \mod \phi= [0, \infty)$.






\subsubsection{Preunital Poisson manifolds and Frobenius algebras}
The structure of relational symplectic groupoid might be reformulated in the category \textbf{Hilb} of Hilbert spaces and in principle it 
should yield a \emph{relational version} of a Frobenius algebra, that we may call preunital Frobenius algebra. 
In that case, the relational symplectic groupoid may be seen as the dequantization of this structure. 
The hard problem consists in going the other way around, namely, in quantizing a relational symplectic groupoid.

In the finite dimensional examples, methods of geometric quantization might be available, 
the problem being that of finding an appropriate polarization compatible with the structures. This question, in the case of a symplectic groupoid, has been addressed by Weinstein \cite{Alan} and Hawkins \cite{Hawkins}. The relational structure might allow more flexibility.

In the infinite-dimensional case, notably the example in 3.1.4., perturbative functional  integral techniques might be available, following the procedure for perturbative quantization on manifolds with boundary \cite{AlbertoPavel2, AlbertoPavel3}. The reduced algebra should give back a deformation quantization of the underlying Poisson manifold.

Finally notice that quantization might require loosening up a bit the notion of preunital Frobenius algebra, allowing for example non associative products, but maybe still associative up to homotopy. However, one expects that the reduced algebra should always be associative.





\subsection{Possible extensions}

\subsubsection{Relational Lie groupoids}
In this example we drop the weak symplectic structure on the construction of relational symplectic groupoids and we consider  relational Lie groupoids as cyclic weak *-monoid in the category $\mbox{\textbf{Man}}^{Ext}$ of smooth manifolds and immersed submanifolds as morphisms. Then $G\rightrightarrows M$ is a relational Lie groupoid regarded as the triple $(G, Graph(\mu), \iota)$. 
In this setting, given a Lie groupoid $G \rightrightarrows M$ it is possible to equip $G \times G$ with a relational Lie groupoid structure as follows.
\newline
Considering the graph of the multiplication $\mu$ a morphism $Graph (\mu)\colon  G\times G \nrightarrow G$ in $\mbox{\textbf{Man}}^{Ext}$, we induce a weak *-monoid structure on $G\times G$ in such way that $Graph (\mu)$ is an equivalence of relational Lie groupoids between $G \times G$ and $G$. This is a priori a relational Lie groupoid structure on $G\times G$ different from the one given by the power of Lie groupoids (see Example \ref{power}).

\subsubsection {Relational symplectic groupoids and the (extended) Poisson category}
We conjecture that there is an equivalence of categroids between 
\textbf{RRelGpd}, the categroid of regular relational symplectic groupoids, with morphisms given by Definition \ref{morphism}  
and $\mbox{\textbf{Poiss}}^{Ext}$, the extended category of Poisson manifolds, with morphism corresponding to (possibly certain class of) coisotropic submanifolds of the product of two Poisson manifolds 
\footnote{Here, as before,  we abuse the language and we call the categroid $\mbox{\textbf{Poiss}}^{Ext}$ a category. The functors that determine the equivalence are defined with respect to the partial
 composition of morphisms.  }. The functor 
$F\colon \mbox{\textbf{Poiss}} \to \mbox{\textbf{RSG}}$ is the one given by the cotangent of the path space, discussed previously 
and the functor $G\colon \mbox{\textbf{RSG}} \to \mbox{\textbf{Poiss}}$ is given by the projection to the base space. 
In this direction, some completeness for morphisms might be needed (see e.g. \cite{Benoit4}).

\subsubsection{Extension for general algebroids}
Using the fact that for any Lie algebroid $A$, the manifold $A^*$ is Poisson, one would like to extend the integration of Poisson manifolds via relational symplectic groupoids to general Lie algebroids, following, for example \cite{Algebroid}.

\subsubsection{Dual pairs and PSM with branes}
The construction of relational symplectic groupoids for Poisson manifolds seems to be helpful to understand the case of PSM with branes and to give a precise
interpretation of the spaces of partially reduced boundary fields.
This is connected with what is known as cohomological resolution of
the reduced phase space for topological field theories and the results proven in \cite{Branes}.


\subsection{The presence of handles}
We define an additional structure that appears in example of relational symplectic groupoids coming from Poisson manifolds.
\begin{definition}
\emph{Let $\mathcal G$ be a relational symplectic groupoid. A relation $H\colon \mathcal G \nrightarrow \mathcal G$ (not necessarily a canonical relation) is called a \emph{handle} of $\mathcal G$ if it satisfies the following conditions:
\begin{enumerate}
\item $L_2 \circ H= H\circ L_2= H$
\item $H \circ L_1,\, (H \times Id)\circ L_3$ and $(Id \times H)\circ L_3$ are immersed submanifolds of $\mathcal G$ and $\mathcal G^3$ respectively.
\item $H^n:= H\circ H \cdots H$ is an immersed submanifold of $\mathcal G \times \mathcal G, \, \forall n \in \mathbb N$.
\end{enumerate}
}
\end{definition}

\subsubsection{Example: Linear Poisson structure in PSM with genus}
It is possible to extend the construction of the phase space before reduction in the case where the source manifold in PSM has genus. More precisely, adding a handle to the disc (that is homeomorphic to a punctured torus) gives rise to a relational symplectic groupoid with a handle, that is, an extension of the relational symplectic groupoid, 
where $\Sigma$ is a disc with an attached handle. This extension includes singularities for the defining spaces $L_i$ and their reductions $\underline{L_i}$.
More precisely, the construction of $L_1$ in the Example $\ref{linear}$ must be adapted to the presence of non trivial holonomies around the boundary circle. 
It can be checked that, for example, in the case of $\mathfrak {su}(2)$, the reduced space $\underline {L_1}$ would correspond to the union of the zero section of the cotangent bundle $T^*(G/G)$, where $G=SU(2)$ and a cotangent fiber for $G/G$ (modulo the action of the Weyl group) at the point $1$.


In the case of $\mathfrak g=\mathfrak{SU}(3)$, it is conjectured that the reduction of $L_1$ would correspond to the union of zero-section of $T^*(G/G)$, the conormal bundle $N^*(K/G )$, where $K$ is a codimension 1 submanifold of $G=SU(3)$ and a full cotangent fiber for $G/ G$ at 1.

\subsubsection{Relational groupoids, symplectic microgeometry and double groupoids}\label{cotangent}
One possible extension of this work is to study the construction of the relational symplectic groupoid in a different version of the symplectic category, where the space of objects and morphisms are replaced by certain equivalence classes which encode local information, this point of view follows the work of Cattaneo, Dherin and Weinstein on what is called \emph{symplectic microgeometry} (see \cite{Benoit, Benoit2, Benoit3}).\\
Also related, relational symplectic groupoids could be compared with the recent notion of \emph{symplectic hopfoid} developped by Ca\~nez in his doctoral thesis \cite{Canes}, as an attempt to undertsand the concept of symplectic stacks and the cotangent functor. The  version of the relational symplectic groupoid in higher categorical terms, i.e. considering a more general notion for immersed canonical relations that allows defining $2-$ morphisms,, could be compared with already existing objects such as double symplectic groupoids \cite{Canes, Mehta}.\\

\newpage

\appendix
\chapter{Dirac structures}\label{Dirac}
This Appendix includes some notions and definitions for Dirac structures, in a way to deal at the same time with pre-symplectic and Poisson structures. The content is quite standard; it is mainly based on \cite{Bursztyn}, where a more detailed overview is done, and includes the extension of the notions of Poisson maps for Dirac structures, important for the discussion in Chapter \ref{PSMmain}.
\section{Definitions and examples}
\begin{definition}\emph{
The generalized tangent bundle is defined as
\[ \mathbb TM:= TM\oplus T^*M \]
and it is equipped with natural projections
\[pr_{T}\colon \mathbb TM \to TM, \, pr_{T^*}\colon \mathbb TM \to T^*M.\]
}
\end{definition}

In addition, the generalized tangent bundle is equipped with a non degenerate, symmetric fiber wise linear form $\langle , \rangle$ defined by
\[\langle (X, \alpha) ,(Y, \beta) \rangle:= \beta (X)+ \alpha (Y),\]
where $X,Y \in T_xM, \alpha, \beta \in T^*_xM$.
It is also endowed with the \emph{Courant bracket} $[[\cdot , \cdot ]]\colon \Gamma (\mathbb T M)\times \Gamma (\mathbb T M) \to \Gamma (\mathbb T M)$, defined by 
\[ [[(X,\alpha),(Y, \beta)]]= ([X,Y], \mathcal L_X\beta-\mathcal L_Y\alpha+ \frac 1 2 d(\alpha(Y)-\beta(X))).\]
\begin{definition}\emph{
A \emph{Dirac structure} on $M$ is a vector subbundle $L \subset \mathbb TM$ such that
\begin{enumerate}
\item $L$ is maximally isotropic, i.e. $L=L^{\perp}$, with respect to $\langle \cdot, \cdot \rangle$.
\item $L$ is involutive w.r.t. the Courant bracket, i.e. $[[ \Gamma(L), \Gamma(L)]]\subset \Gamma (L).$
\end{enumerate}
}
\end{definition}
Some natural examples for Courant structures are Poisson and plesymplectic structures where the subbundles $L$ are controlled by the bivector $\Pi$ and the 2-form $\omega$ respectively.

\begin{example} \emph{(\textbf{Poisson structures}). A bivector field $\Pi \in \Lambda^2(TM)$ induces a subbundle of $\mathbb TM$ given by the graph of the map
\[ \Pi ^{\sharp}\colon  T^*M \to TM,\]
namely,
$$L_{\Pi}:=\{(\Pi^{\sharp}(\alpha), \alpha) \mid \alpha \in T^*M\}.$$
It can be checked that $L_{\Pi}$ is involutive if and only if the bivector field $\Pi$ is Poisson. Therefore, a subbundle $L$ determines a Poisson structure if and only if $L_{\Pi}\cap TM = \{ 0\}$.
}\end{example}
\begin{example}
\emph{(\textbf{Pre-symplectic structures}). A 2-form $\omega \in \Omega^2(M)$ induces a subbundle of $\mathbb TM$ determined by the graph of the bundle map
\[\omega^{\sharp}\colon  TM \to T^*M,\] i.e.
\[L_{\omega}:= \{(X, \omega^{\sharp}(X)) \mid X \in TM.\} \]
It can be checked that in this case $L$ is involutive if and only if $d\omega=0$, i.e. $\omega$ is pre-symplectic. Analogously, a involutive subbundle $L$ determines a presymplectic structure if and only if $L \cap T^*M= \{ 0\}$.
}
\end{example} 
\begin{example} \emph{\textbf{(Regular foliations}). Consider a regular distribution $D \subset TM$, that is equivalent to consider a subbundle of $TM$.  Define
$$L_D= D \oplus D^{\circ},$$ where $D^{\circ} \subset T^*M$ denotes the annihilator of $D$.
Then, by definition, $L_D$ is maximally isotropic and the fact that $L_D$ is involutive in the Dirac sense is equivalent to the fact that the distribution $D$ is involutive with respect to the Lie bracket on vector fields and by Frobenius Theorem, it is equivalent to the integrability of $D$.
}
\end{example}
\section{Morphisms of Dirac manifolds}
This section describe the notion of morphisms compatible with Dirac structures, as a way to generalize the notion of Poisson map (which is covariant) and symplectomorphism (which is contravariant). This leads to the notion of \textit{forward} and \textit{bacward} Dirac maps.
First of all, we study the linear case. Observe that a Dirac structure on a vector space $V$ corresponds to a Lagrangian subspace $L \subset V \oplus V^*$ with respect to the natural symmetric pairing
$$\langle (v_1, \alpha_1),(v_2, \alpha_2) \rangle= \alpha_2(v_1)+ \alpha_1(v_2). $$
Now, consider a linear map $\phi: V \to W.$ Given a 2- form $\omega \in \wedge ^2 W^*$, this can be pulled back to $V$ via $\phi$:
$$ \phi^*\omega(v_1,v_2)= \omega(\phi(v_1), \phi(v_2)),$$
where $v_i \in V$.
In a similar fashion, a bivector $\Pi \in \wedge^2 V$ can be pushed forward to $W$ via $\phi$:
$$\phi_* \Pi(\beta_1, \beta_2)= \Pi(\phi^*\beta_1, \phi^* \beta_2),$$
where $\beta_i \in W^*$.
The analogue for this behavior in the Dirac world is the following
\begin{definition} \emph{
Let $L_W$ be a Dirac structure on $W\oplus W^*$. The \emph{backward image} of $L_W$ under $\phi$ is defined by:
$$\mathfrak B_{\phi}(L_W):=\{(v, \phi^* \beta) \mid (\phi(v), \beta) \in L_W\} \subset V \oplus V^*.$$
In a similar way, let $L_V$ be a Dirac structure on $V \oplus V^*$. The \emph{forward image} of $L_V$ under $\phi$ is given by:
$$\mathfrak F_{\phi}(L_V):=\{(\phi(v), \beta)\mid (v, \phi^* \beta) \in L_V \} \subset W \oplus W^*.$$
}
\end{definition}
It can be proven that
\begin{proposition}\emph{\cite{Bursztyn}. $\mathfrak B_{\phi}(L_W)$ and $\mathfrak F_{\phi}(L_V)$ are Dirac structures.
}
\end{proposition}
The global version of this Definition corresponds to the notion of backward and forward Dirac map.
\begin{definition}\emph{
Let $(M, L_M)$ and $(N,L_N)$ be Dirac manifolds and $\phi\colon \to N$ be a smooth map. $\phi$ is called a \emph{backward Dirac map} if $L_M$ coincides with $\mathfrak B_{\phi}(L_N)$, i.e.
$$(L_M)_x= \mathfrak B_{\phi}(L_N)_x= \{ X, d\phi^*(\beta) \mid (d\phi(X), \beta)\in (L_N)_{\phi(x), }\}, \forall x \in M.$$
}
\end{definition}

\begin{definition}\emph{
The map $\phi$ is called a \emph{forward Dirac} map if $\phi^*L_N$ (the pullback of $L_N$ to $N$) coincides with $\mathfrak F_{\phi}(L_M)$, that means
$$(L_N)_{\phi(x)}=\mathfrak F_{\phi}(L_M)_x= \{(d\phi(X), \beta)\mid (X, d\phi^*(\beta))\in (L_M)_x\}, \forall x \in M. $$
}
\end{definition}

\end{document}